\documentclass[12pt,letterpaper]{amsart}
\pdfoutput=1 
\usepackage{etex}
\usepackage{amsmath, amsthm, graphicx, amssymb,verbatim,pst-all,color}
\usepackage[margin=1.0in]{geometry}
\usepackage{mathptmx}
\usepackage[latin1]{inputenc}
\usepackage{multirow}
\usepackage{subfig}
\usepackage{enumerate}
\usepackage{array,nicefrac}
\usepackage{url}
\usepackage{times}
\usepackage{textcomp}
\usepackage{verbatim}
\usepackage{epsf}
\usepackage{stmaryrd}
\usepackage{relsize}
\usepackage{mathtools}
\usepackage{amscd, amsfonts}
\usepackage{enumitem}
\usepackage[all]{xy}
\usepackage{lscape}
\usepackage{etex,soul,picins}
\usepackage{pgf}

\usepackage{enumitem}
\usepackage{qtree}
\usepackage{tikz}
\usetikzlibrary{arrows}
\usetikzlibrary{patterns} 
\usetikzlibrary{plotmarks}
\usetikzlibrary{matrix,arrows}
\usetikzlibrary{calc,positioning}
\numberwithin{equation}{section}

\usepackage{hyperref}

\ifx\JPicScale\undefined\def\JPicScale{1.0}\fi
\unitlength \JPicScale mm
{
\theoremstyle{plain}
   \newtheorem{theorem}{Theorem}[section]
   \newtheorem{proposition}[theorem]{Proposition}     
   \newtheorem{lemma}[theorem]{Lemma}

   \newtheorem{corollary}[theorem]{Corollary}
   
}
{
\theoremstyle{definition}

   \newtheorem{example}[theorem]{Example}
   \newtheorem{definition}[theorem]{Definition}
   \newtheorem{remark}[theorem]{Remark}

}

\newcommand{\PP}{\mathbb{P}}

\newcommand{\hatbP}{\check{\bP}}

\newcommand{\ZZ}{\mathbb{Z}}

\newcommand{\Aff}{\mathbb{A}}
\newcommand{\iso}{\cong}

\newcommand{\wbuildtdn}{\mathcal{H}_{\mathcal{A}}}
\newcommand{\wbuildtdnb}{\mathcal{H}_{\mathcal{B}}}

\newcommand{\wmod}{X_\mathcal{A}[n]}

\newcommand{\wtdn}{T_{d,n}^{\mathcal{A}}}
\newcommand{\wtdnu}{(T_{d,n}^{\mathcal{A}})^+}

\newcommand{\wtdno}{(T_{d,n}^{\mathcal{A}})^o}
\newcommand{\wtdnb}{T_{d,n}^{\mathcal{B}}}
\newcommand{\wtdr}{T_{d,r}^{\mathcal{A}(R)}}
\newcommand{\bP}{\mathbb{P}}

\newcommand{\wpdn}{\overline P^{\mathcal A}_{d,n}}
\newcommand{\wpdr}{\overline P^{{\mathcal A}(R)}_{d,r}}
\newcommand{\wpdno}{P^{\mathcal A}_{d,n} }
\newcommand{\wpdnu}{\overline U_{d,n}^{\mathcal{A}}}

\numberwithin{theorem}{section}

\date{}

\address{642 Boyd Graduate Studies Research Center\\ University of Georgia\\ Athens, GA 30602\\ USA}
 \email{gallardo@uga.edu}
\address{Kavli Institute for the Physics and Mathematics of the Universe (WPI), The University of Tokyo Institutes for Advanced Study, The University of Tokyo, Kashiwa, Chiba 277-8583, Japan}
\email{evangelos.routis@ipmu.jp}

\begin{document}
\title{Wonderful compactifications of the moduli space of points in affine and projective space}
\author{Patricio Gallardo and Evangelos Routis}

\begin{abstract}
We introduce and study smooth compactifications of the moduli space of $n$ labeled points with weights in projective space, which have normal crossings boundary and are defined  as GIT quotients of the weighted Fulton MacPherson compactification. We show that the GIT quotient of a wonderful compactification is also a wonderful compactification under certain hypotheses.  We also study a
weighted version of the configuration spaces  parametrizing $n$ points  in affine space up to translation and homothety.  
In dimension one, the above compactifications are isomorphic to Hassett's moduli space of rational weighted stable curves.  
\end{abstract}
\maketitle
\section{Introduction}

For any smooth variety $X$, Fulton and MacPherson  constructed a smooth compactification $X[n]$ of the configuration space of $n$ distinct labeled points in $X$, such that all points remain distinct in the degenerate configurations  \cite{Fulton-MacPherson}. A few years later, Hu and Keel  showed that $\overline{M}_{0,n}$ is a GIT quotient of $\bP^1[n]$ by $SL_2$  \cite{Hu-Keel}.
Recently, the second author extended the Fulton-MacPherson construction by including weight data which allow points to collide depending on the accumulation of their weights \cite{routis2014weighted}. 
The objectives of this article are the following. First, to generalize 
the results of  Hu and Keel by constructing smooth weighted compactifications   $\overline P^{\mathcal A}_{d,n}$ of the moduli space of $n$ points in $\bP^d$ which are birational to the moduli space of weighted hyperplane arrangements  \cite{Hacking-Keel-Tevelev-hyperplane}, \cite{alexeev2013moduli} (see Section \ref{hyparr}).  Second, to describe a novel iterated blow-up construction of the compactification  $T_{d,n}$  of the configuration space of $n$ labeled points  in $\Aff^d$ up to translation and homothety  \cite{Chen-Gibney-Krashen}, as well as to study a weighted version of it (see Section \ref{sec:hM0m}).  Finally,  to develop a theoretical framework that allows us  to study GIT quotients of wonderful compactifications associated with more general moduli problems  (see Section \ref{sec:WondGIT}).

Let us give a brief description of the geometric points of $\wpdn$. We start with an equivalence class of $n$ labeled points in $\bP^d$ parametrized by a GIT quotient which is defined in Lemma \ref{GITpoints}. 
Let $\mathcal{A}:=\{a_1,a_2,\dots a_n\}$ be an ordered set of numbers between $0$ and $1$, which we call weights, associated to the  labeled points. We impose the requirement that the set $\mathcal{A}$  lies in the domain of admissible weights $\mathcal D^P_{d,n}$ (see Section \ref{weights}). If a subset of those points with weight sum larger than one collides, then we blow up the point of collision and attach a new $\bP^d$, which we glue along the exceptional divisor. 
This subset of points then `moves' to the new $\bP^d$ and is not coincident anymore. We continue this procedure until all colliding points with total weight larger than one are separated. The resulting degenerations are called  \emph{ weighted stable trees} with respect to the set of weights $\mathcal A$.
Let $P_{d,n}^{\mathcal{A}}$ be the open locus in $\wpdn$ parametrizing equivalence classes of $n$ labeled points $\{p_1,\dots p_n\}$ in $\bP^d$ such that any subset of colliding points $\{p_i \;| \; i\in I\}$ has total weight $\sum_{i\in I} a_i$ less than or equal to 1 (see Definition \ref{defpdno}).  The following theorem is proven in Section \ref{conspdn}. 
\begin{theorem}
\label{mainpdn}Let $\mathcal{A}$ be a set of admissible weights in $\mathcal D^P_{d,n}$.
\begin{enumerate}
\item 
The compactification   
$\overline P_{d,n}^{\mathcal{A}}$ of $P_{d,n}^{\mathcal{A}}$ is a smooth projective variety, whose boundary is a union of smooth irreducible divisors that intersect with normal crossings.
 \item 
There exists a smooth variety $\wpdnu$ and a flat, proper morphism
 \begin{align*}
\hat \phi_{\mathcal{A}}:
\wpdnu\rightarrow
\wpdn
\end{align*}
equipped with $n$ sections $\hat\sigma_i:\wpdn\rightarrow \wpdnu$ such that 
\begin{enumerate}
\item the images of $\hat \sigma_i$ lie in the relative smooth locus of $\hat \phi_{\mathcal{A}}$ and
\item the geometric fibers of $\hat \phi_{\mathcal{A}}$ are precisely the weighted stable trees with respect to the set of weights $\mathcal A$.
\end{enumerate}
\end{enumerate}
\end{theorem}

\subsection{Comparing $\wpdn$ with the moduli space of hyperplane arrangements. }\label{hyparr}
The moduli space of $n$ generic labeled points in $\PP^d$ is equal to the moduli space of $n$ generic labeled hyperplanes in the dual projective space $\hatbP^d$. 
Let $\vec w\in \mathcal {D}(d+1,n):=\{  (w_1, \ldots, w_n) \in \mathbb Q^n \; |\; w_1+\ldots+w_n \geq 3,\; 1 \geq w_i > 0\}$ be a set of numbers, which are called weights, associated to the hyperplanes.
A compact moduli space $\overline{M}( \hatbP^d,n)$ that contains the moduli space of $n$ generic hyperplanes in $\hatbP^d$ with weights $(1, \ldots, 1)$ was constructed by Kapranov \cite{Kapranov-chow} and Hacking-Keel-Tevelev \cite{Hacking-Keel-Tevelev-hyperplane}.
A compact moduli space $\overline{M}_{\vec w}( \hatbP^d,n)$ 
 that contains the moduli space of $n$ hyperplanes with weights $\vec w$
was constructed by Alexeev \cite{alexeev2013moduli}. The space $\overline{M}_{\vec w}(\hatbP^d,n)$  can be arbitrarily singular, and  can contain many irreducible components when $d \geq 2$.   
The objects parametrized by these compact moduli spaces are called stable hyperplane arrangements and are abbreviated by shas in the literature (see \cite[Def 5.3.1, Thm 5.3.2]{alexeev2013moduli}).
Moreover,
there is  a unique irreducible component 
$\overline{M}^{m}_{\vec w}(\hatbP^d,n)$ of $\overline{M}_{\vec w}( \hatbP^d,n)$, often referred to as the `main component',  that contains the moduli space of $n$ generic hyperplanes in $\hatbP^d$ as a dense open subset.  
In general, there exists a projective morphism from the main component $\overline{M}^{m}(\hatbP^d,n)$  to any GIT quotient  of $n$ hyperplanes in $\hatbP^d$ 
of the form  $ \left( \bP^d \right)^n / \! \! /_{\mathcal L} SL_{d+1}$  (see \cite[Sec 5.5]{alexeev2013moduli}). 

For convenience we restrict to $d=2$.  By   Lemma \ref{GITpoints} 
and the well-known fact that every projective birational morphism is a blow up with respect to an ideal sheaf,  there is a  blow up
$$
\Phi: \overline{M}^{m}(\hatbP^2,n) \to  \left( \bP^{n-4} \right)^{2}
$$
which by definition is an isomorphism on the open set $U_{gen}$ parametrizing arrangements of $n$ generic lines in $\hatbP^2$.  It holds that 
$\bigcup_{I} C_I = (\PP^{n-4})^2 \setminus U_{gen}$, where  $C_I \subset  (\PP^{n-4})^{2}$ is the locus parametrizing configurations where the lines $\{ l_i \; | \; i \in I \}$ are concurrent at a point and $I\subset \{1,\dots n\}$ such that $3\leq |I|\leq n-2$. The locus $C_I$ can be identified with the one parametrizing collinear points 
$\{ p_i \; | \; i \in I \}$ in $\PP^2$. Since coincident lines are trivially concurrent, it holds that 
$
H_I   \subset C_I   
$
where $H_I$ is the locus  parametrizing configurations with coincident lines  $l_{i}=l_{j}$ for all $i, j $ in $I$. The locus $H_I$ can be identified with the 
one parametrizing configurations of $n$ labeled points $p_1, \ldots, p_n$ in 
$ \PP^2$ such that $p_i=p_j$ for all $i, j$ in $I$. Now let $\overline{P}_{2,n}$ be the variety $\overline{P}^{\mathcal{A}}_{2,n}$ where $\mathcal{A}=\{1,1,\dots,1\}$. We have the following Proposition, whose proof can be found in Section \ref{sec:hyper}. 

\begin{proposition}\label{thm:hyper}
There is a sequence of blow ups  
$\phi_k$ (resp. $\rho_k$) and weights $\vec w_k \in  \mathcal {D}(3,n)$ (resp. $\vec \beta_k\in  \mathcal {D}(3,n)$), where $3\leq k\leq n-4$ (resp. $2\leq k\leq n-3$) such that 
\begin{align*}
    \xymatrix{ 
\overline{M}^{m}(\hatbP^2,n) \ar[r]^{\phi_{3}}
& \overline{M}^{m}_{\vec w_{3}}(\hatbP^2,n) \ar[r] ^{\qquad{}  \rho_2}
&
\ldots  \ar[r]^{\rho_{n-2} \qquad{} }
& \overline{M}^{m}_{\vec \beta_{n-2}}(\hatbP^2,n) \ar[r]^{\phi_{n-4}} 
& \overline{M}^{m}_{\vec w_{n-4}}(\hatbP^2,n) \ar[r]^{\rho_{n-3}}
& 
(\PP^{n-4})^2 .
}
\end{align*}
Moreover, each center $B_k$ (resp. $F_k$) of the blowup $\rho_k$ (resp. $\phi_k$) can be written as
\begin{align*}
B_k = \bigcup_I  B_I && (\text {resp.}\,\, F_k = \bigcup_I  F_I)
 && \text{ with}  &&
 I \subset \{ 1, \ldots n\},  |I|=k,  &&&
\end{align*}
where $B_I $ (resp. $F_I$) are subschemes of $\overline{M}^{m}_{\vec \beta_{k}}(\hatbP^2,n)$ (resp.  $\overline{M}^{m}_{\vec w_{k}}(\hatbP^2,n)$) that parametrize those shas with the property that at least one of their irreducible components has $k$ coincident lines $l_{i}=l_{j}$ for all $i,j\in I$ (resp. $k$ concurrent lines $\{l_i|i\in I\}$).  If $k < (n-3)$, then it holds that
\begin{itemize}
\item 
 $B_I$ is strictly larger than the strict transform of $H_I$ in 
$ \overline{M}^{m}_{\vec \beta_{k}}(\hatbP^2,n) $.
\item 
Each $B_I$ is a reducible, non-equidimensional scheme. 
\end{itemize}
On the other hand, there is a sequence of blow ups $\hat\rho_k$ and weights $\vec \alpha_{k}\in \mathcal D^P_{2,n}$, $2\leq k\leq n-3$,
\begin{displaymath}
    \xymatrix{ 
\overline P_{2,n} 
\ar[r] ^{\hat\rho_2}
& \overline P_{2,n}^{ \vec \alpha_{2}}  \ar[r]^{\hat \rho_3} 
&
\ldots  \ar[r]^{\hat \rho_{n-5} \;\;\; }
& \overline P_{2,n}^{\vec \alpha_{n-5}} \ar[r]^{\hat \rho_{n-4} \; \;}
& 
\overline P_{2,n} ^{\vec \alpha_{n-4}} \ar[r]^{\hat \rho_{n-3} \; \;}
& 
(\PP^{n-4})^2. 
}
\end{displaymath}
Each center $S_k$ of the blowup $\hat \rho_k$ can be written as
\begin{align*}
S_k = \bigcup_{I} S_I
 && \text{ with}  &&
 I \subset \{ 1, \ldots n\},  |I|=k,  &&&
\end{align*}
where $S_I$ are smooth subvarieties of $\overline P_{2,n} ^{\vec \alpha_{k}}$ whose geometric points parametrize those weighted stable trees with the property that at least one of their irreducible components has $k$ overlapping points  $p_{i}=p_{j}$ for all  $i, j \in I$. 
If $k  < n-3$, then:
\begin{itemize}
\item Each $S_I$ is equal to the strict transform of  $H_I$ in $\overline P_{2,n}^{ \vec \alpha_{k}} $.
\end{itemize}
\end{proposition}

\subsection{ Higher dimensional analogs of $\overline{M}_{0,n}$. }\label{sec:hM0m}
The Deline-Mumford-Knudsen moduli space of pointed stable curves of genus 0 is a very important chapter of algebraic geometry, which has been studied intensively in the past 40 years. It is natural to ask if one can construct higher dimensional generalizations of it, which share some of its remarkable geometric properties, for example smoothness, normal crossings boundary and explicit blowup construction. In this direction, Chen-Gibney-Krashen \cite{Chen-Gibney-Krashen} introduced and studied compactifications $T_{d,n}$ of the parameter space of $n$ labeled points in $\Aff^d$ with the aforementioned properties, whose closed points parametrize a  generalization of stable pointed rational curves known as stable pointed rooted trees. Further, they can be understood as non-reductive Chow quotients of $(\PP^d)^n$ \cite{gallardo2015chen}. However, in contrast with $\overline{M}_{0,n}$, little progress has been made towards a satisfactory understanding of the geometry of these spaces. One of the difficulties of extending our understanding of $\overline{M}_{0,n}$ to $T_{d,n}$ is that the latter has not been given a description as a sequence of smooth blowups analogous to Kapranov's elegant construction of $\overline{M}_{0,n}$ \cite{Kapranov-chow}. In \cite{Chen-Gibney-Krashen}, the authors provide an inductive construction of $T_{d,n}$ as a sequence of blowups of a projective bundle over $T_{d,n-1}$, which is quite involved. Our theory provides a natural description of $T_{d,n}$ as a sequence of blowups of $\mathbb{P}^{dn-d-1}$ which is similar to the one given by Kapranov when $d=1$ (see also \cite[Section 6.2]{Hassett-weighted}). Each of the intermediate blowups in this sequence has an interpretation as a space of so-called \textit{weighted stable rooted trees} (Section \ref{rootedtrees}, Section 4), recently introduced by the second author \cite{ChowHass} in analogy to Hassett's moduli space of rational \textit{weighted stable curves} \cite{Hassett-weighted}:  

\begin{corollary}\label{descrip}
For $d \geq 1$ and $n\geq2$, we fix $n$ planes $P_1, P_2,\dots ,P_n$ of dimension $d-1$ in $\mathbb{P}^{dn-d-1}$ with the property that for any $S\subset \{1,2\dots ,n\}$ with $|S|\leq n-1$, the set of planes $\{P_i|i\in S\}$ spans a linear subspace in $\mathbb{P}^{dn-d-1}$ of the maximal possible dimension, that is $d|S|-1$ . There exists a sequence of morphisms of smooth varieties
\begin{displaymath}
\xymatrix{
T_{d,n}
\ar[r]
&
T_{d,n}^{\mathcal A_{n-2}} \ar[r]
&
\ldots \ar[r]
& 
T_{d,n}^{\mathcal A_3} \ar[r]
&
T_{d,n}^{\mathcal A_2} \ar[r]
&
\mathbb{P}^{dn-d-1}
}
\end{displaymath}
where
\begin{itemize}
\item $T_{d,n}^{\mathcal A_2}$ is the blow up of $\mathbb{P}^{dn-d-1}$ along $P_1,P_2,\dots ,P_n$ in any order.
\item $T_{d,n}^{\mathcal A_{3}}$ is the blow up of $T_{d,n}^{\mathcal A_2}$ along the strict transforms of the $({2d-1})$-planes spanned by all pairs of the $P_i$, $i=1,\dots ,n$, in any order.
\item$T_{d,n}^{\mathcal A_{4}}$ is the blow up of $T_{d,n}^{\mathcal A_{3}}$ along the strict transforms of the $({3d-1})$-planes spanned by all triples
of the $P_i$, $i=1,\dots ,n$, in any order.
$$
\vdots
$$
\item $T_{d,n}$ is the blow up of $T_{d,n}^{\mathcal A_{n-2}}$ along the strict transforms of the $(d(n-2)-1)$-planes spanned by
all $(n-2)$-tuples of the $P_i$, $i=1,\dots ,n$, in any order.
\end{itemize}
\end{corollary} 

\vspace{0.1 in}
(The proof of Corollary \ref{descrip} can be found in Section 4 right after the proof of Corollary \ref{mainheavy}.)
 We show that the space of weighted stable rooted trees appears naturally in the boundary of the compactification $\wpdn$ (see Theorem \ref{maintdn} for the precise statement), we study its geometry and generalize some of the main results in \cite{Chen-Gibney-Krashen} and \cite{Hassett-weighted}. More specifically, for any ordered set $\mathcal{A}=\{a_1,a_2,\dots,a_n\}$ and for any $I \subset\{1, \ldots, n \}$ let $\mathcal{A}(I):=\{ a_i \ | \ i \in I \}$ and $\mathcal{A}_+(I^c):=\{ a_i \ | \ i \not\in I \} \cup \{ a_{n+1}=1 \}$. Further, for any $\mathcal{A}\in D^T_{d,n}$ (see Section \ref{weights}), let $\wtdno$ be the weighted configuration space of $n$ labeled points in $\Aff^d$, up to translation and homothety, with respect to $\mathcal{A}$ (see Definition \ref{deftdno}). 
\begin{theorem}
\label{maintdn}
 \begin{enumerate}
\item For any $\mathcal{A}\in\mathcal D^T_{d,n}$, $\wtdn$ is a smooth projective variety, which contains $\wtdno$ as a dense open subset. The boundary $\wtdn\setminus \wtdno$ is a union of smooth irreducible divisors that intersect with normal crossings.
 
\item   \begin{enumerate}\item Let $\mathcal{A}\in\mathcal D^T_{d,n}$. Each irreducible divisor in the boundary of $T_{d,n}^{\mathcal{A}}$ has the form
$$
\Gamma_I=T_{d,|I|}^{\mathcal{A}(I)} \times T_{d,n-|I|+1}^{\mathcal{A}_+(I^c)} ,
$$
\noindent where $I\subsetneq \{1,\dots n\}$ and $\sum_{i\in I} a_i>1$.

\item Let $\mathcal{A}\in\mathcal D^P_{d,n}$. Each irreducible divisor in the boundary of $\overline P_{d,n}^{\mathcal{A}}$, has the form
$$
E_I \cong T_{d,|I|}^{\mathcal{A}(I)} \times  \overline P^{\mathcal A_+(I^c)}_{n-|I|+1}, 
$$
where $I\subsetneq \{d+1,\dots n\}$ and $\sum_{i\in I} a_i>1$.
\end{enumerate}
\end{enumerate}
\end{theorem}
\vspace{0.1in}
\begin{remark} Theorem \ref{maintdn} is proven in Section \ref{pTdn}. Note that part (2) is not obvious from the definitions, due to the fact that we do not yet have functors being represented by the varieties $\wtdn$ or $\wpdn$.
\end{remark}

Among all choices of weight data for fixed $d$ and $n$, there exist particular ones, which provide us with interesting examples of smooth toric varieties  $T_{d,n}^{\scriptscriptstyle LM}$ and $\overline P_{d,n}^{\scriptscriptstyle LM}$. These varieties can be viewed as higher dimensional analogs of the Losev-Manin spaces (Section \ref{sec:toric}).

\subsection{GIT quotients of wonderful compactifications} While our focus is the study of $\wpdn$ and $\wtdn$, the theoretical framework developed in this paper allows for many more compactifications of moduli problems arising in different contexts.  In particular, in Section \ref{relativeGIT}, we study GIT quotients of so called `wonderful compactifications of arrangements of subvarieties' of an arbitrary smooth variety (see \cite{LiLi}). Wonderful compactifications are always smooth with normal crossings boundary and can be described as a sequence of smooth blowups. Several compactifications in the literature can be obtained as wonderful compactifications: among these are the Fulton-MacPherson compactification \cite{Fulton-MacPherson}, Keel's construction of $\overline{M}_{0,n}$ \cite{Keel-thesis}, Ulyanov's polydiagonal compactification \cite{ulyanov} and Hu's compactification of open varieties \cite{hu2003compactification} (see Section 4 of [ibid.]). In Proposition \ref{WonGIT}, we show that the GIT quotients of wonderful compactifications are also wonderful compactifications under certain conditions. As a result, we can study degenerations of equivalence classes of points in an arbitrary smooth variety $X$ with a given group action (see Section \ref{screens} for a description).  We  use  stability conditions which are numerical generalizations of the ones used to define the moduli space of weighted stable curves of genus 0 (see Section 2 in \cite{Hassett-weighted}). Our requirement that the total weight of colliding points is less than or equal to one and that the points lie away from the singular locus resembles the one asking for at worst log canonical singularities.   We also require a minimum total weight on each component which is similar to  the ampleness condition  and is used to prevent additional blow ups. 
\\


\subsection{Examples}\label{sec:examples} Next, we illustrate our results. 
Two distinct points in $\mathbb{A}^2$ up to translation and homothety  have one degree of freedom. Indeed, we can always translate one of them to the origin, and  we can scale the second point along the line spanned by the two points. Therefore,   $T_{2,2} \cong \mathbb P^1$. To describe $T_{2,3}$,  we notice that the open locus parametrizing configurations of three distinct points in $\mathbb{A}^2$ up to translation and homothety is   $\mathbb{P}^3 \setminus \{L_{12},L_{13},L_{23} \}$ where $L_{ij}$ are disjoint lines. Each line $L_{ij}$ parametrizes a configuration with the double point $p_i=p_j$.  

\begin{figure}[h!]
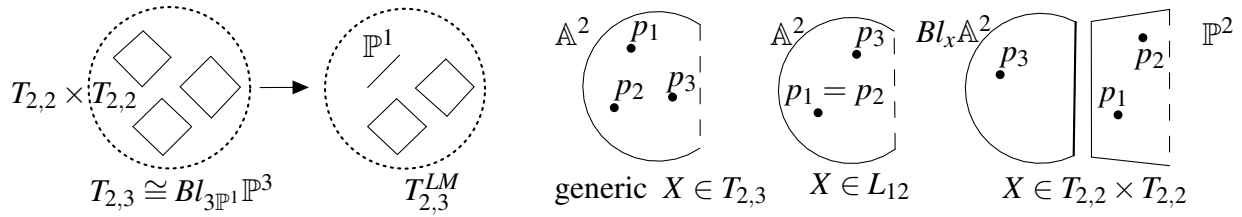

\hspace{-20em}
\tikzpicture[line cap=round,line join=round,>=triangle 45,x=1.0cm,y=1.0cm, scale=0.26]

\begin{scope}[shift={(0,-2)}]
\draw [dotted,  thick ]  (-9.14,13.99) circle (4.13cm);
\draw  [dotted, thick ]  (2.91,13.94) circle (4.09cm);

\draw [shift={(15.83,14.06)}] plot[domain=0.97:5.31,variable=\t]({1*3.83*cos(\t r)+0*3.83*sin(\t r)},{0*3.83*cos(\t r)+1*3.83*sin(\t r)});

\draw [shift={(25.74,13.85)}] plot[domain=0.95:5.35,variable=\t]({1*3.74*cos(\t r)+0*3.74*sin(\t r)},{0*3.74*cos(\t r)+1*3.74*sin(\t r)});

\draw [dash pattern=on 6pt off 6pt] (27.97,10.85)-- (27.92,16.89);
\draw [dash pattern=on 6pt off 6pt] (18,17.22)-- (18,10.89);

\draw [shift={(35.4,13.96)}] plot[domain=1.08:5.17,variable=\t]({1*3.84*cos(\t r)+0*3.84*sin(\t r)},{0*3.84*cos(\t r)+1*3.84*sin(\t r)});

\draw [ thick] (37.2,17.36)-- (37.1,10.53);
\draw [dash pattern=on 6pt off 6pt] (37.2,17.36)-- (37.1,10.53);

\draw (38,17.38)-- (38.01,10.44);
\draw (38.01,10.44)-- (42,10);
\draw (38,17.38)-- (42,18);
\draw [dash pattern=on 6pt off 6pt] (42,18)-- (42,10);
\draw (-10.48,16.9)-- (-12,15.53);
\draw (-12,15.53)-- (-10.53,14.04);
\draw (-10.53,14.04)-- (-9.06,15.55);
\draw (-10.48,16.9)-- (-9.06,15.55);
\draw (-9.52,13.44)-- (-11,12);
\draw (-11,12)-- (-9.47,10.53);
\draw (-9.47,10.53)-- (-8,12);
\draw (-8,12)-- (-9.52,13.44);
\draw (-8.59,14.04)-- (-7.02,15.6);
\draw (-7.02,15.6)-- (-5.51,14.02);
\draw (-5.51,14.02)-- (-6.99,12.44);
\draw (-6.99,12.44)-- (-8.59,14.04);
\draw (1,12)-- (2.48,13.43);
\draw (2.48,13.43)-- (4,12);
\draw (2.52,10.57)-- (4,12);
\draw (1,12)-- (2.52,10.57);
\draw (3.44,13.97)-- (5,15.46);
\draw (5,15.46)-- (6.44,14.04);
\draw (6.44,14.04)-- (5.06,12.53);
\draw (5.06,12.53)-- (3.44,13.97);
\draw (1,14)-- (2.63,15.66);
\draw [->] (-4.5,14) -- (-2,14.01);
\draw[color=black] (-14.0,13.5) node {$T_{2,2} \times T_{2,2}$};

\draw[color=black] (1.59,16.29) node {$\PP^1$};
\fill [color=black] (14.47,16.01) circle (6.5pt);
\draw[color=black] (15.22,16.89) node {$p_1$};
\draw[color=black] (11.50,16.89) node {$\mathbb A^2$};
\fill [color=black] (13.61,12.97) circle (6.5pt);
\draw[color=black] (14.39,13.85) node {$p_2$};
\fill [color=black] (16.59,13.51) circle (6.5pt);
\draw[color=black] (17.07,14.28) node {$p_3$};
\draw[color=black] (22.50,16.89) node {$\mathbb A^2$};
\fill [color=black] (24.03,12.71) circle (6.5pt);
\draw[color=black] (24.89,13.59) node {$p_1=p_2$};
\fill [color=black] (26,15.7) circle (6.5pt);
\draw[color=black] (26.72,16.59) node {$p_3$};
\fill [color=black] (33.34,14.66) circle (6.5pt);
\draw[color=black] (34.09,15.53) node {$p_3$};
\fill [color=black] (39.37,12.6) circle (6.5pt);
\draw[color=black] (39.13,13.49) node {$p_1$};
\draw[color=black] (31.00,16.89) node {$Bl_x\mathbb A^2$};
\draw[color=black] (44.50,16.89) node {$ \PP^2$};
\fill [color=black] (40.64,16.55) circle (6.5pt);
\draw[color=black] (41.0,15.46) node {$p_2$};
\draw[color=black] (-8.5,8.64) node {$T_{2,3} \cong Bl_{3\bP^1} \bP^3$};
\draw[color=black] (4.3,8.55) node {$T^{LM}_{2,3}$};
\draw[color=black] (16.01,8.64) node {\text{generic } $X \in T_{2,3} $};
\draw[color=black] (26.2,8.94) node {$X \in L_{12}$};
\draw[color=black] (38.22,8.64) node {$X \in T_{2,2} \times T_{2,2}$};
\end{scope}
\endtikzpicture
\subfloat[
Space parametrizing three points in $\mathbb A^2$ up to translation and homothety.
]{\hspace{.5\linewidth}}
\subfloat[
Parametrized stable rooted trees by the interior 
and the boundary.
]{\hspace{.5\linewidth}}
\caption[]{(A) depicts the compactifications  $T_{2,3}$ and $T^{LM}_{2,3}$, while (B) depicts the objects they parametrize \label{Fex23a}. }
\end{figure}
For weights equal to $1$ those double points are not allowed and we find that $T_{2,3}$ is the blow up  of $\mathbb{P}^3$ along these three lines $L_{ij}$.  Each of the boundary divisors can be interpreted as 
$T_{2,2} \times T_{2,2}$; they parametrize stable rooted trees that  decompose as the union of two components.   On the other hand, we can  choose weights allowing $p_1=p_2$ while not allowing the other double points.  The respective model  $T^{LM}_{2,3}$   is the blow up of $\bP^3$ along  the lines $L_{13}$ and $L_{23}$
(see Section \ref{sec:toric}).

\noindent
The space $\overline P_{2,5}$ is the blow up of $\bP^1 \times \bP^1$ at three points and the boundary divisor is the union of three disjoint $T_{2,2}$'s.  Indeed, we fix the points
$p_1$, $p_2$ and $p_3$   in general position and away from the other ones.  
The open locus parametrizing 
 five distinct points in $\bP^2$ up to an action of $Aut(\bP^2)$ is  $\bP^1 \times \bP^1$ minus three points. These points parametrize configurations where $p_4=p_5$, $p_3=p_5$ and $p_3=p_4$ respectively. 
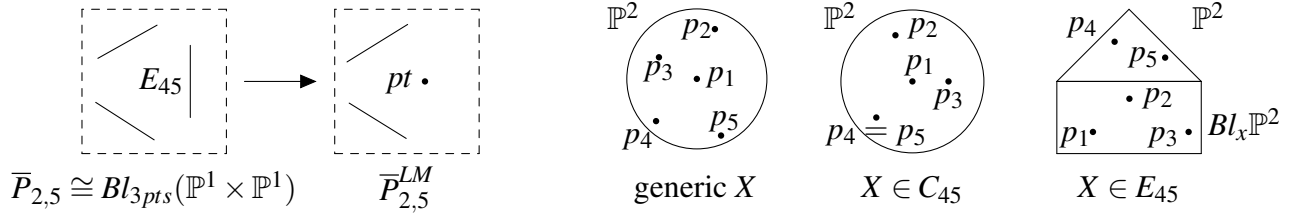
\begin{figure}[h!]
\tikzpicture[line cap=round,line join=round,>=triangle 45,x=1.0cm,y=1.0cm, scale=0.48]
\draw [color=black, dashed] (7,10)-- (7,6);
\draw [color=black, dashed] (7,6)-- (11,6);
\draw [color=black, dashed] (11,6)-- (11,10);
\draw [color=black, dashed] (11,10)-- (7,10);

\draw [->] (4.5,8) -- (6.5,8);

\draw [color=black, dashed] (0,10)-- (0,6);
\draw [color=black, dashed] (0,6)-- (4,6);
\draw [color=black, dashed] (4,6)-- (4,10);
\draw [color=black, dashed] (4,10)-- (0,10);
\draw (0.44,8.64)-- (2.08,9.56);
\draw (0.38,7.42)-- (2,6.38);
\draw (3,7)-- (3,9);
\draw (7.34,8.54)-- (9.02,9.58);
\draw (7.4,7.4)-- (8.98,6.42);
\draw(17.03,8.07) circle (1.94cm);
\draw(23,8) circle (1.98cm);
\draw [color=black] (29,10)-- (31,8);
\draw [color=black] (31,8)-- (27,8);
\draw [color=black] (27,8)-- (29,10);
\draw [color=black] (27,8)-- (27,6);
\draw [color=black] (27,6)-- (31,6);
\draw [color=black] (31,6)-- (31,8);
\draw [color=black] (31,8)-- (27,8);

\fill [color=black] (17.03,8.07) circle (2.5pt);
\draw[color=black] (17.71,8.12) node {$p_1$};

\fill [color=black] (17.53,9.45) circle (2.5pt);
\draw[color=black] (17.03,9.45) node {$p_2$};

\fill [color=black] (15.98,8.67) circle (2.5pt);
\draw[color=black] (15.98,8.27) node {$p_3$};

\fill [color=black] (15.9,6.9) circle (2.5pt);
\draw[color=black]  (15.4,6.4) node {$p_4$};

\fill [color=black] (17.7,6.5) circle (2.5pt);
\draw[color=black]   (17.8,6.9)  node {$p_5$};

\draw[color=black] (15,9.82) node {$\bP^2$};
\draw[color=black] (21,9.82) node {$\bP^2$};

\fill [color=black] (23,8) circle (2.5pt);
\draw[color=black] (23.21,8.44) node {$p_1$};

\fill [color=black] (22.53,9.28) circle (2.5pt);
\draw[color=black](23.3,9.48) node {$p_2$};

\fill [color=black] (22,7) circle (2.5pt);
\draw[color=black] (22,6.5)  node {$p_4=p_5$};

\fill [color=black] (24,8) circle (2.5pt);
\draw[color=black] (24,7.5)  node {$p_3$};

\fill [color=black] (9.5,8) circle (2.5pt);
\draw[color=black](8.8,8) node {$pt$};

\draw[color=black](2,5) node {$\overline P_{2,5} \cong Bl_{3pts}( \bP^1 \times \bP^1)$};
\draw[color=black](2.2,8) node {$E_{45}$};

\draw[color=black](9,5) node {$\overline P_{2,5}^{LM} $};
\draw[color=black](17,5) node {generic $X$};
\draw[color=black](23,5) node { $X \in C_{45}$};
\draw[color=black](29,5) node {$X \in E_{45} $};

\draw[color=black] (31.2,9.82) node {$\bP^2$};
\draw[color=black] (32.2,6.82) node {$Bl_x\bP^2$};

\fill [color=black] (28,6.6) circle (2.5pt);
\draw[color=black] (27.51,6.46) node {$p_1$};

\fill [color=black] (28.6,9.1) circle (2.5pt);
\draw[color=black]  (27.7,9.5)  node {$p_4$};

\fill [color=black] (30,8.66) circle (2.5pt);
\draw[color=black](29.4,8.66)  node {$p_5$};

\fill [color=black] (29.01,7.52) circle (2.5pt);
\draw[color=black]  (29.8 ,7.52)   node {$p_2$};

\fill [color=black] (30.65,6.6) circle (2.5pt);
\draw[color=black] (29.95,6.46) node {$p_3$};
\endtikzpicture
\subfloat[
Space parametrizing five points in $\mathbb P^2$ up to the $SL_3$ action.
]{\hspace{.5\linewidth}}
\subfloat[
Stable  trees parametrized by the interior 
and the boundary.
]{\hspace{.5\linewidth}}
\caption[]{(A) depicts the compactifications  $P^{LM}_{2,3}$ and $P_{2,3}$, while (B) depicts the objects they parametrize. \label{Fex23}  }
\end{figure}

\subsection{Acknowledgements}
We thank Valery Alexeev, Kenny Ascher, Angela Gibney,  Noah Giansiracusa for helpful discussions. We especially thank Danny Krashen for carefully reading an earlier draft of this article and for providing valuable feedback. The first author was supported by the NSF grant DMS-1344994 of the RTG in Algebra, Algebraic Geometry, and Number Theory, at the University of Georgia. The second author is supported by the World Premier International Research Center Initiative (WPI), MEXT, Japan.

\tableofcontents

\subsection{Conventions} Throughout this paper, the term variety will be understood as a reduced and irreducible scheme of finite type defined over an algebraically closed field of characteristic 0. Also, we will often denote the set of integers $\{1,2,\dots,n\}$ by the capital letter $N$.
\section{Description of the Parametrized Objects}\label{sec:desc}

In this section, we describe the three types of parametrized objects that appear in our work.  First,  the weighted stable degenerations of $n$ labeled points in an arbitrary nonsingular variety $X$. Second, the weighted stable rooted trees which are degenerations of  $n$ labeled points in $\mathbb A^d$ defined up to translation and homothety. Third, the weighted  stable trees which are degenerations of $n$ points in $\bP^d$ defined up to an action of  $SL_{d+1}$. 
\subsection{Weight Domains} \label{weights} Let $X$ be a smooth variety with $\dim X=d\geq1$ and let $n\geq 2$. The domain of admissible weights for the \textit{weighted compactifications} of the configuration space of $n$ labeled points in $X$ (section \ref{FM}) is given by
\begin{align*}
\mathcal D_{d,n}^{FM} :=
\left\{
(a_1, \ldots,  a_n) \in \mathbb{Q}^{n} \; : \;
0 < a_i\leq 1\,\,,i=1,\dots, n
\right\}
\end{align*}
\noindent
The domain of admissible weights  for the space of \textit{weighted stable rooted trees} (Section \ref{pTdn}) is 
\begin{align*}
\mathcal D_{d,n}^T :=
\left\{
(a_1, \ldots,  a_n) \in \mathbb{Q}^{n} \; : \;
0 < a_i \leq 1 \;\,\,,i=1,\dots, n   \text{ and } 1 < a_1 + \ldots +a_n 
\right\}
\end{align*}Finally,
let us fix $d$ and $n$ such that $d\geq 1$ and $n \geq d+2$ and let $\epsilon=\frac{1}{n-d}$ and  $\hat \epsilon=\frac{1}{(d+1)(n-d)}$. We consider the set of weights 
\begin{align*}
w_1=\ldots =w_d=1-\hat \epsilon,
& &
w_{d+1}=1-(n-(d+1))\epsilon +d \hat \epsilon, 
& & 
 w_{d+2}=\ldots = w_n=\epsilon. 
\end{align*}
Then the domain of admissible weights for the space of \textit{weighted stable trees} (Section \ref{conspdn}) is
\begin{align*} \mathcal D_{d,n}^P = \left\{ (a_1,  \ldots a_n) \in \mathbb{Q}^{n} \; : 
w_i \leq a_i \leq 1\,\,,i=1,\dots, n
\right\}
\end{align*}
These last constraints are motivated by a technical requirement in Lemma \ref{GITpoints}.
In the sequel, we will often refer to the number $a_i$ as a \emph{weight} (of some labeled point $p_i$ in a configuration).  
Given  $I \subset N:=\{1, \ldots, n \}$ and $\mathcal{A}=\{a_1,\dots, a_n\}$, we define
 \begin{align*}
\mathcal{A}(I):=\{ a_{i} \ | \ i \in I \} \hspace{0.1in}& \text{and} \hspace{0.1in} \mathcal{A}_+(I^c):=\{ a_i \ | \ i \not\in I \} \cup \{ a_{n+1}=1 \}.
\end{align*}

\subsection{Weighted stable degenerations}\label{wdeg}(\cite{routis2014weighted}; see also the descriptions in \cite{Fulton-MacPherson} and \cite{pandharipande1995geometric} for the case where all weights are equal to 1.) Let $X$ be a nonsingular variety of dimension $d$. Let $(x_1,x_2,\dots, x_n)$  be an $n$-tuple of labeled points $x_i\in X$ and consider an ordered set $\mathcal{A}\in \mathcal D_{d,n}^{FM}$. We say that $x_i$ has weight $a_i$.
\begin{definition}\label{defscreens} \begin{enumerate}\item A subset $S\subset\{1, \ldots, n \}$ is said to be $\mathcal{A}$-coincident at $x \in X$ if 
 \begin{enumerate}
 \item the total weight of the points labeled by $S$ is larger than one, that is, $\sum \limits _{i\in S} a_i>1$ and
 \item for all  $i\in S$, $x_i=x$.\\
 \end{enumerate} 
\noindent  We will sometimes write $(S,x)$ in place of $S$ for emphasis.

 \vspace{0.1in}
\item A screen of an $\mathcal{A}$-coincident set $S$ at $x\in X$ consists of data $(t_i)_{i\in S}$ such that
 \begin{enumerate}
\item $t_i\in T_x$, the tangent space of $X$ at $x$ ;
\item there exist $i,j\in S$ such that $t_i\neq t_j$.
\end{enumerate}
Two data sets $(t_i)_{i\in S}$ and $(t_i')_{i\in S}$ are equivalent if there exist $c\in\mathbb G_m $ and $v\in T_x$ such that 
$$
c\cdot t_i +v =t_i'
$$ 
for all $i\in S$. 
\end{enumerate}
\end{definition}
 In other words, if we identify $T_x$ with the affine space $\Aff^d$, then $(t_i)_{i\in S}$ defines an equivalence class of points in $\Aff^d$ up to translation and homothety. Now let $X$ and $\mathcal{A}$ as above. We give the following definition.

 \begin{definition} \label{compatible}A compatible collection of $\mathcal{A}$-coincident sets and screens at $x\in X$ consists of the following data:
 \begin{enumerate}
 \item A collection $\mathcal{C}_x$ of $\mathcal{A}$-coincident sets at $x$ with the following property: given any two sets in $\mathcal{C}_x$, then either one is contained in the other or they are disjoint. 
 \item A screen $Q_S$ for each $S\in\mathcal{C}_x$. Moreover, the collection of all such screens has the following property: given $S_1, S_2\in \mathcal{C}_x$  such that $S_1\supset S_2$, then the equivalence class of data $(t_i)_{S_1}$ of the screen $Q_{S_1}$ satisfies $t_i=t_j$ for all $i,j\in S_2$.
 \end{enumerate}
 \end{definition}
\indent Now, consider the $n$-tuple $\vec{x}:=(x_1,x_2,\dots x_n)$ and let $x\in X$ appear multiple times in $(x_1,x_2,\dots x_n)$. For all such $x$ let $\mathcal{C}_x$ be a compatible collection of $\mathcal{A}$-coincident sets at $x$; such a collection could be empty if $\sum\limits_{\{i|x_i=x\}}a_i\leq 1$. We construct an \textit{$n$-pointed $\mathcal{A}$-stable degeneration of ($X,\vec{x})$} as follows. Let $S\in\mathcal{C}_x$ be the maximal (with respect to inclusion) $\mathcal{A}$-coincident at $x$. If $S$ is empty then we don't modify $X$; otherwise, we blow up $X$ at $x$ and attach the projective completion $\PP(T_x\oplus1) \cong \bP^d$ of $T_x$ along the exceptional divisor $\PP(T_x) \cong \bP^{d-1}$, which is identified with the infinity section. 
Note that the complement $\PP(T_x\oplus1)\setminus \PP(T_x)$ is isomorphic to the affine space $T_x \cong \mathbb A^{d}$. The data (equivalence class of tangent vectors) of the screen corresponding to $S$ specify points of $T_x$ (defined up to translation and homothety) labeled by the elements of $S$. By condition (2) in Definition \ref{defscreens}, we see that some separation of those points occurs inside the new component $\PP(T_x\oplus1)$. The maximal (with respect to inclusion) among the $\mathcal{A}$-coincident sets at $x$ that are contained in $S$ and their screens specify further blowups whose centers are points in the new affine space $T_x$. We continue this process until all coincident sets and their screens have been used, for all coordinates $x$ that occur multiple times in $(x_1,x_2,\dots x_n)$. The resulting variety is equipped with $n$ points $s_i$ lying in its smooth locus. By this description we see that if $T\subset \{1, \ldots, n \}$ and $\sum\limits_{i\in T}a_i>1$ then some separation of the points $(s_i)_{i\in T}$ necessarily occurs. This means that if the sections $(s_i)_{i\in T}$ all coincide for some $T$, then $\sum\limits_{i\in T}a_i\leq1$.\\
\begin{example}\label{exSdg} (see Figure \ref{wsdg})
Let $X$ be a $d$-dimensional variety, $n=6$ and $\mathcal{A}=\{\frac{1}{4}+\epsilon,\frac{1}{4}+\epsilon,\frac{1}{4}+\epsilon,\frac{1}{4}+\epsilon,\frac{1}{2}+\epsilon,\frac{1}{2}+\epsilon\}$. We describe an $\mathcal{A}$-stable degeneration of $X$ associated to the collection of $\mathcal{A}$-coincident sets at a point $x\in X$:
$$
\{ 1,2, 3,4, 5, 6\},  \{ 1, 2, 3, 4\} ,\{  5,6 \}.
$$ 
The distinguished component is a blowup of the original variety $X$ at the point $x$. The two end components are isomorphic to  $\PP^d$, where $d=\dim(X)$; on each of the end components we have two distinct loci of (possibly coincident) smooth markings. \qed
\begin{figure}[h]
\begin{tikzpicture}[scale=0.9]
\begin{scope}[shift={(-5,0)}]
\draw [fill=white] (1.7,1.2) ellipse (2cm and 1cm);
\draw [fill=white] (0.75,1.2) rectangle (2.9,3.8);
\draw [fill=white] (0.85,2.55) rectangle (1.8,4.6);
\draw [fill=white] (2, 2.55) rectangle (2.75,4.6);
\node[mark size=2pt,color=black] at (1.2,3.3) {\pgfuseplotmark{*}};
\node[mark size=2pt,color=black] at (2.3,3.9) {\pgfuseplotmark{*}};
\node[mark size=2pt,color=black] at (2.3,3.2) {\pgfuseplotmark{*}};
\node[mark size=2pt,color=black] at (1.2,4.2) {\pgfuseplotmark{*}};
\node [below] at (-0.3,3.5) {$s_1=s_2$};
\node [above] at (3.1,3.9) {$s_6$};
\node [below] at (2.3,3.2) {$s_5$};
\node [above] at (-0.5,4) {$s_3=s_4$};
\node [above] at (4.2,1) {$Bl_xX$};
\node [above] at (1.95,1.3) {$Bl_{x_1,x_2}\PP^d$};
\node [above] at (1.3,4.65) {$\PP^d$};
\node [above] at (2.45,4.65) {$\PP^d$};
\end{scope}
\begin{scope}[shift={(4,0.5)}]

\node [above] at (-2.5,3.6) {$\{(1,2), (3,4) \}$};
\fill [color=black] (-2,3.6) circle (2.5pt);
\draw [->] (-1,2.5) -- (-2,3.5);

\node [above] at (0.5,3.6) {$\{5,6 \}$};
\draw [->] (-1,2.5) -- (0,3.5);
\fill [color=black] (0,3.6) circle (2.5pt);

\draw [->] (-1,1) -- (-1,2.3);
\fill [color=black] (-1,2.4) circle (2.5pt);

\fill [color=black] (-1,0.9) circle (2.5pt);
\node [below] at (-1,0.9) {};
\node [ below] at (-1,0.8)  [align=left]{ distinguished  \\ component};
\end{scope}
\end{tikzpicture}
\caption[A weighted stable degeneration]{\footnotesize   \textit{The weighted stable degeneration 
described in Example \ref{exSdg} and its associated dual graph.}}\label{wsdg}
\end{figure}
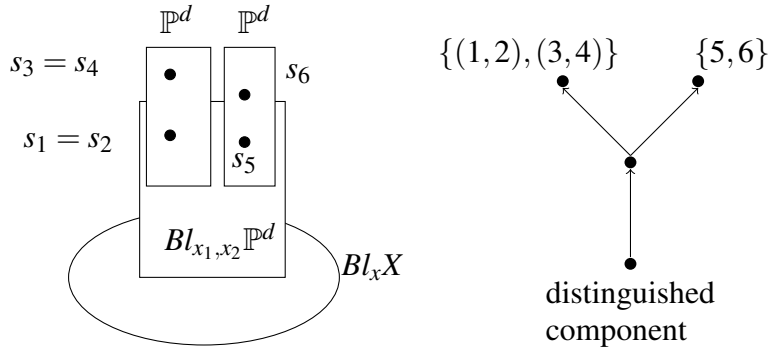
\end{example}
\noindent
To any $\mathcal{A}$-stable degeneration we associate a tree, its dual graph,  whose vertices are in one to one correspondence with its components and whose vertices are in one to one correspondence with the nonempty intersections of its components.  In general, we have
the following types of  components:
\begin{enumerate}
\item A distinguished component which is a blowup of $X$ at a finite set of points. 
\item \textit{End} components are the irreducible components whose vertex has valence equal to 1 and are different from
distinguished component.  Any end component is isomorphic to $\PP^d$ and comes with at least three distinct markings: at least two coming from distinct smooth points  and exactly one from an intersection with another component, which is a divisor of that end component. 
\item  \textit{Ruled} components are the irreducible components whose vertex has valence 2
; they isomorphic to $\bP^d$ blown up at a point.  Any ruled component  also comes with three distinct markings: at least one from a smooth point  and exactly two from intersections with other components (which are divisors of the ruled component). 
\item Any other component different to the above ones
is isomorphic to $\bP^d$ blown up at -at least- two distinct points. It also
comes with at least three distinct markings
which can be either from a smooth point or from  intersections with other components.
\end{enumerate}
Let $(W,s_i)$ be an $\mathcal{A}$- stable degeneration of $X$, where $s_i$ are its labeled smooth sections. An \textit{isomorphism of $\mathcal{A}$- stable degenerations $(W,s_i)$ and $(W',s'_i)$} is an isomorphism of schemes $W\rightarrow W'$ which sends $s_i$ to $s'_i$ for all $i$ and fixes $X$ pointwise. Since, by the above discussion, all components of a stable degeneration of $X$, except for the distinguished one, are equipped with at least three distinct markings, one of which is necessarily  hyperplane, it cannot have nontrivial automorphisms. This justifies the term \textit{stability}.

\subsection{Weighted stable degenerations of $X$ with respect to a group action} \label{screens} 
Let $\mathcal{A}=\{a_1,a_2,\dots,a_n\}$ be a set of rational numbers $a_i\in(0,1]$, $X$ a nonsingular variety and $G$ an algebraic group acting on $X$. Further, let $U\subset X^n$ be an open subvariety which is invariant under the diagonal action of $G$ on $X^n$ and $\vec{x}=(x_1,x_2,\dots, x_n)\in U$. The group action allows us to define an equivalence relation on the set of weighted stable degenerations of $X$ associated with all $n$ tuples in the orbit $G\cdot\vec{x}$. Indeed,  we have:\\
\indent \textbf{1}.\,  \textit{An equivalence class of $n$-tuples in $U$. } We consider all $n$-tuples in the orbit $G\cdot \vec{x}$ equivalent. 
  
\textbf{2}. \textit{An equivalence class of $\mathcal{A}$-coincident sets}. Recall (Definition \ref{defscreens}) that a subset $S\subset N$ is $\mathcal{A}$-coincident at $x$ if $\sum \limits _{i\in S} a_i>1$ and $x_i=x$ for all $i\in S$. We write $(S,x)$ for emphasis and we say that $(S,x)$ is equivalent to $(S,g\cdot x)$. 
\\

\textbf{3}. \textit{An equivalence class of screens } as follows.  Let $(t_i)_{i\in S}$ be the screen data of an $\mathcal{A}$-coincident set $(S,x)$ at $x\in X$ (Definition \ref{defscreens}). 
An element $g\in G$ induces a map on tangent spaces $T_x  \rightarrow T_{g \cdot x}$. Then, we define $(g \cdot t_i)$ 
in  $T_{g \cdot x}$ to be the image of $t_i$ via $T_x\to T_{g \cdot x}$. It follows immediately that
\begin{align*}
(t_i)_{i \in S} \sim (t'_i)_{i \in S}  
\iff
 (g \cdot t_i)_{i \in S} \sim (g \cdot t'_i)_{i \in S}
\end{align*}
where $(\sim)$ is the equivalence relation by translation and homothety among the screen data (Definition \ref{defscreens}(2)). Therefore, the $(g \cdot t_i)_{i \in S}$ form screen data for the $\mathcal{A}$-coincident set $(S,{g\cdot x})$ at $g\cdot x$. We say that the screen data $( t_i)_{i \in S}$ and $(g \cdot t_i)_{i \in S}$ are \textit{equivalent} and write $( t_i)_{i \in S}\equiv(g \cdot t_i)_{i \in S}$. It is straightforward to check that $(\equiv)$ is an equivalence relation, so we define an equivalence class of screens.\\

The above equivalence relations give us an obvious equivalence relation among the set of all compatible collections of $\mathcal{A}$-coincident sets and screens (Definition \ref{compatible}), thus we get:\\
\indent \textbf{4}. \textit{An equivalence class of compatible collections of $\mathcal{A}$-coincident sets and screens}.\\

Finally, since a weighted stable degeneration of $(X,\vec{x})$ is in one to one correspondence with compatible collections of $\mathcal{A}$-coincident sets and screens at each coordinate that appears with multiplicity greater than one in $\vec{x}$ (if any), we get:\\
\indent \textbf{5}. \textit{An equivalence class of $n$-pointed $\mathcal{A}$ stable degenerations} of $X$ with respect to $G$. In particular, two degenerations of $X$ are equivalent if  and only if 
\begin{enumerate}
\item Their distinguished components are equivalent in the following sense. Recall that the distinguished components of the two degenerations are equivalent to data $(X,\vec{y})$ and $(X,\vec{z})$ where $\vec{x}$ and $\vec{y}$ are two initial  configurations of $n$ labeled points. Then$(X,\vec{y})$ is equivalent to $(X,\vec{z})$ if and only if there exists $g\in G$ such that $y_i=g\cdot z_i$ for all $i=1,\dots,n$.
\item Their non-distinguished components are equivalent in the following sense: each $\PP(T_x\oplus1)$ arising in the construction of \ref{wdeg} from a coordinate $x$ that appears multiple times in $(x_1,\dots x_n)$ is identified with $\PP(T_{g \cdot x}\oplus1)$ via $T_x\to T_{g \cdot x}$.  Moreover, $\PP(T_x)$ is identified with $\PP(T_{g \cdot x})$ and the corresponding markings (smooth points) defined by the screen data, which lie in $T_x=\PP(T_x\oplus1)\setminus \PP(T_x)$ by construction, are identified with the corresponding markings in $T_{g\cdot x}=\PP(T_{g\cdot x}\oplus1)\setminus \PP(T_{g\cdot x})$.

\end{enumerate}
Next, we describe how the geometric objects parametrized by $\wtdn$ and $\wpdn$ are obtained by the above procedure. They will be weighted pointed stable degenerations of $X$ with respect to $G$ for suitably chosen input data $X, \mathcal{A}, G$ with an action 
$G \times X \to X$ and an open subvariety $U\subset X^n$, which is invariant under the diagonal induced action of $G$ on $X^n$ . 

\subsection{Weighted stable rooted trees}\label{rootedtrees}  Let $\mathcal{A}$ be a set of weights in $\mathcal D_{d,n}^T$.   The geometric points of $\wtdn$ are obtained by the procedure in Section \ref{screens} for input data $\mathbb A^d, \mathcal{A}$, the group $G$ that acts by translation and homothety on $\mathbb A^d$ and $U=(\Aff^d)^n\setminus \Delta _N$, where $\Delta _N$ is the small diagonal, i.e. the locus parametrizing configurations of $n$ coincident points in $\Aff^d$. 
\begin{figure}[h!]
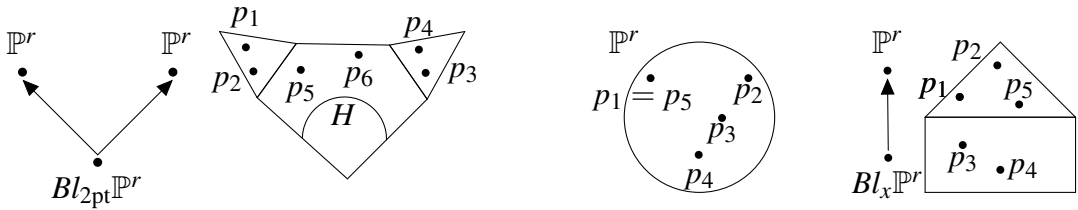



\tikzpicture[line cap=round,line join=round,>=triangle 45, scale=1.0]

\begin{scope}[shift={(-10,0)}]
\draw [->] (-1,2.5) -- (-2,3.5);
\draw [->] (-1,2.5) -- (0,3.5);

\draw [shift={(2.28,2.71)}] plot[domain=-0.02:3.12,variable=\t]({1*0.56*cos(\t r)+0*0.56*sin(\t r)},{0*0.56*cos(\t r)+1*0.56*sin(\t r)});
\draw [color=black]  (2.27,3.03)node {$H$};

\draw [color=black] (1.12,3.26)-- (1.64,3.98);
\draw [color=black] (1.64,3.98)-- (2.88,3.96);
\draw [color=black] (2.88,3.96)-- (3.38,3.24);
\draw [color=black] (3.38,3.24)-- (2.32,2.18);
\draw [color=black] (2.32,2.18)-- (1.12,3.26);
\draw [color=black] (2.88,3.96)-- (3.88,4.14);
\draw [color=black] (3.88,4.14)-- (3.38,3.24);
\draw [color=black] (3.38,3.24)-- (2.88,3.96);
\draw [color=black] (1.64,3.98)-- (0.62,4.14);
\draw [color=black] (0.62,4.14)-- (1.12,3.26);
\draw [color=black] (1.12,3.26)-- (1.64,3.98);
\fill [color=black] (-1,2.4) circle (1.7pt);
\draw[color=black] (-1.01,2) node {$Bl_{2\text{pt}}\PP^r$};
\fill [color=black] (-2,3.6) circle (1.7pt);
\draw[color=black] (-2.01,3.98) node {$\PP^r$};
\fill [color=black] (0,3.6) circle (1.7pt);
\draw[color=black] (0.05,3.98) node {$\PP^r$};

\fill [color=black] (0.97,3.93) circle (1.5pt);
\draw [color=black]  (0.97,4.33)node {$p_1$};

\fill [color=black] (1.7,3.63) circle (1.5pt);
\draw [color=black]  (1.7,3.33)node {$p_5$};

\fill [color=black] (2.47,3.83) circle (1.5pt);
\draw [color=black]  (2.47,3.53)node {$p_6$};

\fill [color=black] (1.07,3.61) circle (1.5pt);
\draw [color=black]   (0.72,3.46)  node {$p_2$};

\fill [color=black] (3.36,3.59) circle (1.5pt);
\draw [color=black]  (3.86,3.59) node {$p_3$};

\fill [color=black] (3.27,3.9) circle (1.5pt);
\draw [color=black]  (3.27,4.2)  node {$p_4$};
\end{scope}

\begin{scope}[shift={(-1,1)}]
\draw(-2,2) circle (1cm);
\draw[color=black] (-3,3) node {$\PP^r$};
\draw [color=black] (1,2)-- (1,1);
\draw [color=black] (1,1)-- (3,1);
\draw [color=black] (3,1)-- (3,2);
\draw [color=black] (3,2)-- (1,2);
\draw [color=black] (1,2)-- (2,3);
\draw [color=black] (2,3)-- (3,2);
\draw [color=black] (3,2)-- (1,2);

\draw [->] (0.51,1.56) -- (0.5,2.52);
\fill [color=black] (0.51,1.46) circle (1.5pt);
\fill [color=black] (0.5,2.62) circle (1.5pt);
\draw[color=black] (0.5,3) node {$\PP^r$};
\draw[color=black] (0.5,1.1) node {$Bl_x\PP^r$};

\fill [color=black] (-2.65,2.52) circle (1.5pt);
\draw [color=black] (-2.75,2.22)node {$p_1=p_5$};

\fill [color=black] (-1.35,2.52) circle (1.5pt);
\draw [color=black]  (-1.35,2.3) node {$p_2$};

\fill [color=black] (-1.7,1.99) circle (1.5pt);
\draw [color=black]  (-1.7,1.78)  node {$p_3$};

\fill [color=black] (-2.01,1.5) circle (1.5pt);
\draw [color=black]  (-2,1.18)  node {$p_4$};

\fill [color=black] (2.25,2.17) circle (1.5pt);
\draw [color=black] (2.25,2.37) node {$p_5$};

\fill [color=black] (1.96,2.69) circle (1.5pt);
\draw [color=black] (1.56,2.89) node {$p_2$};

\fill [color=black] (1.46,2.27) circle (1.5pt);
\draw [color=black] (1.12,2.39)  node {$p_1$};

\fill [color=black] (1.46,2.27) circle (1.5pt);
\draw [color=black] (1.12,2.39)  node {$p_1$};

\fill [color=black] (1.5,1.63) circle (1.5pt);
\draw [color=black] (1.5,1.40)  node {$p_3$};

\fill [color=black] (2,1.3) circle (1.5pt);
\draw [color=black](2.3,1.3)  node {$p_4$};

\end{scope}

\endtikzpicture
\subfloat[
weighted stable rooted tree
]{\hspace{.5\linewidth}}
\subfloat[
weighted stable trees
]{\hspace{.5\linewidth}}
\caption[]{Examples of parametrized objects and their dual graphs}
\end{figure}

 It is convenient  to think of points in $\mathbb A^d$ as points in $\mathbb{P}^d$ that lie away from a fixed hyperplane $H \subset \mathbb{P}^d$ which is called the \textit{root}. 
Let $G \cong \mathbb G_m \rtimes \mathbb G_a^d \subset Aut(\mathbb{P}^d)$ be the group which fixes the root pointwise. Under this interpretation, $X=\bP^d\setminus H$ and  the equivalence class of $n$ points is determined by the restriction of the action of  $G$ to $X=\bP^d\setminus H$.  \\

The dual graph of the resulting variety is a rooted tree.  The distinguished vertex corresponds to the variety that contains the root $H$.

\subsection{Weighted stable trees}\label{trees}
Let $\mathcal{A}$ be a set of weights in $\mathcal D_{d,n}^P$. 
The geometric points of $\wpdn$  are obtained by the procedure in Section \ref{screens} for input data $\mathbb P^d, \mathcal{A}$, $G \cong SL_{d+1}$ with the usual action on $\bP^d$ and $U\subset (\PP^d)^n$ defined by the following conditions:
 \begin{enumerate}
\item $p_1, \ldots , p_d, p_{d+1}$ are in general position;
\item none of the $p_i,\, i\in\{d+2,\dots n\}$ can lie in the linear subspace spanned by $p_1, \ldots , p_d$  ;
\item we cannot have $p_{d+1}=\ldots =p_n$ and
\item the points $p_i,\, i=d+2,\dots ,n$ cannot all lie on the hyperplane  spanned by\\
 $p_1, \ldots  p_{k-1},p_{k+1}, \ldots ,p_{d+1} $ simultaneously.
\end{enumerate}
The geometric meaning of these last conditions will become apparent in Lemma \ref{GITpoints}. The resulting variety has a dual graph that is a tree whose distinguished vertex corresponds to the original $\bP^d$, where $p_1, \ldots , p_d, p_{d+1}$ lie. 

\section{Wonderful Compactifications and GIT}\label{sec:WondGIT}

In order to construct our compactifications $\wpdn$ and $\wtdn$ we make use of the theory of wonderful compactifications (\cite{LiLi}; Section \ref{wonder}) and relative GIT (\cite{hu1996relative}; Section \ref{relativeGIT}). The theory of wonderful compactifications will allow us to describe our compactifications as iterated blowups at smooth centers. $\wpdn$ will be constructed in Section 4 as a GIT quotient of a particular wonderful compactification, the weighted Fulton-MacPherson compactification of $n$ labeled points in $\PP^d$ associated to $\mathcal{A}$(\cite{routis2014weighted}; Section \ref{FM}). A descent result for wonderful compactifications (Proposition \ref{WonGIT}) will then allow us to conclude that $\wpdn$ is also wonderful.

\subsection{ Wonderful Compactifications} \label{wonder}
We give a summary of the results in \cite{LiLi} that we will use in the sequel.

\begin{definition}\cite[Sec 2.1]{LiLi}\label{factors}
An \textit{arrangement} of subvarieties of a nonsingular variety $Y$ is a finite set $\mathcal{S}=\{S_i\}$ of nonsingular closed  subvarieties $S_i \subset Y$ such that 
\begin{enumerate} 
\item any two varieties $S_i$ and $S_j$ \textit{  intersect cleanly}, i.e. their set theoretic intersection is nonsingular and their tangent spaces satisfy the equality $T_{S_i \cap S_j, y} = T_{S_i,y}\cap T_{S_j,y}$ for all $y\in S_i\cap S_j$. 
\item $S_i \cap S_j$ is either equal to some $S_k$ or empty.
\end{enumerate}
\end{definition}
We say that $S_i$ and $S_j$ intersect \textit{transversally} if, for every point $y$ in $Y$, we have   
$T_{Y,y} = T_{S_i,y}+T_{S_j,y}$; here, if $y\notin S$, for some $S\subsetneq Y$, we adopt the convention that $T_{S,y}:=T_{Y,y}$. \\
\begin{lemma}\cite[Lemma 5.1]{LiLi} Let $S_1,S_2$ be two smooth subvarieties of a smooth variety $Y$. Then $S_1,S_2$ intersect cleanly if and only if their scheme theoretic intersection is smooth.
\end{lemma}\label{clean}
\begin{definition}\label{defbuild}\cite[Sec 2.1]{LiLi}\begin{enumerate}
\item Let $\mathcal S$ be an arrangement of subvarieties of a nonsingular variety $Y$. A subset $\mathcal G \subset \mathcal S$ is a \textit{building set of $\mathcal S $} if for all $S\in\mathcal S$ the minimal elements of $\mathcal G$ containing $S$ intersect transversally and their intersection is equal to $S$ (this condition is trivially satisfied if $S\in \mathcal G$). These minimal elements are called the $\mathcal G$-factors of $S$.
\item A finite set $\mathcal{G}$ of nonsingular subvarieties of $Y$ is called a \textit{building set} if the set of all possible intersections of collections of subvarieties from $\mathcal{G}$ forms an arrangement $\mathcal{S}$ and $\mathcal{G}$ is a building set of $\mathcal{S}$. In this situation, $\mathcal{S}$ is called the \textit{arrangement induced} by $\mathcal{G}$.
\end{enumerate}
\end{definition}

\begin{definition} \label{Aconst}\cite[Def 1.1]{LiLi}
Let $\mathcal{G}$ be a building set of an arrangement of subvarieties of a nonsingular variety $Y$.  
The \textit{wonderful compactification} $Y_{\mathcal G}$ of $\mathcal G$ is the closure of the image of the natural locally closed embedding 
$$
Y^{\circ} := Y \setminus \bigcup_{S_k \in \mathcal G} S_k
\hookrightarrow\ 
 \prod \limits_ {S_k \in\ \mathcal G} Bl_{S_k} Y.
$$
\end{definition}

\begin{definition}\label{dominant}\cite[Definition 2.7]{LiLi} Let $Z$ be a nonsingular subvariety of a nonsingular variety $Y$ and $\pi : Bl_ZY \rightarrow Y$ be the blow-up of $Y$ along $Z$. For any subvariety $V$ of $Y$, the \textit{dominant transform} of $V$, denoted by $\widetilde{V}$ or $V\texttildelow$, is the strict transform of $V$ if $V \not\subset Z$ or the scheme-theoretic inverse $\pi^{-1}(V )$ if $V\subset Z$. \\
For a sequence of blow-ups, we still denote the iterated dominant transform $(\dots((V^{\texttildelow})^{\texttildelow})\dots)^{\texttildelow}$ by $\widetilde{V}$ or $V\texttildelow$.
\end{definition}
\begin{definition}\label{iterated} Let $\mathcal{G}=\{S_1,S_2,\dots ,S_n\}$ be a totally ordered set of subvarieties of a nonsingular variety $Y$. For $k=0,1,\dots, n$, we define $(Y_k,\mathcal{G}^{(k)})$ inductively as follows.
\begin{enumerate}[label=\roman*.]
\item Let $Y_0=Y, S_i^{(0)}=S_i$ for $1\leq i\leq n$ and $\mathcal{G}^{(0)} =\mathcal{G}=\{S_1,S_2,\dots ,S_n\}$ .
\item Let $k\geq1$ and suppose that $(Y_{k-1}, \mathcal{G}^{(k-1)}=\{S_1^{(k-1)},S_2^{(k-1)},\dots,S_n^{(k-1)}\} )$ has been constructed. 
\begin{itemize}
\item We define $Y_k$ to be the blowup of $Y_{k-1}$ at the subvariety $S_k^{(k-1)}$.
\item We define $S_i^{(k)}$ to be the dominant transform of $S_i^{(k-1)}$ under the blowup $Y_k\rightarrow Y_{k-1}$ for $1\leq i\leq n$ and $\mathcal{G}^{(k)}$ to be the set $\{S_1^{(k)},S_2^{(k)},\dots,S_n^{(k)}\} $.
\end{itemize}
\end{enumerate}

The \textit{iterated blowup} $Bl_{\mathcal{G}} Y$ of $Y$ at $\mathcal{G}$ is defined to be the variety $Y_n$. \\

\begin{proposition}\label{buildpullback}\ Let $\mathcal{G}=\{S_1,S_2,\dots ,S_n\}$ be a set of subvarieties of a nonsingular variety $Y$. Assume that $\mathcal{G}$ is a building set and that it is given a total order compatible with inclusions, that  is $i\leq j$ if and only if $S_i\subseteq S_j$. Let $\mathcal{T}^{(k)}$ be the arrangement induced by $\mathcal{G}^{(k)}$.  
\begin{enumerate}\item $\mathcal{G}^{(k)}$ is a building set of $\mathcal{T}^{(k)}$. 
\item Let any $i$ and $k$ such that $1\leq i, k\leq n$ and $i\neq k$. If $S_k\not\subseteq S_i $, then the intersection $S_{k}^{(k-1)}\cap S_{i}^{(k)}$ is transversal. 
\end{enumerate}
\end{proposition}
\begin{proof} See the Remark after Definition 2.12 and Lemma 2.6(i) in \cite{LiLi}.
\end{proof}
\end{definition}
The following theorem  will be central to our constructions:
\begin{theorem}(\cite[Theorems 1.2,1.3 and Proposition 2.13]{LiLi})\label{thmLi}
Let $Y$ be a nonsingular variety and $\mathcal G=\{S_1,S_2,\dots ,S_n\}$ be a building set of an arrangement of subvarieties of $Y$. 
\begin{enumerate}
\item The wonderful compactification $Y_\mathcal G$ is a nonsingular variety. Moreover, for every $S_i\in\mathcal{G}$ there exists an irreducible smooth divisor $D_{S_i}$ such that :
\begin{enumerate}[label=\roman*.]
\item  $Y_\mathcal G\setminus Y^{\circ}=\bigcup_{i=1}^n D_{S_i}$.
\item Any set of divisors $D_{S_i}$ intersects  transversally. 
\end{enumerate}
\item Suppose that the elements of $\mathcal G$ are arranged either
\begin{enumerate} 
\item in an order compatible with inclusion relations (i.e. $i\leq j$ iff $S_i\subseteq S_j$) or 
\item in such an order that the first $i$ terms  $S_1,S_2,\dots ,S_i$ form a building set for all $1\leq i \leq n$. \\
\end{enumerate}
\noindent Then the wonderful compactification $Y_\mathcal G$  of $\mathcal G$ is equal to the iterated blowup $Bl_{\mathcal{G}}Y$ of $Y$ at $\mathcal{G}$. Moreover, under this identification, each divisor $D_{S_i}$ in $Y_\mathcal G\setminus Y^{\circ}$ is the iterated dominant transform of the variety $S_i$ under $ Bl_{\mathcal{G}}Y\rightarrow Y$.\\
\item Let $\mathcal{I}_i$ be the ideal sheaf of $S_i$. Then, the wonderful compactification  $Y_\mathcal{G}$ is isomorphic to the blowup of $Y$ with respect to the ideal sheaf $\mathcal{I}_1\mathcal{I}_2\dots \mathcal{I}_n$.
\end{enumerate}
\end{theorem}
 \begin{remark}\label{partialtotal} Let $Y$ be a nonsingular variety and let $\mathcal{G}$ be a building set  of an arrangement of subvarieties of $Y$. Suppose also that $\mathcal{G}$ is given a partial order with the property that any total order extending it satisfies either of the hypotheses of Theorem \ref{thmLi}(2). Then, Theorem \ref{thmLi}(2) asserts that, for any such extension, the corresponding iterated blowup of $Y$ at $\mathcal{G}$ gives the same variety up to isomorphism, that is $Y_\mathcal{G}$. We can therefore define the iterated blowup of $Y$ at such $\mathcal{G}$ without ambiguity.
 \end{remark}
 \begin{definition}\label{defpart} Let $Y$ be a nonsingular variety and $\mathcal{G}$ be a set of subvarieties of $Y$ which is a building set. Assume that $(<)$ is a partial order on $\mathcal{G}$ with the property that any total order that extends $(<)$ satisfies either of the hypotheses of Theorem \ref{thmLi}(2). We define the \textit{iterated blowup $Bl_{\mathcal{G}} Y$ of $Y$ at }$\mathcal{G}$ in the order $(<)$ to be the iterated blowup of $Y$ at $\mathcal{G}$ in any total order that extends $(<)$.
 \end{definition}
 
\begin{definition}An \textit{ascending dimension order} on a set $\mathcal{G}$ of subvarieties of a variety $Y$ is a partial order $(<)$ on $\mathcal{G}$ with the property that for any $T,S\in \mathcal{G}$, we have $T<S$ if and only if $\dim T<\dim S$. 
\end{definition}

 \vspace{0.1in}
\subsection{Weighted Compactifications of Configuration Spaces}\label{FM}
We recall some results from \cite{routis2014weighted}. Let $X$ be a nonsingular variety of dimension $d$ over an algebraically closed field $k$ and $\mathcal{A}:=\{a_1,a_2, \dots,a_n\} $ be an ordered set of rational numbers in $ \mathcal D^{FM}_{d,n}$.
 Also, let 
\begin{equation*}\label{diag}
\mathcal{K}_{\mathcal{A}}:= \{ \Delta_I \subset X^n | I \subset N \,\text{and}\, \sum \limits _{i\in I} a_i >1\}
,  \,\text{where}\,\,\,
\Delta_I:=
\{ (x_1, \ldots, x_n) \in X^n | x_i=x_j, \; \text{for all}\,\, i,j \in I \}
\end{equation*}
 and let us list its elements in ascending dimension (hence inclusion preserving) order . The above set is shown in [ibid.] to be a building set. The work [ibid.] is concerned with the study of a natural compactification of the configuration space $X^n\setminus \bigcup\limits_ {\Delta_I \in \mathcal{K}_{\mathcal{A}}} \Delta_I$, i.e. the parameter space of $n$ labeled points in $X$ carrying weights $a_i$ subject to the following condition:\\
\renewcommand{\labelitemi}{$\bullet$}
\begin{itemize}
\item for any set of labels $I \subset N$ of coincident points we have $\sum \limits _{i\in I} a_i \leq 1\ $.
\end {itemize}
\begin{definition}The weighted compactification $\wmod$ of $X^n\setminus \bigcup\limits_ {\Delta_I \in \mathcal{K}_{\mathcal{A}}} \Delta_I$ is the wonderful compactification of the building set $\mathcal{K}_{\mathcal{A}}$. \end{definition}
 We have the following result (\cite[Theorems 2 and 3]{routis2014weighted}):

\begin{theorem}\label{wcomp} For any set of admissible weights $\mathcal A\in \mathcal D^{FM}_{d,n}$:
\begin{enumerate}
\item {$\wmod$ is a nonsingular variety. Further, if $\sum \limits _{i\in N} a_i >1$, the boundary $X_\mathcal{A}[n] \setminus (X^n \setminus \bigcup \limits_ {\Delta_I \in \mathcal{K}_{\mathcal{A}}} \Delta_I) $ is the union of $|\mathcal{K}_{\mathcal{A}}|$ smooth irreducible divisors $D_I,$ where $I\subset N $ and $ \sum \limits _{i\in I} a_i >1\ $.}

\item Any set of boundary divisors $D_I$ intersects transversally. 

 \item{$X_\mathcal{A}[n]$ is the iterated blowup of $X^n$ at $\mathcal{K}_{\mathcal{A}}$} in ascending dimension order. Moreover, each divisor $D_I$ is the iterated dominant transform of $\Delta_I$ in $\wmod$.  
 \item There exists a `universal' family  $\phi_\mathcal{A}: X_\mathcal{A}[n]^+\rightarrow X_\mathcal{A}[n]$ equipped with $n$ sections 
 $\sigma_i : X_\mathcal{A}[n]\rightarrow X_\mathcal{A}[n]^+, i=1,\dots n$ whose images lie in the relative smooth locus of  $\phi_\mathcal{A}$. The morphism $\phi_\mathcal{A}$ is a flat morphism between nonsingular varieties, whose fibers are the $n$ pointed $\mathcal{A}$ stable degenerations of $X$ described in Section \ref{wdeg}.
 \end{enumerate}
\end{theorem} 

Now, with notation as above, let $X=\bP^d$ . By Theorem \ref{wcomp}(3), $\bP^d_{\mathcal A}[n]$ is the iterated blowup of $(\bP^d)^n$ at the iterated strict transforms of the diagonals $\Delta _I \in \mathcal{K}_{\mathcal{A}}$ in (any) ascending dimension order (see also Remark \ref{partialtotal}). Let $\pi_{\mathcal A}:\bP^d_{\mathcal A}[n] \to (\bP^d)^n$ be the resulting morphism. Further, consider $SL_{d+1}$ with its usual action on $\bP^d$ and let $SL_{d+1}$ act diagonally on $(\bP^d)^n $. We will use the following Lemma in the construction of $\wpdn$:
\begin{lemma}\label{equiv} 
\begin{enumerate}
\item There is a lift of the diagonal action of $SL_{d+1}$ on $(\bP^d)^n$ to $\bP^d_{\mathcal A}[n]$ and $\bP^d_{\mathcal A}[n]^+$, so that $\pi_{\mathcal A}:\bP^d_{\mathcal A}[n] \to (\bP^d)^n$ and $\phi_\mathcal{A}:\bP^d_{\mathcal A}[n]^+ \rightarrow \bP^d_{\mathcal A}[n] $ (cf Theorem \ref{wcomp}(4)), as well as the sections $\sigma_i:\bP^d_{\mathcal A}[n]\rightarrow\bP^d_{\mathcal A}[n]^+ ,\, i=1,\dots,n$ of $\phi_{\mathcal A}$ become $SL_{d+1}$-equivariant morphisms.
\item The morphism $\phi_\mathcal{A}:\bP^d_{\mathcal A}[n]^+ \rightarrow \bP^d_{\mathcal A}[n] $ is projective.
\end{enumerate}
\end{lemma}

\begin{proof}
(1) $\bP^d_{\mathcal A}[n]$ is a sequence of blowups at iterated strict transforms of diagonals of $(\bP^d)^n$, hence a sequence of blowups at $SL_{d+1}$-invariant centers. Therefore the diagonal action of $SL_{d+1}$ on $(\bP^d)^n$ lifts to $\bP^d_{\mathcal A}[n]$ so that $\pi_{\mathcal A}$ becomes $SL_{d+1}$ equivariant. Also, by the construction in \cite[Section 3]{routis2014weighted}, $\bP^d_{\mathcal A}[n]^{+}$ is a sequence of blowups of $\bP^d_{\mathcal A}[n]\times \PP^d$ at iterated dominant transforms of certain diagonals of $(\PP^d)^n\times \PP^d$ (via $\pi_{\mathcal A}\times id$), hence an iterated blowup at $SL_{d+1}$-invariant centers. Therefore the morphism $\bP^d_{\mathcal A}[n]^{+}\rightarrow \bP^d_{\mathcal A}[n] \times \PP^d$ is $SL_{d+1}$-equivariant. By construction [ibid.], $\phi_\mathcal{A}$ is the composition of the morphism $\bP^d_{\mathcal A}[n]^{+}\rightarrow \bP^d_{\mathcal A}[n] \times \PP^d$ with the projection $\bP^d_{\mathcal A}[n] \times \PP^d \rightarrow \bP^d_{\mathcal A}[n] $, which is $SL_{d+1}$-equivariant, therefore $\phi_{\mathcal{A}}$ is equivariant. Finally, the sections $\sigma_i$ are obtained by successively blowing up strict transforms of $SL_{d+1}$-invariant subvarieties of $\bP^d_{\mathcal A}[n] \times \bP^d$ along $SL_{d+1}$-invariant centers. Indeed, let $S_i$, resp. $D_{I+}$, be the graph of the composite morphism $\bP^d_{\mathcal A}[n]\rightarrow (\PP^d)^n\xrightarrow{q_i} \bP^d$, resp $D_I\hookrightarrow \bP^d_{\mathcal A}[n]\rightarrow (\PP^d)^n\xrightarrow{q_i} \bP^d$, where $q_i$ is the projection to the $i$-th factor and $D_I$ is the divisor corresponding to $I\subset N$ (see Theorem \ref{wcomp}). Then each $\sigma_i$ is the iterated blowup of $S_i$ at $D_{I+}$ ([ibid.]). \\
(2) As noted above, $\phi_{\mathcal A}$ is the composition of a sequence of blowups $\bP^d_{\mathcal A}[n]^{+}\rightarrow \bP^d_{\mathcal A}[n] \times \PP^d$ with the projection  $\bP^d_{\mathcal A}[n] \times \PP^d \rightarrow \bP^d_{\mathcal A}[n]$, i.e. a composition of projective morphisms, hence projective. 
\end{proof}

\subsection{Relative GIT and blowing up} \label{relativeGIT}
Let us first recall some results of the theory known as relative GIT, developed by Hu \cite{hu1996relative}  (see also  
\cite{reichstein1989stability}).

\begin{theorem}\cite[Thm 3.11 and Thm 3.13]{hu1996relative}
\label{GITHu}
Let $\pi: Y \to Z$  be a $G$-equivariant projective morphism between two
 (possibly singular) quasi-projective varieties. Given any linearized
ample line bundle $L$ on $Z$ such that 
the GIT stable locus of $Z$ with respect to $L$ is equal to its strictly semistable locus, that is
$$
Z^{ss}(L) =Z^s(L),
$$ 
 choose a relatively ample linearized line bundle
$M$ on $Y$ . Then 
\begin{enumerate}
\item there exists $n_0$ such that for any $n \geq n_0$, we have
$$
Y^{ss}\left( \pi^*L^n \otimes M\right )
= Y^{s}\left( \pi^*L^n \otimes M\right )
=\pi^{-1}\left(   Z^{s}(L  )  \right).
$$
\item 
Let $n \geq n_0$ and set $\tilde L:=\pi^*L^n \otimes M $. Then 
there is a projective morphism 
\begin{align*}
\hat \pi:
Y / \! \!/_{\tilde L} G
\rightarrow
  Z/ \! \! /_L  G.
\end{align*}
\item For any $z \in Z^s(L)$  with stabilizer $G_z$, we have
$$
\hat \pi^{-1} \left( [G \cdot z ] \right) \cong \pi^{-1}(z) / G_z.
$$ 
\item If $\pi$ is a fibration and $G$  acts freely on $Z^s(L)$, then $\hat  \pi$ is  also a fibration with the same fibers as $\pi$.
\end{enumerate}
\end{theorem}
\if[\begin{definition}\label{buildmore}Let $G$ be a reductive algebraic group acting on a nonsingular projective variety $Y$. Also, let $L$ be a $G$-linearized ample line bundle on $Y$ and $Y^s(L)$ be the stable locus with respect to $L$. For any subvariety $T\subset Y$, we denote by $\hat {T}$ the schematic image of $T\cap Y^s$ in the geometric quotient $Y^s/\! \! /G$. For any set $\mathcal G$ that consists of G-invariant subvarieties of $Y$, we define the set of subvarieties of $Y^s$ $${\mathcal G}^s:=\{T \cap Y^s(L) |\, T \in \mathcal{G}\, \text{and}\,\,  T \cap Y^s(L)\neq\emptyset \}$$ and the set of subvarieties of $Y^s/\! \! /G$ $$\hat{\mathcal{G}}:=\{\hat T |\, T \in \mathcal{G}\, \text{and}\,\,  T \cap Y^s\neq\emptyset \}$$
\end{definition}\fi

The following  result shows that blowing up at a building set is compatible with forming GIT quotients under certain hypotheses. 

\begin{proposition}\label{WonGIT} Let $G$ be a reductive algebraic group acting on a nonsingular projective variety $Y$. Also, let $L$ be a $G$-linearized ample line bundle on $Y$ such that $Y^{s}(L) =Y^{ss}(L)$ and $G$ acts with trivial stabilizers on $Y^s=Y^s(L)$. Further, let $\mathcal G$ be a building set that consists of G-invariant subvarieties of $Y$. 
\begin{enumerate}
\item The variety $Y^s/\! \! /G$ is smooth. Also, the set of subvarieties of $Y^s$, ${\mathcal G}^s:=\{T \cap Y^s |\, T \in \mathcal{G}\, \text{and}\,\,  T \cap Y^s\neq\emptyset \}$, as well as the set $\hat{\mathcal{G}}$ of their scheme theoretic images in $Y^s/\! \! /G$ via $Y^s\rightarrow  Y^s/\! \! /G$ are building sets. 
\item Let ${\mathcal G}$ be given an order compatible with inclusion relations. Let $p: Bl_{{\mathcal G}}Y\rightarrow Y $ be the iterated blowup of $Y$ at $\mathcal G$. Also, let $\widetilde{L_d}:= p^*(L^d)\otimes \mathcal{O}(-E)$, where $E$ is the sum of the exceptional divisors, i.e. the iterated dominant transforms of the $T\in \mathcal{G}$ in $Bl_{{\mathcal G}}Y$.
Then, for sufficiently large $d$, $\widetilde{L_d}$ is a very ample line bundle on $Bl_{{\mathcal G}}Y$ over $Y$ and admits a linearization such that $(Bl_{{\mathcal G}}Y)^{ss}(\widetilde{L_d})=(Bl_{{\mathcal G}}Y)^s(\widetilde{L_d})=p^{-1}(Y^s(L))=Bl_{{\mathcal G}^s}Y^{s}$. Moreover,
$p$ descends to a morphism
$$
\hat p: (Bl_{{\mathcal G}}Y)^s(\widetilde{L_d})/ \! \! /G \to Y^{s}/\! \! /G
$$ 
which is the iterated blowup of $Y^{s}/\! \! /G$ at $\hat{\mathcal{G}}$, where $\mathcal{G}^s$ and $\hat{\mathcal{G}}$ are given the inclusion preserving order induced from ${\mathcal G}$.
\item Let ${\mathcal G}$ be given an order compatible with inclusion relations. Then, for all sufficiently large $d$, the variety $Bl_{{\mathcal G}}Y/ \! \! /_{\widetilde{L^d}}G $ is the wonderful compactification of the building set $\hat{\mathcal{G}}$. 
\end{enumerate}
\end{proposition}
\begin{proof}
We modify the arguments in \cite[Section 7]{hu2003compactification}. For part (1), note that all defining properties of a building set and its induced arrangement (Definition \ref{defbuild}) are Zariski local, so ${\mathcal G}^s$ is readily seen to be a building set. We now show that $\hat{\mathcal{G}}$ is a building set. For any $T\subset Y$, we denote by $T^s$ the restriction $T\cap Y^s$ and by $\hat {T}$ the schematic image of $T\cap Y^s$ in the geometric quotient $Y^s/\! \! /G$. Further, let $\mathcal{S}$ be the arrangement induced by $\mathcal{G}$. Also, let $\hat{\mathcal{S}}$ be the set of all possible nonempty intersections of collections of subvarieties from $\hat{\mathcal G}$. \\

Next we claim that the variety $Y^s/\! \! /G$ is smooth and that $\hat{\mathcal{S}}$ is an arrangement of subvarieties of $Y^s/\! \! /G$.
 To prove this, note that any element of $\hat{\mathcal{S}}$ can be written as $\hat S$, for some $S\in \mathcal{S}$. Also, it is enough to work Zariski locally. Let $x$ be a point in $Y^s$; then, since $G$ acts with trivial stabilizers on $Y^s$, by Luna's \'etale slice Theorem, there exists a locally closed smooth subvariety $W_x$ of $Y^s$ containing $x$ and an open $G$-invariant subvariety $U_x \subset Y^s$ containing $x$ such that the morphism 
$$ G \times W_x \rightarrow U_x$$ is strongly \'etale and $G\cdot W_x=U_x$. Let $S_i,S_j\in \mathcal{S}$; by pulling back the above morphism via $U_x\cap S_i\rightarrow U_x$, we obtain \'etale morphisms 
\begin{align}\label{mor}
G \times (W_x\cap S_i) \rightarrow U_x\cap S_i
\end{align} 
Since $\mathcal{S}$ is an arrangement of subvarieties of $Y$, $S_i$ and $S_j$ intersect cleanly, hence their restrictions to $U_x$ must also intersect cleanly. Since morphism (\ref{mor}) is \'etale, it induces an isomorphism on tangent spaces, therefore $(W_x\cap S_i)$ and $(W_x\cap S_j)$ are smooth and intersect cleanly as well. Now, the (strongly \'etale) morphism  $G \times W_x \rightarrow U_x$ induces an \'etale surjective morphism $W_x\rightarrow U_x/\! \! /G$. Since $x$ is arbitrary, we deduce that $Y^s/\! \! /G$ is smooth. Further, by pullback, we obtain \'etale surjective morphisms
$$
(W_x\cap S_i)\rightarrow (U_x\cap S_i)/\! \! /G= \hat S_i \cap (U_x/\! \! /G)\,\,
\,\text{and}\,\,\,(W_x\cap S_j)\rightarrow (U_x\cap S_j)/\! \! /G=\hat S_j \cap( U_x/\! \! /G)
$$
 which take $((W_x\cap S_i) \cap (W_x\cap S_j))$ to $(U_x\cap S_i \cap S_j)/\! \! /G=\hat S_i \cap \hat S_j \cap (U_x/\! \! /G)$. Therefore, each $\hat S_i$ is smooth and $\hat S_i$ and $\hat S_j$ intersect cleanly. 
\\
 \indent Now let $S\in \mathcal{S}$. To finish the proof of (1), it remains to show that the minimal elements of $\hat{\mathcal{G}}$ that contain $ \hat S$ intersect transversally and that their intersection is equal to $\hat{S}$. Assume that these minimal elements are $\hat S_{i_1}, \hat S_{i_2},\dots \hat S_{i_m}$, where $S_{i_j}\in\mathcal{G}$ for $ j=1,\dots ,m$. Then we see that  $S_{i_j}, j=1,\dots ,m$ are the minimal elements of $\mathcal G$ that contain $S$. Indeed, if $\hat S_{i_j}\supset \hat S' \supset \hat S$ for some $S'\in \mathcal{G}$, then $S_{i_j}\cap Y^s \supset S' \cap Y^s \supset S \cap Y^s$, therefore $S_{i_j}=\overline {S_{i_j}\cap Y^s} \supset \overline {S' \cap Y^s}=S'\supset \overline {S\cap Y^s}=S$. Consequently, by the definition of a building set, the $S_{i_j},\, j=1,\dots ,m$, intersect transversally. Now we may repeat the argument of the proof of the Claim in the previous paragraph to deduce that their images $\hat S_{i_j}$ in $Y^s /\! \! /G$ intersect tranversally as well. Clearly their intersection is equal to $\hat{S}$.\\

Next, we show part (2); Let $k\in \ZZ$ such that $0\leq k\leq |\mathcal{G}|$. We denote by $\mathcal{G}_k$ the subset of the first $k$ elements of $\mathcal{G}$ (here $\mathcal{G}_0=\emptyset$) and let $\mathcal{G}_k^s:=\{T\cap Y^s| T\in \mathcal{G}_k\}$ with the induced order from $\mathcal{G}_k$. Without loss of generality let us assume that $T^s=T\cap Y^s$ is nonempty for all $T\in\mathcal{G}$. Further, let $p_k: Bl_{{\mathcal G}_k}Y\rightarrow Y$  be the natural blowup morphism and set $L_{k,d}:= p_k^*(L^d)\otimes \mathcal{O}(-\sum \limits _{i=1}^k E_{i,k})$, where $E_{i,k}$ are the exceptional divisors of $Bl_{{\mathcal G}_k}Y$, i.e. the iterated dominant transforms of the $T_i,\,i=1,\dots,k$ (in particular $L_{0,d}=L^d$). Let $\hat{T}$ be the images of the $T^s$ in $Y^{s} / \! \!/G$. We prove the following Claim, from which the proof of the Proposition follows. \\

\indent\textit{Claim}. Let $k$ be such that $0\leq k\leq |\mathcal{G}|$.
\begin{enumerate}[label=\roman*.]
 \item For sufficiently large $d$, the line bundle $L_{k,d}$ admits a linearization such that $(Bl_{{\mathcal G}_k}Y)^s(L_{k,d})=(Bl_{{\mathcal G}_k}Y)^{ss}(L_{k,d})= p_k^{-1}(Y^s(L))=Bl_{{\mathcal G}_k^s}Y^{s}$ and $Bl_{{\mathcal G}_k}Y / \! \!/_{ L_{k,d}} G =Bl_{{\mathcal G}_k^s}Y^{s}/ \! \!/ G $ is the iterated blowup, which we denote by $\hat{p}_k$, of $Y^s/ \! \!/ G $ at  \{$\hat T_1,\dots \hat T_k\}$.

\item There is a commutative diagram
 \begin{displaymath}
    \xymatrix{ 
Bl_{{\mathcal G}_k^s}Y^{s} \ar[r]^-{ q_k} \ar[d]_{p_k}
& Bl_{{\mathcal G}_k^s}Y^{s}
/ \! \!/ G  
  \ar[d]^{\hat p_k} \\
Y^s  \ar[r]_{q_0}  &  Y^s
/ \! \!/ G }
\end{displaymath}
such that for every $l$ with $|\mathcal{G}|\geq l>k$  the iterated strict transform
 $(\hat T_{l})^{(k)}$ of $\hat T_{l}$ in $Bl_{{\mathcal G}_k^s}Y^{s}/ \! \!/ G $ is the image of the iterated strict transform $(T_{l}^s)^{(k)}$ of $T_{l}^s$
 in $Bl_{{\mathcal G}_k^s}Y^{s}$ via $q_k$. 
\end{enumerate}
 \indent \textit{Proof of Claim}.  The line bundle $p_k^*(L^d)\otimes \mathcal{O}(-\sum \limits _{i=1}^k E_{i,k})$ is well known to be relatively very ample for large $d$. Therefore, by using Theorem \ref{GITHu}(1)- after twisting by a large enough power of  $p_k^*(L)$ if necessary- we can choose $d_0$ such that that $(Bl_{{\mathcal G}_k}Y)^s(L_{k,d})=(Bl_{{\mathcal G}_k}Y)^{ss}(L_{k,d})= p_k^{-1}(Y^s(L))$ \textit{for all} $d\geq d_0$ and $k=0,\dots,|\mathcal{G}|$. Since blowing up commutes with restricting to Zariski open subsets, the pullback $p_k^{-1}(Y^s(L))$ of $p_k:Bl_{{\mathcal G}_k}Y\rightarrow Y$ via $Y^s\rightarrow Y$ is  $Bl_{{\mathcal G}_k^s}Y^{s}$. Therefore,
 \begin{align}\label{nice}
 (Bl_{{\mathcal G}_k}Y)^s(L_{k,d})=(Bl_{{\mathcal G}_k}Y)^{ss}(L_{k,d})= p_k^{-1}(Y^s(L))=Bl_{{\mathcal G}_k^s}Y^{s}
 \end{align}
\textit{for all} $d\geq d_0$ and $k=0,\dots,|\mathcal{G}|$. From now on, we take $d\geq d_0$, so that (\ref{nice}) holds. We proceed by induction on $k$. For $k=0$ statements (i) and (ii) follow immediately. Assume the above statements are true for some $k\geq0$.  Now, let $p_{k+1,k}:Bl_{{\mathcal G}_{k+1}}Y\rightarrow Bl_{{\mathcal G}_{k}}Y$ be the blowup of $Bl_{{\mathcal G}_{k}}Y$ at the iterated strict transform $T_{k+1}^{(k)}$ of $T_{k+1}$ (see Definition \ref{iterated}) and let $M_{k,d}^{[m]}:=p_{k+1,k}^* (L_{k,d}^m)\otimes \mathcal{O}(-E_{k+1,k+1})$. By \cite[Lemma 3.11]{kirwan1985partial}, it holds that, for infinitely many values of $m>0$, the variety\\
\begin{align}\label{kirwan}
(Bl_{{\mathcal G}_{k+1}}Y)^s(M_{k,d}^{[m]})
/ \! \!/ G  = \left(Bl_{T_{k+1}^{(k)}}(Bl_{{\mathcal G}_k}Y)\right)^s(M_{k,d}^{[m]}) / \! \!/ G  
\end{align}
is the blowup of $Bl_{{\mathcal G}_k^s}Y^{s} / \! \!/ G=(Bl_{{\mathcal G}_k}Y)^s(L_{k,d}) / \! \!/ G$ at the image of $(T_{k+1}^{(k)})^s(L_{k,d})=(T_{k+1}^s)^{(k)}$ via $q_k$, which, by the inductive hypothesis, is equal to the iterated strict transform of $\hat T_{k+1}$ under $\hat p_k$. Therefore, since we have assumed that part (i) is true for $k$, part (i) for $k+1$ will follow if we find $m$ large satisfying (\ref{kirwan}) such that $$(Bl_{{\mathcal G}_{k+1}}Y)^s(M_{k,d}^{[m]})=  (Bl_{{\mathcal G}_{k+1}}Y)^s(L_{k+1,d}).$$ 

But since $M_{k,d}^{[m]}$ is relatively very ample for all large $m$, Theorem \ref{GITHu}(1) gives $(Bl_{{\mathcal G}_{k+1}}Y)^s(M_{k,d}^{[m]})=p_{k+1,k}^{-1}((Bl_{{\mathcal G}_{k}}Y)^s(L_{k,d}))$, which is equal to $p_{k+1,k}^{-1}(p_k^{-1}(Y^s(L)))=p_{k+1}^{-1}(Y^s(L))=(Bl_{{\mathcal G}_{k+1}}Y)^s(L_{k+1,d})$ by (\ref{nice}) for $k+1$. Hence part (i) for $k+1$ is proven and we also obtain a commutative diagram
\begin{displaymath}
    \xymatrix{ 
Bl_{{\mathcal G}_{k+1}^s}Y^{s} \ar[r]^-{ q_{k+1}} \ar[d]_{p_{k+1,k}}
& Bl_{{\mathcal G}_{k+1}^s}Y^{s}
/ \! \!/ G  
  \ar[d]^{\hat p_{k+1,k}}  
\\
Bl_{{\mathcal G}_k^s}Y^{s}  \ar[r]_-{q_k}  &  Bl_{{\mathcal G}_k^s}Y^{s}
/ \! \!/ G  
} 
\end{displaymath}
 
 To prove (ii) for $k+1$, it remains to show that for $l>k+1$ the strict transform of $(\hat T_{l})^{(k)}$ under $\hat p_{k+1,k}$ is equal to the image of $(T_{l}^s)^{(k+1)}$ via $q_{k+1}$. 
Now the strict transform of $(\hat T_{l})^{(k)}$ under $\hat p_{k+1,k}$ is, by definition, equal to the blowup of $(\hat T_{l})^{(k)}$ at its intersection with the center 
$(\hat T_{l})^{(k)}\cap (\hat T_{k+1})^{(k)}$, which is nonsingular by Proposition \ref{buildpullback}. By part (ii) for $k$, the blowup of $(\hat T_{l})^{(k)}$ at $(\hat T_{l})^{(k)}\cap (\hat T_{k+1})^{(k)}$ is equal to the blowup of $q_k((T_{l}^s)^{(k)})$ at $q_k((T_{l}^s)^{(k)}\cap (T_{k+1}^s)^{(k)} )$, which in turn, by \cite[Lemma 3.11]{kirwan1985partial} and Theorem \ref{GITHu}(1), is equal to $(Bl_{(T_{l}^s)^{(k)}\cap (T_{k+1}^s)^{(k)} } (T_{l}^s)^{(k)})^s / \! \!/ G=((T_{l}^s)^{(k+1)})^s / \! \!/ G$, so we are done. \textit{End of Proof of Claim.}\\

\noindent(3) The proof follows immediately by Theorem  \ref{thmLi}(2) and Proposition \ref{WonGIT}.
\end{proof}
\vspace{0.1in}


 \if[For (2) note that blowing up is compatible with restricting to an open set.  Therefore, we have the commuting diagram
\begin{displaymath}
    \xymatrix{ Bl_{\mathcal{G}^s} Y^s \ar[r] \ar[d] & Bl_{\mathcal{G}} Y  \ar[d]^{\bar{p}} \\
               Y^s  \ar[r] & Y }
\end{displaymath}
where $(\bar{p})^{-1}(Y^s)= Bl_{\mathcal{G}^s} Y^s$. By Theorem \ref{GITHu}, we have
$(\bar{p})^{-1}(Y^s) = (Bl_{\mathcal{G}} Y)^s$. Consequently,  $Bl_{\mathcal{G}^s} Y^s = (Bl_{\mathcal{G}} Y)^s$  and
$$
 Bl_{\mathcal{G}^s} Y^s/ \! \! /_{\widetilde{L_d}} G  \cong Bl_{\mathcal{G}} Y / \! \! /_{\bar {L}_d} G
$$]\fi

\section{The 
compactifications $\overline P_{d,n}^{\mathcal A}$ and $T_{d,n}^{\mathcal A}$ (Theorems \ref{mainpdn} and \ref{maintdn}). }\label{proof}

In this section, we construct our compactifications, their respective universal families and we describe their boundary.

\subsection{Points in projective space} \label{conspdn}  Our first step is to give a model of
$\overline P_{d,n}^{\mathcal A}$ isomorphic to a product of projective spaces.  
\begin{definition}Let $n,d$ be positive integers such that $n> d+2$. Let $\epsilon=\frac{1}{n-d}$ and  $\hat \epsilon=\frac{1}{(d+1)(n-d)}$.
 We define $L_{d,n}$ to be 
the fractional line bundle $\mathcal{O}(w_1,w_2,\dots,w_n)$ on $(\bP^d)^n$, 
where $$w_1=\ldots =w_d=1-\hat \epsilon, 
w_{d+1}=1-(n-(d+1))\epsilon +d \hat \epsilon\,\, \text{and}\,\, w_{d+2}=\ldots = w_n=\epsilon.$$ 
\end{definition}
$L_{d,n}$ is an ample line bundle that admits a canonical $SL_{d+1}$ linearization (\cite[Chapter 11]{dolgachev2003lectures}). \\
\indent The following Lemma and its proof  was first communicated to the first author by V. Alexeev.
\begin{lemma}\label{GITpoints} 
There exist an isomorphism
\begin{equation}\label{GITiso}(\bP^d)^n /\!\!/_{L_{d,n}} SL_{d+1}\iso \left( \bP^{n-d-2} \right)^{d}\end{equation}
Futhermore, 
\begin{enumerate}
\item there is no strictly semistable locus in $(\bP^d)^n$ with respect to $L_{d,n}$ and
\item
we can choose coordinates so that the point 
$$
\prod_{k=0}^{d-1}
[b^k_{d+2} : \ldots: b^k_{n}]  \in \left( \bP^{n-d-2} \right)^d
$$ 
parametrizes the equivalence class of $n$ points induced by:
\begin{align*}
p_1 = [1:\ldots:0], \quad p_2=[0:1: \ldots 0], \quad p_{d+1} = [0: \ldots: 1] \,\,\,
\text {and} \,\,p_i=[b^0_i : \ldots :b^{d-1}_i:1] 
\end{align*}
 with $ d+2 \leq i \leq n$.
\end{enumerate}
\end{lemma}
\begin{proof}
A configuration of points is GIT semistable (resp. stable) with respect to $L_{d,n}$ if and only  if $\sum_{\{i\in N|p_i \in W\}}w_i \leq (\dim(W)+1)$ (resp. $\sum_{\{i\in N|p_i \in W\}}w_i < (\dim(W)+1))$ for any proper subspace $W \subset \bP^d$ (see \cite[Thm 11.2]{dolgachev2003lectures}). Observe that the sum of the GIT weights of any subset of points is never equal to an integer, so the above inequality is always strict. Therefore, there is no strictly semistable locus.
The above inequality is equivalent to the following conditions:
\begin{enumerate}
\item$p_1, \ldots , p_d, p_{d+1}$ must be in general position.
\item None of the $p_i, \,i\in\{d+2,\dots n\}$ can lie in the linear subspace spanned by $p_1, \ldots , p_d$.  
\item We cannot have $p_{d+1}=\ldots =p_n$.
\item The points $p_i,\, i=d+2,\dots ,n$, cannot all lie on the hyperplane  spanned by \\
$ p_1, \ldots,  p_{k-1},p_{k+1} \ldots p_{d+1} $ simultaneously.
\end{enumerate}
Then, we can fix the configuration of points $\{p_1, \ldots , p_n\}$ to be as in  the statement. Consequently, the automorphism group of the resulting configuration is isomorphic to $\mathbb{G}_m^d$. By our conditions on the weights, the parameter space of each point $p_i$   with $(d+2) \leq i \leq n$  is contained in $\mathbb{A}^d$, because  $p_i$ cannot lie in the  hyperplane $(x_{n+1}=0)$ spanned by the points $\{p_1, \ldots, p_d \}$.  
The only other restriction on the points  $p_i,\, i=d+2,\dots ,n$, is that they cannot all lie on the hyperplane  spanned by $\{ p_1, \ldots  p_{k-1},p_{k+1} \ldots p_{d+1} \}$  at the same time. 
This means that configurations of points  
with 
$$
b_{d+2}^k=\ldots =b_n^k=0
$$ are not allowed as well. The loci parametrizing these last configurations of points are isomorphic to affine spaces of dimension $(d-1)(n-(d+1))$, which we denote by $\mathbb A^{(d-1)(n-(d+1))}_k$ for each $k$ such that $0 \leq k \leq d-1$. Then,  we obtain 
$$
\left(  \prod_{i=d+2}^{n} \mathbb{A}^{d}  \mathbin{\big\backslash}
 \bigcup \limits_{k=0}^{d-1}
\mathbb A^{(d-1)(n-(d+1))}_k  \right)
/ \!\!/\mathbb{G}_m^d
=
\left(( \mathbb{A}^{n-(d+1)} )^{d}
\setminus  \bigcup \limits_{k=0}^{d-1}
( \mathbb{A}^{n-(d+1)} )^{d-1}\right)
/ \!\!/\mathbb{G}_m^d
$$
$$= \left( \mathbb{A}^{n-(d+1)} \setminus \textbf{0}\right)^{d} / \!\!/\mathbb{G}_m^d=
\left( \bP^{n-(d+2)} \right)^d.$$
\end{proof}

\indent 
We now define the open locus in $(\bP^d)^n /\!\!/_{L_{d,n}} SL_{d+1}$ that we want to compactify. Let us fix $\mathcal A\in \mathcal D_{d,n}^P$. Also, let $\mathcal{K}_{\mathcal{A}}= \{ \Delta_I \subset (\bP^d)^n | I \subset N \,\text{and}\, \sum \limits _{i\in I} a_i >1\}$ be the building set associated to the construction of $\bP^d_{\mathcal A}[n]$  (see Section \ref{FM}). Following the notation of Proposition \ref{WonGIT}, for any $\Delta_I\in\mathcal{K}_{\mathcal{A}}$ such that $\Delta_I \cap\left((\bP^d )^n \right)^s\neq \emptyset $, let $\hat{\Delta}_I$ be the scheme theoretic image of $\Delta_I \cap\left((\bP^d )^n \right)^s $ in $(\bP^d)^n /\!\!/_{L_{d,n}} SL_{d+1}$ and let 
\begin{align}
\hat{\mathcal{K}_{\mathcal{A}}}:=\displaystyle\{\hat{\Delta}_I\subset (\bP^d)^n /\!\!/_{L_{d,n}} SL_{d+1}|\,\Delta_I\in\mathcal{K}_{\mathcal{A}}\,\,\text{and}\,\, \Delta_I \cap\left((\bP^d )^n \right)^s\neq \emptyset\}.
\end{align}

\begin{definition} \label{defpdno}  Let $\mathcal A\in \mathcal D_{d,n}^P$. The \textit{weighted moduli space} of $n$ labeled points in $({\PP})^d$ with respect to $\mathcal{A}$ and $L_{d,n}$ is the open subvariety
$$
\wpdno:=\left((\bP^d)^n / \! \! /_{L_{d,n}}
SL_{d+1}\right)\setminus \bigcup \limits_ {\hat{\Delta}_I \in \hat{\mathcal{K}_{\mathcal{A}}}} \hat{\Delta}_I
$$

of $(\bP^d)^n / \! \! /_{L_{d,n}}SL_{d+1}$.

\end{definition}

Recall that, by Theorem \ref{wcomp}, there is a sequence of blowups $\pi_{\mathcal A}:\bP^d_{\mathcal A}[n] \longrightarrow  (\bP^d )^n$ along the iterated strict transforms of the varieties in the building set $\mathcal{K}_{\mathcal{A}}$. Let $E$ be the the sum of the exceptional divisors, i.e. the iterated dominant transforms of the $\Delta_I\in\mathcal{K}_{\mathcal{A}}$ under this blowup. The strictly semistable locus of $({\bP^d})^n$ with respect to $L_{d,n}$ is empty, by Lemma \ref{GITpoints}, so the hypotheses of Proposition \ref{WonGIT}(2) are satisfied. Hence, we can pick $e_o$, such that for all $e\geq e_0$, the line bundle $\tilde{L}_{\mathcal A,e}:=\pi_{\mathcal A}^*(L_{d,n}^{\otimes e}) \otimes \mathcal{O}(-E)$ is very ample and admits a linearization such that
\begin{equation}\label{stablelocus}
(\bP^d_{\mathcal A}[n] )^{ss}(\tilde{L}_{\mathcal A,e})=(\bP^d_{\mathcal A}[n] )^s(\tilde{L}_{\mathcal A,e}) = \pi_{\mathcal A}^{-1}[\left(  (\bP^d )^n \right)^s],
\end{equation}
 where $\left((\bP^d )^n \right)^s=\left((\bP^d )^n \right)^s(L_{d,n})$ is the stable locus induced by $L_{d,n}$. Therefore, the GIT (and actually geometric) quotient $$\bP^d_{\mathcal A}[n] / \! \!/_{\tilde L_{\mathcal A,e}} SL_{d+1}$$ is independent of the choice of $e\geq e_0$. From now on we will fix one such line bundle $\tilde{L}_{\mathcal A,e}$ and we will simply denote it by $\tilde{L}_{\mathcal A}$.

\begin{definition}\label{defpdn}
Let $\mathcal A\in\mathcal D_{d,n}^P$. The \textit{weighted compactification} of $\wpdno$ is the GIT quotient
 $$
\overline P^{\mathcal A}_{d,n}:= \bP^d_{\mathcal A}[n] / \! \!/_{\tilde L_{\mathcal A}} SL_{d+1}.
$$
\end{definition}
\vspace{0.1in}
\begin{remark} 
When $\mathcal{A}=\{a_1,a_2\dots ,a_n\}$ is such that $a_1= \ldots=a_d=1- \hat{\epsilon},\, a_{d+1}=1-(n-(d+1))\epsilon +d \hat \epsilon$ and $a_{d+2}=\ldots=a_n=\epsilon$, where $\epsilon=\frac{1}{n-d}$ and  $\hat \epsilon=\frac{1}{(d+1)(n-d)}$, then $\mathcal{A}$ coincides with the set of GIT weights corresponding to $L_{d,n}$. Then $\wpdn$ is the GIT quotient of Lemma \ref{GITpoints}. Indeed, 

$$
\overline P^{\mathcal A}_{d,n}=(\bP^d_{\mathcal A}[n] )^s(\tilde{L}_{\mathcal A}) / \! \!/SL_{d+1}=\pi_{\mathcal A}^{-1}[\left(  (\bP^d )^n \right)^s] / \! \!/SL_{d+1}=\left(  (\bP^d )^n \right)^s / \! \!/SL_{d+1}.
$$ The last equality follows because  $\left(  (\bP^d )^n \right)^s=\left(  (\bP^d )^n \right)^s(L_{d,n})$ is contained in the open locus $(\bP^d)^n  \setminus \bigcup_{\Delta_I \in \mathcal K_{\mathcal A}} \Delta_I$, hence $\pi_{\mathcal A}^{-1}[\left(  (\bP^d )^n \right)^s]$ is isomorphic to $\left(  (\bP^d )^n \right)^s$. The above equalities still hold if we increase $a_1,\ldots, a_d$ to any number between $1-\hat{\epsilon}$ and $1$, because $\mathcal K_{\mathcal A}$ remains invariant. $\wpdn$ changes only if we increase the weights $a_{d+1},\dots, a_n$. 
\end{remark}
We now describe the loci parametrizing equivalence classes of configurations with coincident points in the GIT quotient of Lemma \ref{GITpoints}.  These loci can also be interpreted in the context of the moduli of hyperplane arrangements (see Section \ref{hyparr}) as part of the loci parametrizing 
configurations of lines with non-log canonical singularities. This last fact was communicated to the first author by V. Alexeev. Recall that the GIT quotient $(\bP^d)^n /\!\!/_{L_{d,n}} SL_{d+1}$ is  isomorphic to $d$ copies of ${\PP}^{n-d-2}$ (Lemma \ref{GITpoints}).  
\begin{definition}\label{hi}
\begin{enumerate} \item Let $[b^k_{d+2} : \ldots: b^k_{n}], \, k=0,1,\dots, d-1,$ be a system of projective coordinates for each copy of ${\PP}^{n-d-2}$ in the product  $\left( \bP^{n-d-2} \right)^d$. Let $I\subsetneq \{d+1,\dots n\}$ such that $ |I| \geq 2$. We define subvarieties $H_I$ of $({\PP}^{n-d-2})^d$ as follows:
  \begin{equation*}
  H_I :=\begin{cases}  
\bigcap_{i,j \in I} V((b^0_i-b^0_j, \ldots ,b^{d-1}_i-b^{d-1}_j))  &,  \  \text{if } d+1 \notin I 
\\
\bigcap_{i\in I \setminus d+1} V((b^0_i,\ldots,b^{d-1}_i))        &,  \   \text{if } d+1 \in I 
\end{cases}
  \end{equation*}
  \item Let $\mathcal{A}=\{a_1,a_2,\dots, a_n\}\in \mathcal D_{d,n}^P$. We define the set $$\mathcal{G}_{\mathcal{A}}:=\{ H_I  \subset ( \bP^{n-d-2})^d  \; | \;I\subsetneq \{d+1,\dots n\}\,\, \text{and}\,\, \sum_{i \in I} a_i > 1 \}.$$
  \end{enumerate}
\end{definition}

\begin{lemma}\label{wallStrata}\begin{enumerate}
\item The subvarieties $H_I$ of $({\PP}^{n-d-2})^d$ parametrize those equivalence classes of configurations of $n$ points $p_1,\dots,p_n$ in $((\bP^d)^n)^s(L_{d,n})$ where all points $ p_{i}$ such that $i \in I $ coincide.  Furthermore, $H_I\iso  ({\PP}^{(n-|I|)-d-1})^d$.\\
\label{buildpdn}
\item Let $\mathcal{A}\in \mathcal D_{d,n}^P$. Then:\begin{itemize}
 \item The set $\hat{ \mathcal{K}_{\mathcal{A}}}$ is a building set.
\item Under the isomorphism (\ref{GITiso}) of $(\bP^d)^n /\!\!/_{L_{d,n}} SL_{d+1}$ with $({\PP}^{n-d-2})^d$, the set  $\hat{ \mathcal{K}_{\mathcal{A}}}$ is identified with the set $\mathcal{G}_{\mathcal{A}}$.
\end{itemize}
\end{enumerate}
\end{lemma}
\begin{proof}
Part (1) is clear. As for part (2), by  Proposition \ref{WonGIT}(1), $\hat{ \mathcal{K}_{\mathcal{A}}}$ is a building set. Also, observe that the equivalence class $[SL_{d+1} \cdot (p_1, \ldots p_n)]$ is  in $\{ H_I  \subset ( \bP^{n-d-2})^d  \; | \;I\subsetneq \{d+1,\dots n\}\,\, \text{and}\,\, \sum_{i \in I} a_i > 1 \}$ if and only if the $n$-tuple $(p_1, \ldots p_n)$ is contained in the stable locus $((\bP^d )^n )^{s}$ and there exists some $I \subsetneq \{d+1,\dots n\}$ such that $ p_i=p_j$ for all $i,j\in I $ and $\sum_{i_k \in I}a_i > 1$. These two conditions are the ones defining  $\mathcal{K}_{\mathcal{A}}^s= \{\Delta_I \cap ((\bP^d )^n )^{s} | \Delta_I \in \mathcal{K}_{\mathcal{A}}\,\,\text{and}\,\, \Delta_I \cap ((\bP^d )^n )^{s}\neq\emptyset\} $.
\end{proof}
\noindent

\begin{corollary}\label{Wonpdn} $\bP^d_{\mathcal A}[n] / \! \!/_{\tilde L_{\mathcal A}}SL_{d+1}$ is the wonderful compactification of $\hat{ \mathcal{K}_{\mathcal{A}}}$.
\end{corollary}
\begin{proof}By Section \ref{FM}, the weighted Fulton-Macpherson space 
$\bP^d_{\mathcal A}[n] $  is the wonderful compactification of the building set $\mathcal{K}_{\mathcal{A}}$. Moreover,  the hypotheses of Proposition \ref{WonGIT} are satisfied by the variety $(\bP^d)^n $ with $SL_{d+1}$-linearized ample line bundle $L_{d,n}$. Indeed, there is no strictly semistable locus induced by $L_{d,n}$ (Lemma \ref{GITpoints}); and since the quotient $(\bP^d)^n /\!\!/_{L_{d,n}} SL_{d+1}$ is smooth (Lemma \ref{GITpoints}), we conclude that $SL_{d+1}$ acts with trivial stabilizers on $(\bP^d)^n$. Therefore, 
by Proposition \ref{WonGIT}(3), the variety $\bP^d_{\mathcal A}[n] / \! \!/_{\tilde L_{\mathcal A}}SL_{d+1}$ is the wonderful compactification of the building set $\hat{ \mathcal{K}_{\mathcal{A}}}$.
\end{proof}
\noindent\textit{Proof of Theorem} \ref{mainpdn}:
~\\
\noindent
\textbf{Part 1:} Follows immediately from Corollary \ref{Wonpdn} and Theorem \ref{thmLi} (1).\\
\textbf{Part 2}:
First we define $\hat \phi_{\mathcal{A}}$ and its sections $\hat\sigma_i$ by descending the analogous morphisms that appear in the weighted Fulton MacPherson construction. Let $\phi_{\mathcal{A}} : \bP^d_{\mathcal A}[n]^{+}\rightarrow \bP^d_{\mathcal A}[n]$ be the `universal' family of the weighted Fulton MacPherson compactification $\bP^d_{\mathcal A}[n]$ (Theorem \ref{wcomp}(4)). We have already seen that the diagonal action of $SL_{d+1}$ on $(\bP^d)^n$ lifts to $\bP^d_{\mathcal A}[n]$ and $\bP^d_{\mathcal A}[n]^{+}$ so that $\phi_{\mathcal{A}}$ becomes equivariant (Lemma \ref{equiv}). Recall by Lemma \ref{equiv}(2) that $\phi_{\mathcal{A}}$ is projective. Now let us choose a relatively ample linearized line bundle $M_\mathcal{A}^+$ for $\phi_{\mathcal{A}}$. Using Theorem \ref{GITHu}(1), we can pick $m_0$ such that for all $m\geq m_0$ the line bundle $L^+_{\mathcal A,m}:=M_\mathcal{A}^+\otimes \phi_{\mathcal{A}}^{*}({{\tilde {L}}_{\mathcal A}}^{\otimes m})$ is ample and admits a linearization such that 
\begin{equation}\label{stablepullback}(\bP^d_{\mathcal A}[n]^{+})^{ss}(L^+_{\mathcal A,m})= (\bP^d_{\mathcal A}[n]^{+})^{s}(L^+_{\mathcal A})= \phi_{\mathcal{A}}^{-1}((\bP^d_{\mathcal A}[n])^s(\tilde L_{\mathcal A})).\end{equation}
 Now, we define $\wpdnu:=\bP^d_{\mathcal A}[n]^{+} / \! \!/_{ L^+_{\mathcal A,m_0}} SL_{d+1}$. The above GIT quotient is independent of the choice of $M_\mathcal{A}^+$, so $\wpdnu$ is well defined. Then, $\phi_{\mathcal{A}}$ descends to a morphism 

$$\hat \phi_{\mathcal{A}}:\wpdnu\rightarrow\wpdn. $$

Moreover, $\phi_{\mathcal{A}}$ is equipped with $n$ sections $\sigma_i:\bP^d_{\mathcal A}[n]\rightarrow\bP^d_{\mathcal A}[n]^{+}$(Theorem  \ref{wcomp}), which are $SL_{d+1}$ equivariant by Lemma \ref{equiv}. Recall that $(\bP^d_{\mathcal A}[n] )^{ss}(\tilde L_{\mathcal A})=(\bP^d_{\mathcal A}[n]) ^s(\tilde L_{\mathcal A})$. Since $ \phi_{\mathcal{A}}^{-1}(\bP^d_{\mathcal A}[n] ^s(\tilde L_{\mathcal A}))=(\bP^d_{\mathcal A}[n]^{+})^{s}(L^+_{\mathcal A})$, the restriction $\sigma_i^s$ of the section $\sigma_i$ to $\bP^d_{\mathcal A}[n] ^{s}(\tilde L_{\mathcal A})$ maps to the subvariety $(\bP^d_{\mathcal A}[n]^{+})^{s}(L^+_{\mathcal A})$ of $\bP^d_{\mathcal A}[n]^{+}$. Let $\phi_{\mathcal{A}}^s$ be the restriction of $\phi_{\mathcal{A}}$ to the stable locus (which is also the pullback of $\phi_{\mathcal{A}}$ via $\bP^d_{\mathcal A}[n] ^{s}\rightarrow \bP^d_{\mathcal A}[n]$). We may now descend $\sigma_i^s$ to obtain sections $\hat\sigma_i:\wpdn\rightarrow\wpdnu,\, i=1,\dots,n$,  that fit in the following commutative diagram 

    \[
   \xymatrix
 {(\bP^d_{\mathcal A}[n]^{+})^{s}\ar[d] ^{\phi_{\mathcal{A}}^s}\ar[r] & (\bP^d_{\mathcal A}[n]^{+})^{s} / \! \!/ SL_{d+1}\ar[d]^{\hat \phi_{\mathcal{A}}}\\
\bP^d_{\mathcal A}[n] ^{s}\ar[r] \ar@/^/[u]^{\sigma_i^s} & (\bP^d_{\mathcal A}[n])^{s} / \! \!/ SL_{d+1}\ar@/^/[u]^{\hat\sigma_i }}\\
  \]

  By Theorem \ref{wcomp}(4), $\sigma_i$ lie in the relative smooth locus of $\phi_{\mathcal{A}}$. By part 1 of Theorem \ref{mainpdn}, the geometric quotient $\wpdn=\bP^d_{\mathcal A}[n] ^{s} /\!\!/SL_{d+1}$ is smooth, so $SL_{d+1}$ acts with trivial stabilizers on $\bP^d_{\mathcal A}[n] ^{s}$. Consequently, by Theorem \ref{GITHu}(3), the fiber of $\hat \phi_{\mathcal{A}}$ over a geometric point $[SL_{d+1}\cdot x]$ in $\bP^d_{\mathcal A}[n] ^{s} /\!\!/SL_{d+1}$ is isomorphic to the fiber of $\phi_{\mathcal{A}}^s$ over any point in the orbit $SL_{d+1}\cdot x$. Therefore, the relative smooth locus of $\phi_{\mathcal{A}}^s$ maps to the relative smooth locus $\hat \phi_{\mathcal{A}}$ via the quotient morphism. Hence, by the commutativity of the above diagram we deduce that the images of $\hat\sigma_i$ lie in the relative smooth locus of $\hat\phi_{\mathcal{A}}$.\\
 \indent  Moreover, since $\bP^d_{\mathcal A}[n]$ and $\bP^d_{\mathcal A}[n]^{+}$ are projective, the GIT quotients $\wpdn$ and $\wpdnu$ are projective, so $\hat \phi_{\mathcal{A}}$ is proper. By Theorem \ref{GITHu} (3) again and the fact that fibers of $\phi_{\mathcal{A}}^s$ are equidimensional, it follows that the fibers of $\hat \phi_{\mathcal{A}}$ are also equidimensional. Further, since $\phi_{\mathcal{A}}^s$ is equivariant and $SL_{d+1}$ acts with trivial stabilizers on $\bP^d_{\mathcal A}[n] ^{s}$, we
 see that $SL_{d+1}$ acts with trivial stabilizers on the smooth variety $(\bP^d_{\mathcal A}[n] ^{+})^{s}$ as well. Therefore $\wpdnu=(\bP^d_{\mathcal A}[n] ^{+})^{s} / \! \!/ SL_{d+1}$ is smooth (for example, by Luna's \'etale slice theorem)  and, since we have shown that $\wpdn$ is also smooth (part 1 of Theorem \ref{mainpdn}), we deduce that $\hat \phi_{\mathcal{A}}$ is flat. It remains to verify part 2(b); Theorem \ref{GITHu} (3) yields that if $x$ is a geometric point in $\bP^d_{\mathcal A}[n] ^{s}$, then $\hat \phi_{\mathcal{A}}^{-1}([SL_{d+1}\cdot x])\iso\phi_{\mathcal{A}}^{-1}(x) $. Then the statement of  part 2(b) is equivalent to the statement that the weighted stable $n$-pointed degeneration $(\phi_{\mathcal{A}}^{-1}(x), \sigma_i^s(x))$ is isomorphic to $(\hat \phi_{\mathcal{A}}^{-1}([SL_{d+1}\cdot x]), \hat\sigma_i^s([SL_{d+1}\cdot x]))$. This follows immediately from the descriptions in Sections \ref{wdeg} and \ref{trees} and Theorem \ref{wcomp}(4).
\qed

\vspace{0.1in}

Now let $\mathcal{G}_{\mathcal{A}}$ be the set defined in Definition \ref{hi}.
\begin{corollary}\label{corpdn}
For every $\mathcal A\in \mathcal D_{d,n}^P$:
\begin{enumerate}
\item $\wpdn$ is a sequence of blowups of $({\PP}^{n-d-2})^d$ at the iterated strict transforms of the varieties in the set $\mathcal{G}_{\mathcal{A}}$ in ascending dimension order.
\item The boundary $\overline P^{\mathcal A}_{d,n}\setminus \wpdno$ is the union of $|\mathcal{G}_{\mathcal{A}}|$ smooth irreducible divisors $E_I$, where each $E_I$ is the iterated dominant transform of the variety $H_I\in\mathcal{G}_{\mathcal{A}}$ in the sequence of blowups $\wpdn\rightarrow({\PP}^{n-d-2})^d$.
\item Any set of boundary divisors  intersects transversally.  
\end{enumerate}
\end{corollary}
\begin{proof}
By Theorem \ref{thmLi}(2) the first statement 
is equivalent to the claim that $\wpdn$ is the wonderful compactification of the set $\mathcal{G}_{\mathcal{A}}$. This last claim follows from Lemma \ref{buildpdn} and 
Corollary \ref{Wonpdn}.  The other statements  follow from Theorem \ref{thmLi}.
\end{proof}

\subsection{Points in affine space}\label{pTdn}
 We first recall the definition of $\wtdn$ and then we give a  
birational model of it, which is isomorphic to a  projective space. \\
Let $\wmod$ be the weighted compactification of the weighted configuration space of $n$ labeled points in a smooth variety $X$ (Section \ref{FM}), where $\mathcal{A}=\{a_1,a_2,\dots a_n\}\in \mathcal{D}_{d,n}^{T}$, that is $0<a_i\leq1$ and $\sum \limits _{i=0}^n a_i>1$. The above inequality implies that $\Delta_N$ belongs to the building set $\mathcal{K}_{\mathcal{A}}$ associated to the construction of $\wmod$. Now let $D_N\subset \wmod$ be the divisor corresponding to $\Delta_N$, that is the iterated dominant transform of $\Delta_N$ under the sequence of blowups $\pi_\mathcal{A}:\wmod\rightarrow X^n$ (Theorem \ref{wcomp}(3)). Let  $\pi_{\mathcal{A}}|_{D_N}:D_N \rightarrow \Delta_N\iso X$ be the restriction of $\pi_\mathcal{A}$ to $D_N$.\\

\definition\cite{ChowHass} \label{defTdn}Let $\mathcal{A}\in \mathcal{D}_{d,n}^{T}$, let $X$ be a smooth variety of dimension $d$ and $x\in X$ a geometric point of $X$. We define $\wtdn:=({\pi_{\mathcal{A}}|_{D_N}})^{-1}(x)$.
$$
\xymatrix{
T^{\mathcal{A}}_{d,n} \ar[r]  \ar[d] & D_N \ar[d]^{\pi_{\mathcal{A}}|_{D_N}} \ar[r] & X_{\mathcal{A}}[n] \ar[d]
\\
x \ar[r] & \Delta_N \cong X \ar[r] & X^n
}
$$
\remark It is shown in [ibid.] that the above definition is independent of $x$ and $X$, as long as $d,n$ and $\mathcal{A}$ are fixed.
\remark \label{utdn}The universal family $\wtdnu\to\wtdn$ is defined in \cite{ChowHass} as the pullback of the universal family $\phi_{\mathcal{A}}:\wmod^+\rightarrow \wmod$ via $\wtdn\rightarrow \wmod$. Its geometric fibers are precisely the objects described in Section \ref{rootedtrees}.\\
$$
\xymatrix{
\wtdnu \ar[r]  \ar[d] &\wmod^+ \ar[d]^{\phi_{\mathcal{A}}}
\\
T^{\mathcal{A}}_{d,n} \ar[r] & \wmod 
}
$$
\vspace{0.1in}

Recall that given a configuration of $n$ points in affine space defined up to translation and homothety,  it is convenient  to think of them as points in $\mathbb{P}^d$ that lie away from a fixed hyperplane $H \subset \mathbb{P}^d$ called the \textit{root} and defined up to the action of the subgroup $G \subset SL_{d+1}$ that fixes the root pointwise.

\begin{lemma}\label{base}
Let $\mathcal{B}:= \left( \frac{1}{n}+\epsilon, \ldots, \frac{1}{n}+\epsilon  \right)\in\mathcal D_{d,n}^T$, where $\epsilon <\frac{1} {n(n-1)}$. Then $T^{\mathcal B}_{d,n} \cong \mathbb{P}^{dn-d-1}$ and there is a choice of coordinates so that the point
$$
[x_{11}: x_{12}:  \ldots : x_{1d} :\ldots : x_{21}:  x_{22} : \ldots x_{2d}:
 \ldots : x_{(n-1)1}:  x_{(n-1)2} : \ldots : x_{(n-1)d}]
\in \mathbb{P}^{dn-d-1}
$$
parametrizes the equivalence class associated to the collection of $n$ points:
\begin{align*}
p_1:=[1:x_{11}: \ldots : x_{1d}],  & &
\ldots  & &
p_{n-1}:=[1:x_{(n-1)1}: x_{(n-1)2}: \ldots:  x_{(n-1)d}],  & &
p_n:= [1:0: \ldots: 0]
\end{align*}
\end{lemma}
\begin{proof}
By Definition \ref{defTdn}, $T^{\mathcal B}_{d,n}$ is  the fiber 
$q_{\mathcal{B}}^{-1}(x)$ of a point in the divisor $D_N$ over $X$  where $D_N$ is the iterated dominant of the small diagonal
$\Delta_N$ along the sequence of blowups $X_{\mathcal B}[n ] \to X^n$. For our choice of weights there is only one blow up involved, hence the dominant transform of $\Delta_N$ in $X_{\mathcal B}[n]$ is the projective bundle $\PP(N_{\Delta_N/X^n})$. Therefore, its fiber over $\Delta_N=X$ is isomorphic to the projective space $\mathbb{P}^{dn-d-1}$. 
To obtain  the coordinates, we describe an alternative and instructive construction. We consider $\PP^d$ with homogeneous coordinates $[x_0:\dots :x_d]$ and take the root $H$ to be $(x_0=0)$. We can choose the location of one of the points, say $p_n$, to be  $[1:0\dots :0] \in (\PP^d\setminus H)=\Aff^d$. 
The location of the other $(n-1)$ points can be anywhere in $\bP^d \setminus H$, but they cannot all overlap with $p_n$ simultaneously. The automorphism group of $\PP^d$ that fixes the hyperplane $H$ pointwise and the point $p_n$ is $\mathbb G_m$. Then, we conclude that the our parameter space is 
\begin{center}
$
\left((\Aff^d)^{n-1}\setminus (0,0,\dots0)\right) / \! \! / \mathbb G_m \iso \PP^{d(n-1)-1}$
\end{center} 
with the coordinates described in the statement.
 \end{proof}
\begin{remark}
It can be shown that  there exist a linearization $L'_{d,n}$ such that 
$\PP^{nd-d-1}$ is a non-reductive GIT quotient of the form 
$\PP^{nd-d-1} \cong (\PP^d)^n / \! \! /_{L_{d,n}'} \mathbb{G}_a^d \rtimes \mathbb{G}_m$. That type of compactification is explored in \cite{NPP}.  Within this context Lemma \ref{base} is very similar to Lemma \ref{GITpoints}. 
\end{remark}
Next, we describe the loci in $T^{\mathcal B}_{d,n}$, where $\mathcal{B}= \left( \frac{1}{n}+\epsilon, \ldots, \frac{1}{n}+\epsilon  \right)\in\mathcal D_{d,n}^T$, such that $\epsilon <\frac{1} {n(n-1)}$, that parametrize configurations with overlapping points. 
 Let $$
[x_{11}: x_{12}:  \ldots : x_{1d} :\ldots : x_{21}:  x_{22} : \ldots x_{2d}:
 \ldots : x_{(n-1)1}:  x_{(n-1)2} : \ldots : x_{(n-1)d}]
$$
be a system of projective coordinates for $\mathbb{P}^{dn-d-1}$. For each $I\subset N$ with $2 \leq |I| \leq (n-1)$, we define a subvariety 
$\delta_I$ of $ \PP^{d(n-1)-1}$ as follows:
  \begin{equation}\label{coincident}
  \delta_I =\begin{cases}  
\bigcap_{i,j \in I}V((x_{i1}-x_{j1},\ldots , x_{id}-x_{jd}))  & | \  \text{if } n\notin I 
\\
\bigcap_{i\in I\setminus n }V(( x_{i1} ,\ldots ,x_{1d}))             & | \   \text{if } n\in I 
\end{cases}
  \end{equation}
   Each $\delta_I$ is isomorphic to $ \PP^{d(n-|I|)-1}$. By Lemma \ref{base}, it follows immediately that $\delta_I$ parametrizes those equivalence classes (with respect to translation and homothety) of configurations of $n$-tuples $(p_1,p_2,\dots p_n)\in(\Aff^d)^n$ of $\mathbb{P}^{dn-d-1}$ where $p_i=p_j$ for all  $i,j\in I$.  

\begin{definition}\label{defset}
Let $\mathcal A=\{a_1,\dots, a_n\} \in \mathcal D^T_{d,n}$. We define the set 
\begin{center}
$
\wbuildtdn:=
\left\{
\delta_I \subset \PP^{d(n-1)-1} | \;
 \sum_{i \in I}a_i >1 
\right\}
$
\end{center}
with partial order $(<)$ determined by the rule $\delta_I<\delta_J$ if and only if $|J|<|I|$.
\end{definition}
\begin{lemma}\label{bset}
The  set $\wbuildtdn$ is a building set.
\end{lemma} 
\begin{proof}
By Lemma \ref{clean}, we see immediately that the set of all possible nonempty intersections of the $\delta_I\in\wbuildtdn$ is an arrangement of subvarieties of $\PP^{d(n-1)-1}$. Consider an arbitrary nonempty intersection $S:=\delta_{I_1}\cap\dots \cap \delta_{I_k} $ of varieties that belong to the set $\wbuildtdn$. To finish the proof, we need to show that the minimal elements of $\wbuildtdn$ containing $S$  intersect transversally and  their intersection is $S$. To see this, observe that the above intersection can be written uniquely as an intersection of the form $\delta_{I'_1}\cap\dots \cap \delta_{I'_m}$, where $m\leq k$, the sets $I'_i$  are pairwise disjoint and every $I'_i,i=1,\dots m$, is a union of $I_j$'s. Further, each of the varieties $\delta_{I'_i}$ belongs to $\wbuildtdn$; indeed, $I'_i$ cannot be the set $\{1,\dots,n\}$ (otherwise S would be empty), and  since $I'_i$ contains some $I_j$, we have
\begin{equation*}
\sum \limits_{k\in I'_i} a_k \geq \sum \limits_{k\in I_j} a_k >1.
\end{equation*} 
\bigskip
Clearly the varieties $\delta_{I'_i},\,i=1,\dots m$, are the minimal elements of $\wbuildtdn$ that contain $S=\delta_{I'_1}\cap\dots \cap \delta_{I'_m}$. Finally, since the indices $I'_i$ are disjoint, we see that the varieties $\delta_{I'_1},\dots, \delta_{I'_m}$ intersect transversally.
\end{proof}
\definition \label{deftdno}Let  $\mathcal{A}\in \mathcal D^T_{d,n}$. The weighted configuration space of $n$ labeled points in $\Aff^d$, up to translation and homothety, with respect to $\mathcal{A}$ is the open subvariety
$$\wtdno:=\bP^{d(n-1)-1}  \setminus \bigcup_{\Delta_I \in \wbuildtdn} \delta_I $$
of $\bP^{d(n-1)-1}$.

 \begin{lemma} \label{tdnwon}
For any $\mathcal{A}\in \mathcal D^T_{d,n}$, $\wtdn$ is isomorphic to the wonderful compactification of $\mathcal H_{\mathcal A}$.
\end{lemma} 
\begin{proof}Let $X$ be a smooth variety of dimension $d$. We extend the partial order of ascending dimension on $\mathcal{K}_{\mathcal{A}}= \{ \Delta_I \subset X^n | I \subset N \,\text{and}\, \sum \limits _{i\in I} a_i >1\}$ to a total order ($\vartriangleleft$). We also extend the partial order of ascending dimension on $\mathcal H_{\mathcal A}$ to a total order ($\blacktriangleleft$) compatible with ($\vartriangleleft$), that is for any $I,J\subsetneq N$, we have $\delta_I\blacktriangleleft\delta_J $ if and only if  $\Delta_I\vartriangleleft \Delta_J$. By Remark \ref{partialtotal}, any sequence of blowups of $X^n$ dictated by an extension of the partial order on $\mathcal{K}_{\mathcal{A}}$  is isomorphic to $\wmod$. For any $Y\subset X^n$ (resp. $Y\subset \PP^{d(n-1)-1}$), we denote by $Y^{(i)}$ the iterated dominant transform of $Y$ in the $i$-th step of the sequence of blowups $\wmod\rightarrow X^n$ (resp. $Bl_{\wbuildtdn}\PP^{d(n-1)-1}\rightarrow \PP^{d(n-1)-1}$); see also Definition \ref{iterated}. 
 The proof of the Lemma follows from part (1) of the following Claim and Theorem \ref{thmLi}(2). \\

\textit{Claim}: For \textit{every} $\Delta_I\subset X^n$, where $I\neq N$, and $i\geq 1$:
\begin{enumerate}
\item The fiber $\Delta_N^{(i)} \times_X x$ is isomorphic to the $(i-1)$-th iterated blowup of $\PP^{d(n-1)-1}$. 
\item The strict transform $(\Delta_N^{(i-1)}\cap \Delta_I^{(i-1)})^{\tilde{}}$\,   of $\Delta_N^{(i-1)}\cap \Delta_I^{(i-1)}$ in $(X^n)^{(i)}$ is equal to $\Delta_N^{(i)}\cap \Delta_I^{(i)}$.
\item The fiber $(\Delta_N^{(i)}\cap \Delta_I^{(i)})\times_X x $ is isomorphic to  $\delta_I ^{(i-1)}$. 
\end{enumerate}
\textit{Proof of Claim}: We proceed by induction on $i$. To prove the above claim for $i=1$, observe that $\Delta_N^{(1)}$ is the exceptional divisor of the first blowup, hence isomorphic to $\PP(N_{\Delta_N/X^n})=\PP\left(T_{X^n}/T_X\right)$ and  $\Delta_I^{(1)}\cap \Delta_N^{(1)}= \PP(N_{\Delta_N/ \Delta_I})=\PP\left(T_{\Delta_I}/T_X\right)$. Further, we have seen that the $\delta_I$ are the images of the diagonals of $(\Aff^d)^{n-1}\setminus\textbf{0}$ labeled by $I$ (i.e. the subvarieties $\{(z_1,z_2,\dots ,z_{n-1})\in(\Aff^d)^{n-1}\setminus \textbf{0}\,\,|z_i\in \Aff^d\,\,\text {and}\,\,z_i=z_j\,\, \text{for all}\,i,j\in I \}$) in the quotient $\left((\Aff^d)^{n-1}\setminus (0,0,\dots0)\right) / \! \! / \mathbb G_m \iso \PP^{d(n-1)-1}$. Therefore the embedding $\PP\left(T_{\Delta_I}/T_X\right) \hookrightarrow\PP\left(T_{X^n}/T_X\right) $ over $X$ pulls back to $\delta_I\hookrightarrow\PP^{d(n-1)-1}$ via $x\rightarrow X$. \\
\indent Now let $\Delta_J^{(i)}$ be the center of the $(i+1)$-th blowup of $X^n$ and assume the Claim is true for some $i\geq1$. We will show the Claim is true for $i+1$. By Proposition \ref{buildpullback} the set $\{\Delta_I^{(i)}|\, \Delta_I \in \mathcal K_{\mathcal{A}}\}$ is a building set (of its induced arrangement of subvarieties of the $i$-th blowup $(X^n)^{(i)}$ of $X^n$) and $\Delta_N^{(i)}\cap \Delta_J^{(i)}$ is a smooth variety. Moreover, by parts (1) and (3) of the Claim for i we see that $\Delta_N^{(i)}\cap \Delta_J^{(i)}$ and $\Delta_N^{(i)}$ have equidimensional fibers over $X$, so they are flat over $X$. Part (1) of the Claim for $i+1$ now follows from Lemma \ref{commute}. \\
 \indent We now show part (2) of the Claim for $i+1$. If $\Delta_I^{(i)}$ is the center $\Delta_J^{(i)}$ of the $i+1$-th blowup, then the Claim is immediate. Otherwise, by part (3) of the Claim for $i$, we see that the fiber of $\Delta_N^{(i)}\cap \Delta_I^{(i)}$ (resp. $\Delta_N^{(i)}\cap \Delta_I^{(i)}\cap \Delta_J^{(i)}$) over any $x\in X$ is isomorphic to $\delta_I ^{(i-1)}$ (resp. $\delta_I ^{(i-1)}\cap \delta_J^{(i-1)}$), so in particular, $\Delta_N^{(i)}\cap \Delta_I^{(i)}$ (resp. $\Delta_N^{(i)}\cap \Delta_I^{(i)}\cap \Delta_J^{(i)}$) is smooth. Now, by applying Proposition \ref{buildpullback} to the building set $\wbuildtdn$, we deduce that $\delta_I ^{(i-1)}$ is smooth and intersects transversally with  $\delta_J^{(i-1)}$ in $(\PP^{d(n-1)-1})^{(i-1)}$.  Therefore, the sum of the codimensions of $\delta_I^{(i-1)}$ and $\delta_J^{(i-1)}$ in $(\PP^{d(n-1)-1})^{(i-1)}$ is equal to the codimension of $\delta_I ^{(i-1)}\cap \delta_J^{(i-1)}$. Since $\Delta_N^{(i)}\rightarrow X$ is flat, by a straightforward dimension count, we see that the intersections of $\Delta_J^{(i)}$ with $\Delta_N^{(i)}$ and $\Delta_N^{(i)}\cap \Delta_I^{(i)}$ as well as the intersection of $\Delta_I^{(i)}$ with $\Delta_N^{(i)}$ in $(X^n)^{(i)}$ are transversal. Moreover, by Proposition \ref{buildpullback}(2), $\Delta_I^{(i)}$ either contains or intersects transversally with $\Delta_J^{(i)}$. Therefore, by using Lemma \ref{basic}(2) and (3), part (2) of the Claim for $i+1$ follows.

   Finally we show part (3) of the Claim for $i+1$. As noted above, the varieties $(\Delta_N^{(i)}\cap \Delta_I^{(i)}\cap \Delta_J^{(i)})$ and $(\Delta_N^{(i)}\cap \Delta_J^{(i)}) $ are smooth and have equidimensional fibers over $X$, hence they are flat over $X$. Lemma \ref{commute} then shows that $(\Delta_N^{(i)}\cap \Delta_I^{(i)})^{\tilde{}} \times_X x = ((\Delta_N^{(i)}\cap \Delta_I^{(i)}) \times_X x )^{\tilde{}}$. Consequently, by parts (2) and (3) of the Claim for $i+1$ and $i$ respectively, we deduce that
    $$(\Delta_N^{(i+1)}\cap \Delta_I^{(i+1)})\times_X x= (\Delta_N^{(i)}\cap \Delta_I^{(i)})^{\tilde{}} \times_X x = ((\Delta_N^{(i)}\cap \Delta_I^{(i)}) \times_X x )^{\tilde{}}=(\delta_I ^{(i-1)})^{\tilde{}}=\delta_I ^{(i)}.$$ 
    This concludes part (3) of the Claim for $i+1$. \textit{End of Proof of Claim.}
 \end{proof}
\textit{Proof of Theorem \ref{maintdn}(1)}: Immediate by Lemma \ref{tdnwon} and Theorem \ref{thmLi}.
\begin{corollary}\label{mainheavy}
For every $\mathcal A\in \mathcal D^T_{d,n}$:
 \begin{enumerate}
 \item $\wtdn$ is the iterated blowup of $\PP^{d(n-1)-1}$ at $\mathcal{H}_{\mathcal{A}}$ in ascending dimension order.
\item 
The boundary $\wtdn\setminus \wtdno$ is the union of $|\mathcal{H}_{\mathcal{A}}|$ smooth irreducible divisors $\Gamma_I$.
\item 
Each of the divisors $\Gamma_I$ is the iterated dominant transform of $\delta_I$ in the sequence of blowups 
$\wtdn\rightarrow\PP^{d(n-1)-1}$.
\item 
Any set of boundary divisors intersects transversally.
\end{enumerate}
\end{corollary}
\begin{proof}
The proof is a direct consequence of Lemma \ref{bset} and Theorem \ref{thmLi}.  
\end{proof}

\textit{Proof of Corollary \ref{descrip}}.
We can take $\delta_I$ (see \ref{coincident}), where $|I|=n-1$, to be the planes isomorphic to $\PP^{d-1}$ that we fix in the statement of the Corollary. Then, for $i=2,\dots n-1$, we set $\mathcal A_i$ to be the ordered set of $n$ numbers $\{\frac{1}{n-i},\frac{1}{n-i},\dots,\frac{1}{n-i}\}$. Then, clearly we have a sequence of nested sets $ \mathcal{H}_{\mathcal{A}_2}\subset \dots\subset\mathcal{H}_{\mathcal{A}_{n-1}}$  (see Definition \ref{defset}), such that the elements of $\mathcal{H}_{\mathcal{A}_i}\setminus \mathcal{H}_{\mathcal{A}_{i-1}} $ (note $\mathcal{H}_{\mathcal{A}_1}=\emptyset$) are precisely the planes in $\PP^{d(n-1)-1}$ spanned by the $(i-1)$-tuples of the $\delta_I,\,|I|=n-1$. Also, $T_{d,n}=T_{d,n}^{\mathcal{A}_{n-1}}$ and, by Corollary \ref{mainheavy}, each $T_{d,n}^{\mathcal{A}_i}$ is precisely the iterated blowup of the set $\mathcal{H}_{\mathcal{A}_{i}}$, where $i=2,\dots n-1$, in \textit{any} order of ascending dimension (see Remark \ref{partialtotal}). Finally, by Theorem \ref{maintdn}(1), each $T_{d,n}^{\mathcal{A}_i}$ is smooth.
\qed

\subsection{Structure of the boundary}
In this Section, we give a proof of Theorem \ref{maintdn} (2). We only prove the result about $\wtdn$; by our proof it will be made apparent that the analogous statement for $\wpdn$ follows in the exact same way using its description as an iterated blowup of $({\PP}^{n-d-2})^d$ (Corollary \ref{corpdn}). Let $I\subsetneq N$ such that $\sum_{i\in I} a_i>1$. Recall that the divisor $\Gamma_I$ is the dominant transform of the variety $\delta_I$ under the sequence of blowups $\wtdn\rightarrow \bP^{d(n-1)-1}$ (see (\ref{coincident}) and Theorem \ref{mainheavy}).\\

 We will prove Theorem  \ref{maintdn} (2) by studying the iterated dominant transform of $\delta_I$ in $\wtdn$ under an alternate, yet equivalent, order of blowups of $\bP^{d(n-1)-1}$. We now give a different order $(\prec)$ on $\wbuildtdn$ (Definition \ref{defset}) as follows.
\begin{definition}\label{part}
Let $\mathcal{A}\in D^T_{d,n}$. For any $I\subsetneq N$, we define the following subsets of $\wbuildtdn$:
\begin{itemize}
\item  $\mathcal{H}_1^I:=\{\delta_J\in \wbuildtdn \,| J\supsetneq I\}$ .
\item $\mathcal{H}_2^I:=\{\delta_J\in \wbuildtdn \,| J\cap I=\emptyset \}$.
\item 
$\mathcal{H}_3^I:=\{\delta_J\in \wbuildtdn \,| \emptyset \neq J \cap I \subsetneq J\}$.
\item 
$\mathcal{H}_4^I:=\{\delta_J\in \wbuildtdn \,| J\subseteq I\}$.
\end{itemize}
We give a partial order $(\prec)$ on the set $\wbuildtdn$, which is defined by the following rules:
\begin{itemize} 
\item Let $i\in\{1,2,3,4\}$ and $\delta_J,\delta_{J'} \in \mathcal{H}_i^I$ such that $|J'|<|J|$. Then  $\delta_{J}\prec\delta_{J'}$.
\item $\delta_{J_i}\prec\delta_{J_j}$ for any $\delta_{J_i}\in \mathcal{H}_i^I$ and $\delta_{J_j}\in \mathcal{H}_j^I$ such that $i<j$. 
 \end{itemize}
  \end{definition}
It is clear that the union of the sets $\mathcal{H}_i^I,\,i\in\{1,2,3,4\}$ is equal to $\wbuildtdn$.
\begin{lemma}\label{anyextension} For any  total order on $\wbuildtdn$ that extends the partial order $(\prec)$, the set $\wbuildtdn$ satisfies hypothesis (b) of Theorem \ref{thmLi}(2). 
\end{lemma}
\begin{proof}The proof is similar to the proof of Theorem \ref{maintdn} (1) so we omit it.
\end{proof}
In the remainder of this Section, we fix a \textit{total} order on $\wbuildtdn$ extending $(\prec)$ which we also denote by $(\prec)$ for convenience. The following Corollary is an immediate consequence of Lemma \ref{anyextension} and Theorem \ref{thmLi}(2).
\begin{corollary} \label{neworder}$\wtdn$ is isomorphic to the iterated blowup of $\PP^{d(n-1)-1}$ at ($\wbuildtdn$, $(\prec)$).
\end{corollary}
Corollary \ref{neworder} allows us to consider the iterated dominant transform of $\delta_I$ along the iterated blowup of  $\PP^{d(n-1)-1}$ at $\wbuildtdn$ in the order $(\prec)$ instead of the order $(<)$ given in Definition \ref{defset}. 
\begin{definition}  Let $i\in\{1,2,3,4\}$.\begin{enumerate}
\item We define $\textbf{P}_I^{[i]}$ to be the iterated blowup of $\PP^{d(n-1)-1}$ at $(\mathcal{H}^I_1\cup \dots \cup \mathcal{H}^I_i,(\prec))$.
\item Let $V$ be a subvariety of $\PP^{d(n-1)-1}$. We define $V^{[i]}$ to be the iterated dominant of $V$ in $\textbf{P}_I^{[i]}$ under the sequence of blowups $\textbf{P}_I^{[i]}\rightarrow\PP^{d(n-1)-1}$.
\end{enumerate} 
\end{definition}

We have following results, whose proofs can be found in the Appendix.

\begin{lemma}\label{dominant}
The iterated strict transform ${\delta_I}^{[3]}$ of $ \delta_I \cong \PP^{d(n-|I|)-1}$ under $\textbf{P}_I^{[3]}\rightarrow \PP^{d(n-1)-1}$  is isomorphic to $T^{\mathcal{A}_+(I^c)}_{d,n-|I|+1}$. 
\end{lemma}

\begin{lemma}\label{normal}
\begin{enumerate}
\item  The normal bundle of ${\delta_I}^{[3]}$ in $\textbf{P}_I^{[3]}$ is isomorphic to $\bigoplus\limits_{i=1}^{d(|I|-1)} \mathcal{O}_{{\delta_I}^{[3]}}(p-1)$, where $p$ is the cardinality of $\mathcal{H}^I_1$.
\item For any $\delta_{I'}\in\mathcal{H}^I_4$, the normal bundle of ${\delta_I}^{[3]}$ in ${\delta_{I'}}^{[3]}$ is isomorphic to $\bigoplus\limits_{i=1}^{d(|I|-|I'|)} \mathcal{O}_{{\delta_I}^{[3]}}(p-1)$, where $p$ is the cardinality of $\mathcal{H}^I_1$.
\end{enumerate}
\end{lemma}

Now, we are ready to finish the proof of our main theorem.\\
\textit{Proof of Theorem \ref{maintdn}(2)}:  We need to keep track of the iterated dominant transform of ${\delta_I}^{[3]}\subset \textbf{P}_I^{[3]}$ in the sequence of blowups $\textbf{P}_I^{[4]}\rightarrow \textbf{P}_I^{[3]}$ in the order $(\prec)$. By definition, the first blowup in this sequence has center ${\delta_I}^{[3]}$. As a result, the dominant transform of ${\delta_I}^{[3]}$ in this first blowup is the exceptional divisor. The latter is equal to $\PP(N_{{\delta_I}^{[3]}/\textbf{P}^{[3]}})$, which in turn is equal to (cf Lemma \ref{normal}(1)) 

\begin{align*}
\PP\left(\bigoplus\limits_{i=1}^{d(|I|-1)} \mathcal{O}_{{\delta_I}^{[3]}}(p-1)\right)  & \iso \PP\left((\bigoplus\limits_{i=1}^{d(|I|-1)} \mathcal{O}_{{\delta_I}^{[3]}}(p-1))\otimes  \mathcal{O}_{{\delta_I}^{[3]}}(1-p))\right)
=\PP\left(\bigoplus\limits_{i=1}^{d(|I|-1)} \mathcal{O}_{{\delta_I}^{[3]}}\right)\\
&=
\PP^{d(|I|-1)-1}
\times
{\delta_I}^{[3]} 
=
 \PP^{d(|I|-1)-1} \times T^{\mathcal{A}_+(I^c)}_{d,n-|I|+1} 
\end{align*}
where the last equality follows from Lemma \ref{dominant}. Moreover the dominant transform of ${\delta_{I'}}^{[3]}$ in the blowup of $\textbf{P}_I^{[3]}$ along ${\delta_I}^{[3]}$ intersects the above exceptional divisor in $\PP(N_{{\delta_I}^{[3]}/{\delta_{I'}}^{[3]}})$. As above, using Lemmas \ref{dominant} and \ref{normal}(2) we have   
\begin{center}
$\PP(N_{{\delta_I}^{[3]}/{\delta_{I'}}^{[3]}})\iso \PP^{d(|I|-|I'|)-1} \times T^{\mathcal{A}_+(I^c)}_{d,n-|I|+1} $
\end{center}
Therefore, the iterated dominant transform of ${\delta_I}^{[3]}$ in $\textbf{P}_I^{[4]}$, that is $\Gamma_I$, is isomorphic to the product 
\begin{center} 
$
Bl_{\mathcal{H}_{\mathcal{A}(I)}}\PP^{d(|I|-1)-1} \times T^{\mathcal{A}_+(I^c)}_{d,n-|I|+1} 
$
\end{center}
where $Bl_{\mathcal{H}_{\mathcal{A}(I)}}\PP^{d(|I|-1)-1}$ is the iterated blowup of $\PP^{d(|I|-1)-1}$ at the iterated strict transforms of the varieties that belong to the set
\begin{center}
$\{\PP^{d(|I|-|I'|)-1}\subset \PP^{d(|I|-1)-1}|\, I'\subsetneq I \,\text{and} \sum\limits_{i\in I'} a_i>1 \}$
\end{center}
in ascending dimension order. This is precisely the set $\mathcal{H}_{\mathcal{A}(I)}$ in Definition \ref{defset} for input data $d,I$ and $ \mathcal{A}(I)$, where $\mathcal{A}(I)$ is  defined in Section \ref{weights}. Therefore, by Corollary \ref{mainheavy}(1), we deduce the statement of Theorem.
\qed

\subsection{Proof of Proposition \ref{thm:hyper}}\label{sec:hyper}

 The iterative blow up construction of $\overline{M}^{m}(\hatbP^2,n)$ is described 
in  \cite[Thm. 5.5.2, Thm 5.5.3]{alexeev2013moduli} and it is induced by increasing the weights used in Section \ref{conspdn}
to the weights $(1, \ldots, 1)$. 
By  \cite[Thm. 5.5.2(2)]{alexeev2013moduli}, if the wall crossing is not an isomorphism, then it is a blow up  whose center is supported in  the loci parametrizing the shas that are not stable with respect to the new weights. The set of the possible weights $(b_1, \ldots, b_n) \in \mathcal D(3,n)$ has a chamber decomposition induced by two types of walls:
\begin{itemize}
\item  
$\sum_{i \in I} b_i=1$  for  all $I \subset \{ 1, \ldots, n \}$, 
$2 \leq |I| \leq n-3$.
 After crossing  this type of walls for a fixed $I$, the shas where at least one component has coincident lines $l_{i}=l_{j}$ for $i, j \in I$ become unstable with respect to the new weights. We define the loci parametrizing such pairs as  $B_I$.  
\item 
$\sum_{i \in J} b_i=2$  for  all $J \subset \{ 1, \ldots, n \}$, 
$3 \leq |J| \leq n-2$. 
 After crossing  this type of walls for a fixed $J$, the shas where at least one component has concurrent lines $\{ l_i \; | \; i \in J \}$ at a point 
become unstable with respect to the new weights.  We will not consider this locus or pairs any further.
\end{itemize} 
The order of the blow ups arises by considering the order of the walls in the  weight domain. We first cross the walls associated to $(n-3)$ coincident lines
that is 
$\sum_{i \in I} b_i=1$  for  all $I \subset \{ 1, \ldots, n \}$ with
$|I| = n-3$. Afterwards, we cross the walls associated with $(n-2)$ concurrent lines, followed by the wall associated with $(n-2)$ coincident lines, and so forth. 
\begin{figure}[h!]
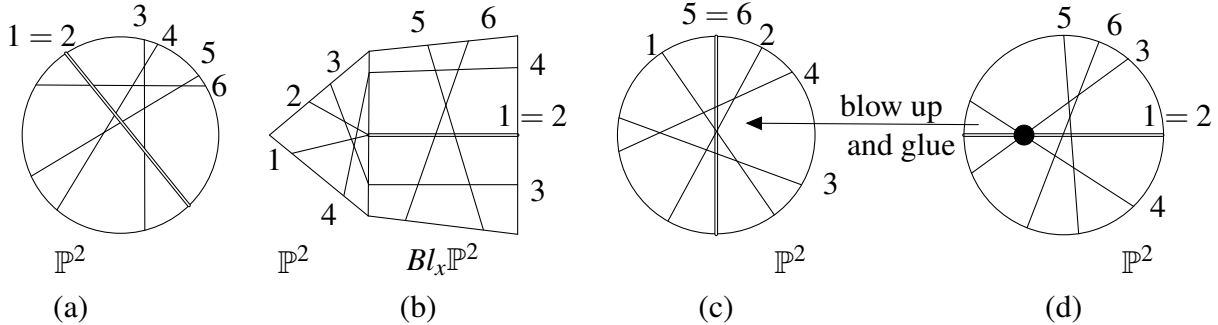

\tikzpicture[line cap=round,line join=round,>=triangle 45,x=1.0cm,y=1.0cm,
scale=0.66
]
\draw(5+1,7) circle (1.98cm);
\draw (5.49+1,5.08)-- (5.48+1,8.92);
\draw (3.3+1,8.01)-- (6.71+1,7.99);
\draw (3.72+1,5.49)-- (5.75+1,8.83);
\draw (3.2+1,6.17)-- (6.58+1,8.2);
\draw[double] (3.9+1,8.65)-- (6.38+1,5.58);

\draw[double] (11,7)-- (14,7);
\draw (9,7)-- (11.01,8.69);
\draw (9,7)-- (10.99,5.37);
\draw (10.99,5.37)-- (11.01,8.69);
\draw (11.01,8.69)-- (14,9);
\draw (14,5)-- (14,9);
\draw (10.99,5.37)-- (14,5);
\draw (9.8,7.67)-- (11,7);
\draw (9.45,6.63)-- (11,7);
\draw (10.5,5.78)-- (11.01,8.26);
\draw (10.23,8.03)-- (11,6);
\draw (11.01,8.26)-- (14,8.36);
\draw (11,6)-- (14,6);
\draw (11.74,5.28)-- (13.01,8.9);
\draw (12.19,8.81)-- (13.31,5.09);
\draw(18,7) circle (1.99cm);
\draw[double]  (18,5)-- (18,8.99);
\draw (16.91,8.66)-- (19.17,5.39);
\draw (17.03,5.27)-- (18.92,8.76);
\draw (16.04,6.67)-- (19.53,8.27);
\draw (16.04,7.35)-- (19.71,6);

\draw[double]  (21+1.99,7)-- (25.99+1,7);
\draw(24+1,7) circle (2.01cm);
\draw (22+1.14,6.23)-- (25.3+1,8.54);
\draw (22.11+1,7.69)-- (25.41+1,5.56);
\draw (23.26+1,5.13)-- (24.71+1,8.88);
\draw  (24+1,9.01)-- (24.3+1,5.01);

\draw[color=black] (5.4+1,9.4) node {$3$}; 
\draw[color=black] (6.99+1,8) node {$6$}; 
\draw[color=black] (6+1,9) node {$4$}; 
\draw[color=black] (6.78+1,8.66) node {$5$}; 
\draw[color=black] (9.5,7.8) node {$2$}; 
\draw[color=black] (9.1,6.5) node {$1$}; 
\draw[color=black] (14.29,7.4) node {$1=2$}; 
\draw[color=black] (10.3,8.49) node {$3$};
\draw[color=black] (14.41,8.42) node {$4$}; 
\draw[color=black] (14.41,5.87) node {$3$}; 
\draw[color=black] (13.42,9.37) node {$6$}; 
\draw[color=black] (12.0,9.29) node {$5$}; 
\draw[color=black] (18.02,9.45) node {$5=6$}; 
\draw[color=black] (16.7,8.82) node {$1$}; 
\draw[color=black] (19.03,9.03) node {$2$}; 
\draw[color=black] (19.93,8.24) node {$4$}; 
\draw[color=black] (26.28+1,7.4) node {$1=2$}; 
\draw[color=black] (25.6+1,8.7) node {$3$}; 
\draw[color=black] (25.9+1,5.56) node {$4$}; 
\draw[color=black] (25.03+1,9.24) node {$6$}; 
\draw[color=black] (24.02+1,9.38) node {$5$}; 
\draw[color=black] (3.4+1,9) node {$1=2$}; 
\draw[color=black] (10.2,5.4) node {$4$}; 
\draw[color=black] (20.31,6.03) node {$3$}; 
\draw[color=black] (4.0+1,4.5) node {$\PP^2$}; 
\draw[color=black] (9.5,4.5) node {$\PP^2$}; 
\draw[color=black] (12.5,4.5) node {$Bl_x\PP^2$}; 
\draw[color=black] (19.5,4.5) node {$\PP^2$}; 
\draw[color=black] (25.5+1,4.5) node {$\PP^2$}; 
\draw [->] (22.28+1,7.21) -- (17.59+1,7.24);
\draw[color=black] (20.53+1,7.54) node {blow up};
\draw[color=black] (20.73+1,6.74) node {and glue};
\draw [fill]  (24.2,7)  circle [radius=0.2];
\draw[color=black] (4.0+1,3.5) node {(a)}; 
\draw[color=black] (12,3.5) node {(b)}; 
\draw[color=black] (18,3.5) node {(c)}; 
\draw[color=black] (25,3.5) node {(d)}; 
\endtikzpicture
\caption{Configurations described in the proof of Theorem \ref{thm:hyper}}
\label{figProof}
\end{figure}

To show our claims about $ B_I$, we exhibit two types of shas parametrized by 
two different locally closed subvarieties of $ B_I$ each of which is not in the closure of the other.
From this, it will follow that $B_I$ is a reducible scheme and $B_I$ is strictly larger than the strict transform of $H_I$.  For the case $n=6$ and $I=\{ 1,2\}$ the shas are represented by (a) and (b) in Figure \ref{figProof}. 

The first sha (see Figure \ref{figProof}(a))  is a configuration of $n$ lines in $\hatbP^2$ where $l_1=\ldots =l_{|I|}$ are coincident and the rest of the lines are in general position. The locus parametrizing this configuration has dimension equal to 
$$
2((n-|I|+1)-4))=2(n-|I|-3)
$$ 
and its closure in 
$ \overline{M}^{m}_{\vec \beta_{k}}(\hatbP^2,n) $
is by definition the strict transform of $H_I$. 

The second sha (see also Figure \ref{figProof}(b)) is given by a pair $(X,D)$ where $X$ is the union of $\hatbP^2$ and  
the blow up $Bl_x\hatbP^2$ of $\hatbP^2$ at a point $x$ and $D$ is a divisor in $X$ obtained as follows. Let $\hatbP^1$ be a distinguished hyperplane in $\hatbP^2$ and let $E\iso\hatbP^1$ be the exceptional divisor of $Bl_x\hatbP^2$. The variety $Bl_x\hatbP^2$ is a $\hatbP^1$-bundle over $\hatbP^1$.
Further, let $D_1\subset {\hatbP^2}$ be the union of 
two generic lines $l_{|I|+1}, l_{|I|+2} $ in  ${\hatbP^2}$ that intersect the distinguished $\hatbP^1$ at distinct points $x_{|I|+1}$ and $x_{|I|+2}$ and concurrent lines $\{ l_i \; | \; i \in I \}$ in ${\hatbP^2}$ that intersect the distinguished $\hatbP^1$ at a single point $x_{|I|}$ which is different from $x_{|I|+1}$ and $x_{|I|+2}$.  
Further, let $ D_2\subset {Bl_x\hatbP^2}$ be the union of the  three types of divisors:
\begin{itemize}
\item a divisor with multiplicity $|I|$ defined by coincident fibers $\{  l'_{i} \; | \;  i \in I \}$ of $Bl_x\hatbP^2\rightarrow \hatbP^1$ that intersect $E\iso\hatbP^1$ at a single point $x'_{|I|}$;
\item the union of generic irreducible divisors $\{  l_s' \; | \; |I|+3 \leq s \leq n \}$ of self-intersection one in $Bl_x\hatbP^2$ and 
\item the union of two generic fibers $ l'_{|I|+1}, l'_{|I|+2} $ of $Bl_x\hatbP^2\rightarrow \hatbP^1$ that intersect the exceptional divisor at distinct points $x'_{|I|+1}$ and $x'_{|I|+2}$ both different from $x'_{|I|}$.  
\end{itemize}
We now glue the pairs ($\hatbP^2$,$\hatbP^1$) and ($Bl_x\hatbP^2$,$\hatbP^1$) along $\hatbP^1$ via the automorphism of $\hatbP^1$ that takes the triple ($x_{|I|},x_{|I|+1},x_{|I|+2}$) to ($x'_{|I|},x'_{|I|+1},x'_{|I|+2}$) to obtain $X$. The divisor $D$ is then the image of $D_1\times D_2$ in $X$ via the above gluing. We suppose that the divisor $D$ and its restrictions to $\hatbP^2$ and $Bl_x\hatbP^2$ do not have any coincident or concurrent lines besides the ones described above.


To count the dimension of the locus parametrizing $(X,D)$ we interpret this pair as two configurations of lines in $\hatbP^2$ glued together (see
\cite[6.8]{keel2006geometry}). 
In our particular example, (c) and (d) in Figure \ref{figProof} represent those configurations of lines for the sha in (b). For general $n$ these configurations and the gluing locus are described in (1)-(3) below.
Configuration (1) is obtained by contracting the component $Bl_x\hatbP^2\subset X$ to $\hatbP^1$ and configuration (2) by contracting the component $\hatbP^2\subset X$ to a point and $Bl_x\hatbP^2$ to $\hatbP^2$ by blowing down.
\begin{enumerate}
\item  
The lines $l_{1}, \ldots, l_{|I|} $ are concurrent at a point. Through that point, we have the coincident lines
$l_{|I|+3}=\ldots=l_{n}$ while the lines $l_{|I|+1}$ and $l_{|I|+2}$ are generic.
The dimension of the locus parametrizing this configuration of lines is 
$|I|-1$.
\item 
The coincident lines
$\{ l'_1= \ldots =l'_{|I|} \}$ also support the intersection of lines $l'_{|I|+1}$ and  
$l'_{|I|+2}$. The lines $\{ l'_{|I|+3}, \ldots , l'_{n} \}$ are generic.
The dimension of the locus parametrizing this configuration is $2(n-|I|)-7$.
\item The gluing locus always parametrizes configurations of three points in $\PP^1$, hence it is zero dimensional.
\end{enumerate}

By above discussion, the dimension of the loci parametrizing the stable pair $(X,D)$ is 
$
2(n-|I|)-7+(|I|-1)= 2(n-4)-|I|
$. 

Our result follows because  $2(n-4)-|I| \geq 2(n-|I|-3)$ for all $ |I| \geq 2$. Therefore, the dimension of the locus parametrizing $(X,D)$ is larger than the dimension of the strict transform of $H_I$. 
The two shas must be parametrized by different components of $B_I$ because
degenerations of the pair $(X,D)$ correspond to pairs $(X',D')$ where $X'$ is a further degeneration of $X$.  Then, the strict transform of $H_I$ cannot be in the closure of the loci parametrizing $(X,D)$ (see \cite[Sec 4]{alexeev2013moduli}). \\

Finally, the existence of the sequence of blowups resulting in $\overline P_{d,n}$ follows from Corollary \ref{corpdn}(1) for $\mathcal{A}=\{1,\dots 1\}$. The intermediate spaces in that sequence are obtained by taking $\vec \alpha_{k}$ to be the ordered set of cardinality $n$, whose first two elements are equal to $1$ and the rest equal to $\frac{1}{k}$. Each $S_I$, $|I|=k$, is the iterated strict transform of the variety $H_I$ under the iterated blowup $Bl_{\mathcal{G}_{\vec \alpha_{k}}}(\PP^{n-4})^2\rightarrow (\PP^{n-4})^2$, where $\mathcal{G}_{\vec \alpha_{k}}$ is defined in Definition \ref{hi}. Therefore, $S_I$ is smooth by Proposition \ref{buildpullback}(1). By using Theorem\ref{mainpdn}(2), we also see that its geometric points parametrize the stable trees mentioned in the statement.

\section{Reduction, Forgetful Morphisms and Toric models }\label{prodfor}
\subsection{Reduction and forgetful morphisms} We define certain operations on the weight sets in the domains $\mathcal D_{d,n}^T  $ and $\mathcal D_{d,n}^P$ (see Section \ref{weights}) that induce morphisms among our compactifications. 
\begin{proposition}\label{reduction}(Reduction)
Let $\mathcal{A}:=\{a_1,a_2, \dots,a_n\} $ and  $\mathcal{B}:=\{b_1,b_2, \dots,b_n \} $ be two weight sets in $\mathcal D_{d,n}^T  $ (resp. $\mathcal D_{d,n}^P$) such that $b_i\leq a_i\,$ for all $i=1,2,\dots n$ (resp. for all $i=d+2,\dots n$).
There exists a natural reduction morphism 
\begin{align*}
\rho_{\mathcal{B},\mathcal{A}}: \wtdn \rightarrow \wtdnb
& &
(\text{resp.}\,\, \hat \rho_{\mathcal{B},\mathcal{A}}: \wpdn \rightarrow 
\overline P_{d,n}^{\mathcal B})
 \end{align*}
Given an $\mathcal{A}$-stable rooted tree (resp. $\mathcal{A}$-stable tree) $(W,s_1,\dots,s_n)$, $\rho_{\mathcal{B},\mathcal{A}}((W,s_1,\dots,s_n))$ (resp. $\hat \rho_{\mathcal{B},\mathcal{A}}((W,s_1,\dots,s_n))$) is obtained by successively collapsing all components of $W$ that are unstable with respect to $\mathcal{B}$. 
\end{proposition}
\begin{proof} Our argument follows closely the argument in the proof of \cite[Theorem 5]{routis2014weighted}.
 Let $\wbuildtdn$ and $\mathcal{H}_{\mathcal{B}}$ with  notation as in Lemma \ref{defset}.
  By Theorem \ref{mainheavy}, we know that $\wbuildtdn$ and $\wbuildtdnb$ are building sets. Let us denote by $\mathcal{M}_{\mathcal{A}}$ (resp. $\mathcal{M}_{\mathcal{B}}$) the set of ideal sheaves of the varieties $\delta_I \in \wbuildtdn$ (resp. $\delta_I \in \wbuildtdnb$) in $\PP^{d(n-1)-1}$. By the hypothesis $\mathcal{M}_{\mathcal{B}}\subset\mathcal{M}_{\mathcal{A}}$. Now let us consider the following ideal sheaves in $\PP^{d(n-1)-1}$:
  
  \begin{align*}
  \mathcal{I}_{\mathcal{B}}:=\displaystyle\prod\limits_{\mathcal{I}\in\mathcal{M}_{\mathcal{B}}} \mathcal{I}, && \mathcal{I}_{\mathcal{A}\mathcal{B}}:=\displaystyle\prod\limits_{\mathcal{I}\in\mathcal{M}_{\mathcal{A}}\setminus\mathcal{M}_{\mathcal{B}}} \mathcal{I} &&\text{and}\,\,\,\,\,\,\,\, \mathcal{I}_{\mathcal{A}}:=\displaystyle\prod\limits_{\mathcal{I}\in\mathcal{M}_{\mathcal{A}}} \mathcal{I}\end{align*}

  Now, by Lemma \ref{bset} and Theorem \ref{thmLi}(3), $\wtdnb$ is isomorphic to the blowup of $\PP^{d(n-1)-1}$ with respect to the ideal sheaf $\mathcal{I}_{\mathcal{B}}$; let $\chi_{\mathcal{A}}:\wtdnb\rightarrow  \PP^{d(n-1)-1}$ be the blowup morphism. Further, let $\chi_{\mathcal{A}}^{-1}\mathcal{I}_{\mathcal{A}\mathcal{B}}\cdot \mathcal{O}_{\wtdnb}$ be the inverse image ideal sheaf of $\mathcal{I}_{\mathcal{A}\mathcal{B}}$ in $\wtdnb$. Then the blowup of $\wtdnb$ with respect to the ideal sheaf $\chi_{\mathcal{A}}^{-1}\mathcal{I}_{\mathcal{A}\mathcal{B}}\cdot \mathcal{O}_{\wtdnb}$ is isomorphic to the blowup of $\PP^{d(n-1)-1}$ with respect to the ideal sheaf $\mathcal{I}_{\mathcal{A}}$ (see \cite[Lemma 3.2]{LiLi})
, which, in turn, is isomorphic to $\wtdn$, by Lemma \ref{bset} and Theorem \ref{thmLi}(3) .
  
Therefore we  obtain
a natural blowup morphism   $\rho_{\mathcal{B},\mathcal{A}}: \wtdn \rightarrow \wtdnb$, which has the desired interpretation on geometric points.
The proof for $\hat \rho_{\mathcal{B},\mathcal{A}}$ is entirely analogous. 
\end{proof}

\begin{figure}[h]\label{sbrooted}
\caption{\footnotesize\textit{Example of reduction morphisms for stable rooted trees with $n=6$.}}
 \begin{tikzpicture}[scale =0.75]

\begin{scope}[shift={(15,0)}] 
\draw [fill=white] (0,0) rectangle (3.5,2.7);
\node[mark size=2pt,color=black] at (1.8,2) {\pgfuseplotmark{*}};
\node [above] at (1.8,2) {\footnotesize $1,2,3,4,5$};
\node [below] at (1.8,2) {\footnotesize $x$};

\node[mark size=2pt,color=black] at (2.8,1.4) {\pgfuseplotmark{*}};
\node [below] at (2.8,1.4) {\footnotesize $6$};

\node [below] at (1.75,-0.01) {\footnotesize $\bP^d$};
\draw [thick](0.5,0.6)--(3,0.6);
\node [below] at (1.5,0.6) {\footnotesize $H$};
\end{scope}

\begin{scope}[shift={(2,0)}] 
\draw [fill=white] (6,0) rectangle (9.5,2.7);
\draw [fill=white] (6.2,1.7) rectangle (9.1,3.8);
 \node[mark size=2pt,color=black] at (7,3) {\pgfuseplotmark{*}};
\node [below] at (7,3) {\footnotesize $1,2,3$};
\node [above] at (7,3) {\footnotesize $x'$};
 \node[mark size=2pt,color=black] at (8.4,2.7) {\pgfuseplotmark{*}};
\node [below] at (8.4,2.7) {\footnotesize $4,5$};
\node [above] at (8.4,2.7) {\footnotesize $x''$};
\node [below] at (7.75,-0.01) {\footnotesize $Bl_x\bP^d$};

\node[mark size=2pt,color=black] at (8.5,1.3) {\pgfuseplotmark{*}};
\node [below] at (8.5,1.3) {\footnotesize $6$};

\draw [thick](6.5,0.6)--(9,0.6);
\node [below] at (8,0.6) {\footnotesize $H$};
\node [above] at (8,3.8) {\footnotesize {$\bP^d$}};
\end{scope}

\begin{scope}[shift={(-12,0)}] 
\draw [fill=white] (12,0) rectangle (15.2,2.7);
\draw [fill=white] (12.2,1.7) rectangle (15.1,3.8);
\draw [fill=white] (12.4,3) rectangle (13.7,5);
 \node[mark size=2pt,color=black] at (12.7,4.5) {\pgfuseplotmark{*}};
 \node [below] at (12,4.7) {\footnotesize$2,3$};
  \node[mark size=2pt,color=black] at (13.4,4) {\pgfuseplotmark{*}};
 \node [below] at (13.3,4) {\footnotesize$1$};
 \draw [fill=white] (14,3) rectangle (15,5);
 \draw [thick](12.5,0.6)--(15,0.6);

\node[mark size=2pt,color=black] at (14.5,1.3) {\pgfuseplotmark{*}};
\node [below] at (14.5,1.3) {\footnotesize $6$};

 \node[mark size=2pt,color=black] at (14.2,4.5) {\pgfuseplotmark{*}};
 \node [below] at (14.3,4.5) {\footnotesize$4$}; 
  \node[mark size=2pt,color=black] at (14.7,3.8) {\pgfuseplotmark{*}};
 \node [below] at (14.7,3.8) {\footnotesize$5$}; 
 \node [below] at (13.75,-0.01) {\footnotesize$Bl_x\bP^d$};
 \node [above] at (14.5,5) {\footnotesize {$\bP^d$}};
\node [above] at (13,5) {\footnotesize {$\bP^d$}};
\node [below] at (13.5,2.7) {\footnotesize $Bl_{x',x''}\bP^d$};
\node [below] at (13.5,0.7) {\footnotesize $H$};
\end{scope}

 \begin{scope}[shift={(8,0)}] 
  \node at (4,1.5) (a){} ;
 \node at (5,2) (text) {$\rho_{\mathcal C, \mathcal B}$};
  \node at (6,1.5) (b) {};
  \draw[->] (a.east)  to (b.east);
  \end{scope}

\begin{scope}[shift={(-5,0)}] 
    \node at (9,1.5) (c){} ;
     \node at (11,2) (text) {$\rho_{\mathcal B, \mathcal A}$};
  \node at (11.6,1.5) (d) {};
  \draw[->] (c.east)  to (d.east);
\end{scope}
  \end{tikzpicture}
$$
\mathcal A
 =\left\lbrace
1, \frac{1}{3} +\epsilon,\frac{1}{3} +\epsilon, 1,1,1
 \right\rbrace
, \qquad{}
\mathcal B=
\left\lbrace
\frac{1}{5}+\epsilon, \ldots,\frac{1}{5}+\epsilon,1
 \right\rbrace
, \qquad{}
\mathcal C=
\left\lbrace
\frac{1}{6}+\epsilon, \ldots, \frac{1}{6}+\epsilon, 1
 \right\rbrace
$$
\end{figure}
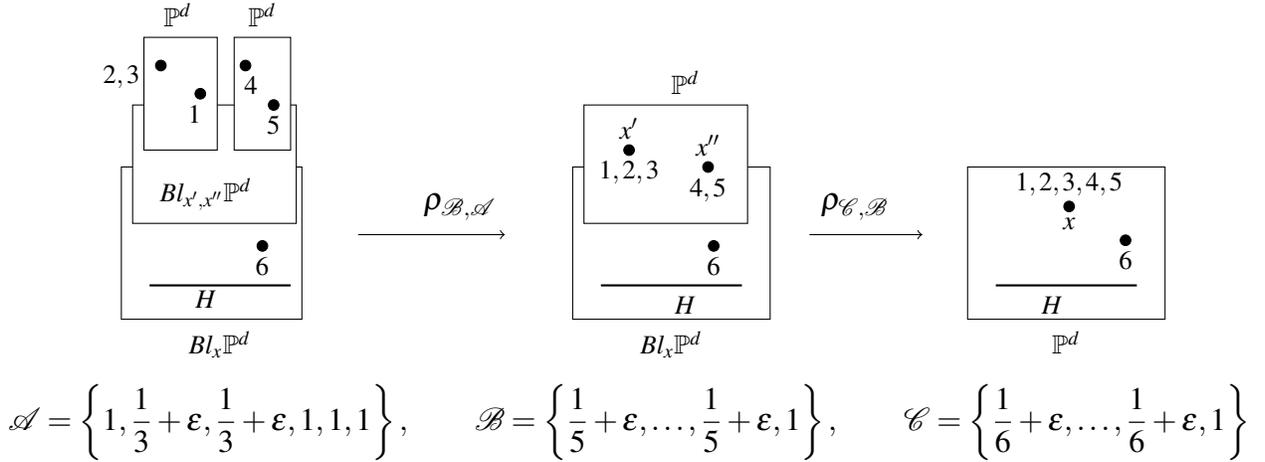

The above morphisms behave favourably under weight reduction, as the following Proposition shows.  We omit its proof, since it is identical to the proof of \cite[Proposition 5]{routis2014weighted}.

\begin{proposition}
 Let $\mathcal{A}:=\{a_1,a_2, \dots,a_n \} $,  $\mathcal{B}:=\{b_1,b_2, \dots,b_n \} $ and $\mathcal{C}:=\{c_1,c_2, \dots,c_n \} $ be weight sets in $\mathcal D_{d,n}^T $ (resp. $\mathcal D_{d,n}^P$) such that $c_i\leq b_i\leq a_i$ for all $i=1,2,\dots n$. Then: 
\begin{align*}
\rho_{\mathcal{C},\mathcal{A}}= \rho_{\mathcal{C},\mathcal{B}}\circ \rho_{\mathcal{B},\mathcal{A}}, 
&& 
(resp.\,\, \hat \rho_{\mathcal{C},\mathcal{A}}= \hat \rho_{\mathcal{C},\mathcal{B}}\circ \hat \rho_{\mathcal{B},\mathcal{A}})
\end{align*}
\end{proposition}

\begin{proposition}\label{forgetful}(Forgetful)
Let $R$ be a subset of $N=\{1,2,\dots,n\}$ and $\mathcal{A}$ be a weight set in $\mathcal D_{d,n}^T $ (resp. $\mathcal D_{d,r}^P,$ where $r= |R|$). Let $\mathcal{A}(R)$ be the subset of 
$\mathcal{A}$ (resp. with the additional assumption that $R\supseteq \{1,\dots ,d+1\}$) . Then, there exists a natural forgetful morphism
\begin{align*}
\phi_{\mathcal{A},\mathcal{A}(R)}: \wtdn \rightarrow \wtdr
& &
(resp. \,\,\hat \phi_{\mathcal{A},\mathcal{A}(R)}: 
\overline P^{\mathcal A}_{d,n} \to 
\wpdr).
\end{align*}
 Given an $\mathcal{A}$-stable rooted tree (resp. $\mathcal{A}$-stable tree) $(W,s_1,\dots,s_n)$,  $\phi_{\mathcal{A},\mathcal{A}(R)}((W,s_1,\dots,s_n))$ (resp. $\hat \phi_{\mathcal{A},\mathcal{A}(R)}((W,s_1,\dots,s_n))$) is obtained by successively collapsing all components of $W$ that are unstable with respect to $\mathcal{A}(R)$. \\
\end{proposition}
\begin{proof} 
We start with the morphism $\phi_{\mathcal{A},\mathcal{A}(R)}$. By \cite[Theorem 6]{routis2014weighted} and its proof, there exists a morphism 
$$
 \bP^d_{\mathcal A}[n] \to  \bP^d_{\mathcal{A}(R)}[r] \times (\bP^d)^{n-r}
$$
Let $D_N\subset  \bP^d_{\mathcal A}[n]  $ and $D_{R}\subset  \bP^d_{\mathcal{A}(R)}[r] $ be the divisors corresponding to the small diagonals $\Delta_N\subset (\bP^d)^N$ and $\Delta_{R}\subset (\bP^d)^{R}$ respectively. 
 \\

\textit{Claim}: The restriction of $ \bP^d_{\mathcal A}[n] \to  \bP^d_{\mathcal{A}(R)}[r] \times (\bP^d)^{n-r}$ to $D_N$ surjects onto a subvariety isomorphic to $D_{R}$.
\\
\textit{Proof of Claim}: 
By the proof of \cite[Theorem 6]{routis2014weighted}, 
$ \bP^d_{\mathcal{A}(R)}[r] \times ( \bP^d)^{n-r} \to (\bP^d)^{n}$  is obtained as a sequence of blowups of $(\bP^d)^n$ along the iterated dominant transforms of the set 
$$
\mathcal K_{\mathcal{A}(R)}=\{\Delta_I \subset (\bP^d)^n | \, I\subset R\, \text{and}\,
 \sum\limits _{i_k\in I} a_{i_k} >1\}
$$ 
in ascending dimension order. Moreover, $ \bP^d_{\mathcal A}[n]$ is obtained from 
$ \bP^d_{\mathcal{A}(R)}[r] \times (\bP^d)^{n-r}$ by blowing up ideal sheaves corresponding to 
$\mathcal K_{\mathcal A} \setminus \mathcal K_{\mathcal{A}(R)}$. By \cite{LiLi}, $D_N$ is the iterated dominant transform of $\Delta_N$ along the sequence of blowups 
$ \bP^d_{\mathcal A}[n] \rightarrow  \bP^d_{\mathcal{A}(R)}[r] \times (\bP^d)^{n-r}\rightarrow (\bP^d)^n$. It therefore suffices to observe that the iterated dominant transform of $\Delta_N$ along the sequence of blowups $ \bP^d_{\mathcal{A}(R)}[r]  \times (\bP^d)^{n-r}\rightarrow (\bP^d)^n$
is isomorphic to the divisor $D_{R}\subset \bP^d_{\mathcal{A}(R)}[r]$. Indeed, consider the embedding $j:(\bP^d)^r\rightarrow(\bP^d)^r\times (\bP^d)^{n-r}$, which is obtained as the graph of the composite morphism $(\bP^d)^r\xrightarrow{q_i}\bP^d\xrightarrow {diag}(\bP^d)^{n-r}$ where $q_i$ is the projection to the $i$-th factor (for any $i\in N\setminus R$). Then $\Delta_N\subset (\bP^d)^n$ is the image of $\Delta_R\subset (\bP^d)^r$  via $j$. Therefore, since the iterated dominant transform of $\Delta_{R}$ along 
$ \bP^d_{\mathcal{A}(R)}[r] \rightarrow (\bP^d)^{r}$ is $D_R$ (theorem \ref{thmLi}), we conclude that the iterated dominant transform of $\Delta_N$ in $ \bP^d_{\mathcal{A}(R)}[r]  \times (\bP^d)^{n-r}$ is the graph of $D_{R}\rightarrow (\bP^d)^{n-r}$. 
\textit{End of proof of Claim.}

In view of the claim, we have a morphism of $D_N$ to $D_{R}$ over $\bP^d$, which pulls back to a morphism $\wtdn\rightarrow \wtdr$ between their fibers over $x\in \bP^d$
as illustrated in the following diagram.
\begin{displaymath}
    \xymatrix{
\wtdn \ar[rr] \ar[dr] \ar[dd] &&  D_N \ar[dd] |!{[d];[rr]}\hole \ar[dr] 
 \ar[rr]&& \bP^d_{\mathcal A}[n] \ar [dr]\ar [dd] |!{[d];[d]}\hole
\\
& \wtdr \ar[dl] \ar[rr]&&  D_{R} \ar[dl] \ar[rr]&& \bP^d_{\mathcal{A}(R)}[r] \times (\bP^d)^{n-r} \ar[dl]
\\
x \ar[rr] && \bP^d \ar[rr]^-{diag.} && (\bP^d)^n =(\bP^d)^r\times (\bP^d)^{n-r}
}
\end{displaymath}

Next we prove the existence of the map $\hat \phi_{\mathcal{A},\mathcal{A}(R)}$. By  \cite[Theorem 6]{routis2014weighted},  there exists a natural  forgetful morphism
$ \psi_R: \bP^d_{\mathcal A}[n] \to \bP^d_{\mathcal{A}(R)}[r]$. From the proof in [ibid.], $\psi_R$ is the composition of a sequence of blowups $\bP^d_{\mathcal A}[n]\rightarrow \bP^d_{\mathcal{A}(R)}[r]\times (\bP^d)^{n-r}$ at $SL_{d+1}$-invariant loci with the projection of $\bP^d_{\mathcal{A}(R)}[r]\times (\bP^d)^{n-r}$ to the first factor. Therefore $\psi_R$ is $SL_{d+1}$-equivariant. To verify our statement, we check that $\psi_R$ takes $ (\bP^d_{\mathcal A}[n])^s$ to $(\bP^d_{\mathcal{A}(R)}[r])^s $. Therefore we have a commutative diagram
$$
\xymatrixcolsep{5pc}\xymatrix{
\bP^d_{\mathcal A}[n] \ar[r]\ar[dr]_{\pi_{\mathcal{A}}}&\bP^d_{\mathcal{A}(R)}[r]\times (\bP^d)^{n-r} \ar[d]_{ \pi_{\mathcal{A}(R)}\times id} \ar[r]& \bP^d_{\mathcal{A}(R)}[r] \ar[d]^{\pi_{\mathcal{A}(R)}}
\\
&(\bP^d)^r \times (\bP^d)^{n-r}  \ar[r]^{q_R} & (\bP^d)^r
}
$$
where the morphism $q_R$ is the projection from $(\bP^d)^n=(\bP^d)^r \times (\bP^d)^{n-r} $ to $(\bP^d)^r$. Now, recall that $(\bP^d_{\mathcal A}[n])^s$ is equal to the preimage of $((\bP^d)^n)^s$ under $\pi_{\mathcal{A}}$
(see \ref{stablelocus}). It would therefore suffice to show that the preimage of $((\bP^d)^n)^s$ under $\pi_{\mathcal{A}(R)}\times id$ maps to $(\bP^d_{\mathcal{A}(R)}[r])^s$ via the projection $\bP^d_{\mathcal{A}(R)}[r]\times (\bP^d)^{n-r}\rightarrow \bP^d_{\mathcal{A}(R)}[r]$. But $(\bP^d_{\mathcal{A}(R)}[r])^s$ is in turn equal to the preimage of $((\bP^d)^r)^s$ under $\pi_{\mathcal{A}(R)}$. Consequently, it is enough to show that the projection $q_R$ takes $((\bP^d)^n)^s$ to $((\bP^d)^r)^s$, which can be seen directly using \cite[Thm 11.2]{dolgachev2003lectures}.

Since $\psi_R$-at the level of $k$-points- successively collapses all components of an $\mathcal A$-stable degeneration that are unstable with respect to $\mathcal A(R)$, we deduce, by (3) in Theorem \ref{GITHu}, that the morphism
$ \overline P^{\mathcal A}_{d,n}  \to \overline  P^{\mathcal{A}(R)}_{d,n}$ 
has the desired moduli interpretation at geometric points.
\end{proof}
Next, we denote as $\pi_I$ the forgetful map $\pi_I:T_{d,n} \rightarrow T_{d,|I|}$  obtained by forgetting the points $\{ p_{i} \; | \; i \in I^c \}$ and stabilizing afterwards.
First, we illustrate  a particular case which leads us to Theorem  \ref{thmProd}.  

\begin{example}\label{cbd24}
Consider the three dimensional loci $T_{2,2} \times ( T_{2,2} \times T_{2,2})   \subset T_{2,4} $ that parametrizes a stable tree 
$X=X_1 \cup X_2 \cup X_3$ as on the adjacent figure. 
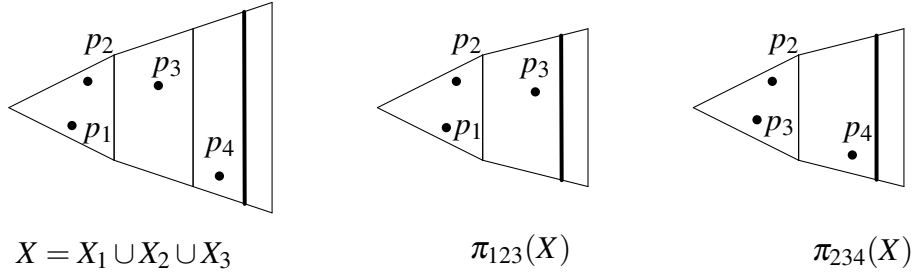
\begin{center}
\begin{tikzpicture}[line cap=round,line join=round,>=triangle 45,x=1.0cm,y=1.0cm, scale=0.7]
\draw [color=black] (0,7)-- (2,6);
\draw [color=black] (2,6)-- (2,8);
\draw [color=black] (2,8)-- (0,7);
\draw [color=black] (2,8)-- (4,8.5);
\draw [color=black] (4,8.5)-- (4,5.5);
\draw [color=black] (4,5.5)-- (2,6);
\draw [color=black] (2,6)-- (2,8);
\draw [color=black] (-7,7)-- (-5,6);
\draw [color=black] (-5,6)-- (-5,8);
\draw [color=black] (-5,8)-- (-7,7);
\draw [color=black] (-5,6)-- (-3.5,5.5);
\draw [color=black] (-3.5,5.5)-- (-3.5,8.5);
\draw [color=black] (-3.5,8.5)-- (-5,8);
\draw [color=black] (-5,8)-- (-5,6);
\draw [color=black] (-3.5,5.5)-- (-2,5);
\draw [color=black] (-2,5)-- (-2,9);
\draw [color=black] (-2,9)-- (-3.5,8.5);
\draw [color=black] (-3.5,8.5)-- (-3.5,5.5);
\draw [color=black] (6,7)-- (8,8);
\draw [color=black] (8,8)-- (8,6);
\draw [color=black] (8,6)-- (6,7);
\draw [color=black] (8,6)-- (10,5.5);
\draw [color=black] (10,5.5)-- (10,8.5);
\draw [color=black] (10,8.5)-- (8,8);
\draw [color=black] (8,8)-- (8,6);
\draw [line width=1.5pt] (-2.53,8.82)-- (-2.52,5.17);
\draw [line width=1.5pt] (3.49,8.37)-- (3.5,5.62);
\draw [line width=1.5pt] (9.48,8.37)-- (9.48,5.63);

\fill [color=black] (9.02,6.1) circle (2.5pt);
\draw[color=black] (9.13,6.54) node {$p_4$};
\fill [color=black] (7.5,7.5) circle (2.5pt);
\draw[color=black] (7.72,8.2) node {$p_2$};
\fill [color=black] (7.21,6.77) circle (2.5pt);
\draw[color=black] (7.64,6.64) node {$p_3$};
\draw[color=black] (9.22,4.25) node {$\pi_{234}(X)$};
\draw[color=black] (2.72,4.29) node {$\pi_{123}(X)$};
\fill [color=black] (3,7.3) circle (2.5pt);
\draw[color=black] (3.0,7.74) node {$p_3$};
\fill [color=black] (1.5,7.5) circle (2.5pt);
\draw[color=black] (1.7,8.2) node {$p_2$};
\fill [color=black] (1.31,6.62) circle (2.5pt);
\draw[color=black] (1.73,6.56) node {$p_1$};
\fill [color=black] (-5.5,7.5) circle (2.5pt);
\draw[color=black] (-5.28,8.2) node {$p_2$};
\fill [color=black] (-5.8,6.66) circle (2.5pt);
\draw[color=black] (-5.3,6.52) node {$p_1$};
\fill [color=black] (-4.16,7.42) circle (2.5pt);
\draw[color=black] (-3.99,7.77) node {$p_3$};
\fill [color=black] (-3,5.7) circle (2.5pt);
\draw[color=black] (-2.99,6.24) node {$p_4$};
\draw[color=black] (-4.8,4.19) node {$X=X_1 \cup X_2 \cup X_3$};
\end{tikzpicture}
\end{center}
The morphism  $\pi_{123}$  contracts the last component $X_3 \cong Bl_x \PP^2$ to a line. We obtain  a configuration of points parametrized by  $T_{2,2} \times T_{2,2}$ with the point $p_{1}$ and $p_2$ supported in the first surface and $p_3$ supported in the last one. 
We can recover the position of $p_1$, $p_2$ and $p_3$ in $X$ from $\pi_{123}(X)$ but we lost the information of $p_4$.   Similarly,  the morphism $\pi_{234}$ contracts  $X_1 \cong \PP^2$ and we lost the information of the points $p_1$ and $p_2$, but   the position of the points  $p_3$ and  $p_4$ can be recovered from $\pi_{234}(X)$. By using all possible subsets $|I|=3$ we can recover the initial configuration of points in $X$ uniquely. 
\end{example}

The following result is essentially the one described in above example, and it started from long discussions with N. Giansiracusa in the contex of \cite{gallardo2015chen}. We recall that  for $T_{1,n} \cong \overline{M}_{0,n+1}$ the product of forgerful morphisms is injective  (see \cite[Thm 1.3]{giansiracusa2010kapranov}). We generalize this result for all $T_{d,n}$.
\begin{theorem}\label{thmProd}
For any $3\leq k \leq n$ the product 
$$
\pi_{k}: T_{d,n} \to \prod_{|I|=k}T_{d,I}  
$$ 
of forgetful morphisms
$\pi_I:T_{d,n} \rightarrow T_{d,|I|}$ over all subsets $I \subset \{1, \ldots, n \}$ of cardinallity $|I|=k$ is injective.  In contrast,  if $k=2$  the morphism  $\pi_{2}$  has positive-dimensional fibers.
\end{theorem}
\begin{proof}
We first  show the statement for the open locus  $T^0_{d,n}$ that parametrizes $n$ distinct points.  Let $X$ be one of those stable rooted trees, and let $p_1, \ldots p_n$ be its  marked points.  For the sake of clarity and only in this paragraph  we write the argument for $k=3$, the one for  general $k$ follows verbatim.   Select a set $I$ with three indices say $I=\{1,2,j\}$.
Recall that $\pi_I: T_{d,n} \to T_{d,|I|}$,  the support of both $X$ and $\pi_I(X)$ is $\mathbb{P}^d$, and without loss of generality, we can fix the same position for $p_1$ and $p_2$ in both $X$ and $\pi_I(X)$.   
The key observation is that fixing $p_1$ and $p_2$ fixes the location of $p_j$ in both $X$ and $\pi_I(X)$  completely.  The situation is identical  to the one for $M_{0,n}$ where fixing three points in a $\mathbb{P}^1$ assigns unique coordinates to the other $(n-3)$ points in that projective line.    Then, we can uniquely recover the coordinates of $p_j$ in $X$ from $\pi_I(X)$.  Since we are considering \emph{all} subsets $I$  with $|I|=3$, we recover uniquely all points in the stable tree $X$ from their images $\pi_I(X)$.

Next, we consider a stable rooted tree $X =\cup_v X_v$ parametrized by the boundary. 
 Let  $k$ be a fixed integer with $3\leq k \leq n$ and  let $I(v)$ be the set of indices of the marked points contained in the component $X_v$. For instance, in Example \ref{cbd24}, we have $I(1)=\{1,2\}$, $I(2) =\{ 3 \}$ and $I(3)=\{4\}$.   
Suppose that $X_v$ is a component such that $  |I(v)| \leq k $. Then, there is a set of indices $K$ with $|K|=k$ such that $\pi_K$  leaves the positions of the points in $X_v$ unchanged because we can choose it to  be $I(v) \subset K$. This means that from $X \to \pi_K(X)$, we recover all the points  $p_i \in X_v$.  For instance in Example \ref{cbd24}, the set   $K=\{1,2,3\}$ allows us to recover the points in the component $X_1$. 
Next, suppose that $X_{\tilde v}$ is a component such that 
$3 \leq k <  |I(\tilde v)|$. If we choose  a $J \subset I(\tilde v) $ then it holds that 
$X \to \pi_J(X) \cong \mathbb{P}^d$. We can uniquely determine the points in 
$X_{\tilde v}$ by using \emph{all} the indices $J$ such that $J \subset I(\tilde v)$ and $|J|=k$.  The argument is the same as the one used in the previous paragraph:  Fixing two points, say $p_{i_1}$ and $p_{i_2}$ in both $X_{\tilde v}$ and $\pi_J(X)$, will completely determine the positions of all $p_i$ with $i \in J$.   Therefore,  the position of the points in any component of  $X$ can be recovered by considering all such subsets $J$ and our statement follows.

Finally, we treat $k=2$.  The problem is that we cannot distinguish configurations where all  points are collinear.  Let $I  \subset \{1, \ldots, n\}$ be a subset of two elements and let $l(I)$ be the line in $\PP^d$ generated by two points $p_i$ with $i \in I$.   Notice that the image of forgetful morphism 
$
\pi_{I}: T^o_{d,n} \to T_{d,2} \cong \PP^{d-1} 
$
 is defined by intersecting the line $l(I)$ with the root  $H$. 
Indeed,  without loss of generality we may take $I =\{1,2 \}$. After making an appropriate translation, we may assume $p_1=[1:0:\dots:0]$. After this choice, our automorphism group is $\mathbb{G}_m$. This $\mathbb{G}_m$-action fixes both the root and the point $p_1$. It acts on the line $l(I)$ by translating $p_2$ along it.  Therefore, all pairs of distinct points ${p_1, p_2}$  supported on $l(I)$ define the same $G$-orbit;  and we can take $l(I) \cap H$ to be the  image of $X$ in $\mathbb{P}^{d-1}$.  In particular,  for a given $X \in T^0_{d,n}$, the image $\pi_I(X)$ does not depend on the points $p_i \not\in I$. 

The above argument implies that the product of forgetful morphisms $\pi_2$ is generated by the intersection of the root $H$ with the lines $l(I)$ such that $I \subset \{1 \ldots, n \}$ and $|I|=2$. If the $n$ points are collinear, there is only one line $l(I)$ generated by all pairs of points. This line intersects the root at the same point regardless of the positions of the points inside $l(I)$. The loci parametrizing configurations of $n$ collinear distinct points is positive dimensional  in $T^0_{d,n}$ for any $n \geq 3$ and it will be contracted by $\pi_2$.
\end{proof}


\subsection{Toric Compactifications}\label{sec:toric}
Next, we describe toric models for our configuration spaces  by choosing appropriate weights.
They are generalizations of the toric compactification of $M_{0,n}$ known as the Losev-Manin space \cite{losev2000new-moduli-spaces}. This toric model can be identified with Hassett's  moduli space of weighted stable curves for the set of weights  $\mathcal{A}_{LM}=(\epsilon, \ldots, \epsilon,1,1)$ (see \cite[Sec 6]{Hassett-weighted}).

To describe a toric model of $T_{d,n}$, we denote the rays of the fan associated to $\mathbb{P}^{d(n-1)-1}$ as  
\begin{align*}
\{ \vec e^{1}_1, \ldots, \vec e^{d}_1,
 \vec e^{1}_2 \ldots, \vec e^{d}_2 , \ldots 
\vec e^{1}_{n-1}, \ldots, \vec e^{d}_{n-1}  \}
& & 
\text{ with  }  & &
\vec e^k_i \in \mathbb{Z}^{d(n-1)} /   \sum_{i,k} \vec e^k_i=0 
\end{align*}
 where  $ \vec e^{k}_i$ has its unique non-zero entry  at the index 
$d(i-1)+k-1$.  For example, for $\mathbb{P}^3$ we have 
\begin{align*}
\vec e^{1}_1=(1,0,0,0), & & \vec e^{2}_1=(0,1,0,0), & &
\vec e^{1}_2=(0,0,1,0), & & \vec e^{2}_2=(0,0,0,1).
\end{align*}

\begin{corollary}\label{toricTdn}
Given $\epsilon = \frac{1}{n-1}$, the compactification  $T^{LM}_{d,n}$  associated to the set of weights 
$(\epsilon, \ldots, \epsilon, 1)$  is a toric variety whose fan has rays of the form $\vec e^{1}_1, \ldots,  \vec e^{d}_{n-1}  $ and
$ \sum_{i \in I} \left( \vec e^1_i+ \ldots+ \vec e^d_i \right) $ where $1 \leq |I| \leq n-2$ and $I \subsetneq \{1, \ldots, n-1 \}$.
\end{corollary}
\begin{proof}
Definition \ref{defset} implies that the building set associated to our weights is  $\{ \delta_I \; | \; n \in I \}$. By Expression \ref{coincident} that loci is supported in the toric boundary, so the compactification is a toric variety.
Each  center is the intersection of divisors associated to the rays 
$\{ \vec e_{i_k}^1, \ldots \vec e_{i_k}^d \}$ where $i_k \in I$. The wonderful compactification involves blowing up these intersections, each of which generates a divisor associated to the ray 
$\sum_{i \in I }\left( \vec e^1_i+ \ldots+ \vec e^d_i \right) $, because the blow up is smooth.
\end{proof}

Next, we describe the toric model of $\overline P_{d,n}$. We denote the rays of the fan of 
$\left( \mathbb{P}^{n-d-2}\right)^d$ as
$\{  e^i_{d+2}, \ldots e^i_{n}\}$
where  $e^i_k \in \mathbb{Z}^{d(n-1)-1}$ 
with $1 \leq i \leq d$.
\begin{corollary}
\label{toricPdn}
Given $\epsilon=\frac{1}{n-d-1}$, 
the compactification $\overline P_{d,n}^{LM}$ associated to the weights
\begin{align*}
a_1=\ldots =a_{d+1}=1, \,\,\,\,\,\,  a_{d+2}= \ldots= a_{n}=\epsilon
\end{align*}
is a toric variety whose fan has rays of the form $e^i_{d+2}, \ldots, e^i_{n}$ and 
$
\sum_{i\in I}\left( e^1_{i} + \ldots +e^{d}_{i} \right)
$
where  $1 \leq |I| \leq n-d-2$  and $I \subset \{d+2, \ldots, n\}$.
\end{corollary}
\begin{proof}
The proof is the same than the one of Corollary \ref{toricTdn}. Only now we use Lemma \ref{wallStrata} which implies the building set is supported in the toric boundary of $\left( \bP^{n-d-2} \right)^{d}$, so our compactification is a toric variety. Each center is the intersection of the divisors associated to the rays $ e^1_{i_1}, \ldots e^d_{i_s}$ with $i_k \in I$. Then, the rays obtained by blowing up these loci are the ones in the statement.
\end{proof}
\begin{remark}
V. Alexeev communicated to the first author that the moduli space of weighted hyperplane arrangements for the choice of weights 
$a_1=\ldots =a_{d+1}=1, a_{d+2}= \ldots = a_n=\epsilon$ is also a toric variety constructed from a sequence of blow ups of  $\left( \bP^{n-d-2} \right)^{d}$ which generalizes the Losev-Manin space (see also Section \ref{hyparr}).
\end{remark}

\section{Appendix}\label{apx}

\begin{lemma}\label{commute} 


Let $Z\hookrightarrow X$ be an embedding of smooth varieties, both flat over a variety $Y$. For any geometric point $y\in Y$ let $Z_y$ and $X_y$ be the fibers of $Z\rightarrow Y$ and $X\rightarrow Y$ over $y$. Then the blowup $Bl_Z X$ of $X$ at $Z$ is flat over $Y$ and we have an isomorphism $$Bl_Z X\times_Y y \iso Bl _{Z_y}X_y.$$
\end{lemma}
\begin{proof}Let $\mathcal {I}$ be the ideal sheaf corresponding to the embedding $Z\hookrightarrow X$. Since $Z$ and $X$ are smooth, that embedding is regular. Consequenlty, $\mathcal{I}^n/\mathcal{I}^{n+1}$ is a locally free, hence flat, sheaf of $ \mathcal{O}_X/\mathcal{I}$- modules for all $n\geq 0$. By the hypothesis $ \mathcal{O}_{X}/\mathcal{I}$ is flat over $\mathcal{O}_Y$, consequently $\mathcal{I}^n/\mathcal{I}^{n+1}$ is also flat over $\mathcal{O}_Y$. Then, by the exact sequence
  $$ 0\rightarrow \dfrac{\mathcal{I}^n}{\mathcal{I}^{n+1}}\rightarrow  \dfrac{\mathcal{O}_X}{\mathcal{I}^{n+1}} \rightarrow  \dfrac{\mathcal{O}_X}{\mathcal{I}^{n}}\rightarrow 0
 $$ 
\vspace{0.1in}
 \noindent we deduce by induction that $ \mathcal{O}_{X}/\mathcal{I}^{n}$ is also flat over $ \mathcal{O}_Y$ for all $n\geq0$. Now, the Lemma follows by \cite[Lemma 1]{ishii1982moduli}. 

\end{proof}
\bigskip

\begin{lemma}\label{basic}Let  $Z$ be a smooth subvariety of a smooth variety $Y$ and let $\pi: Bl_Z Y\rightarrow Y$ be the blowup, with exceptional divisor $E=\pi^{-1}(Z)$.\\ 
\begin{enumerate}
\item Let $V$ be a smooth subvariety of $Y$, not contained in $Z$, and let $\widetilde{V}\subset Bl_Z Y$ be its strict transform. Then,\\
\begin{enumerate}
\item if $V$ meets $Z$ transversally (or is disjoint from $Z$), then $\widetilde{V}=\pi^{-1}(V)$ and $\mathcal{I}_{\pi^{-1}(V)}=\mathcal{I}_{\widetilde{V}}$. Moreover 
\begin{center}
$N_{\widetilde{V}/Bl_Z Y}\iso \pi^*N_{V/Y}$
\end{center}
\vspace{0.1in}
\item if $V\supset Z$, then $\mathcal{I}_{\pi^{-1}(V)}=\mathcal{I}_{\widetilde{V}}\cdot \mathcal{I}_E$. Moreover \\
\begin{center}
$N_{\widetilde{V}/Bl_Z Y}\iso \pi^*N_{V/Y}\otimes \mathcal{O}(E)$
\end{center}
Also, if Z has codimension 1 in V, the projection from $\widetilde{V}$ to $V$ is an isomorphism. 
\vspace{0.1in}

\end{enumerate}
\item Let $Z_1, Z_2$ be smooth subvarieties of $Y$ intersecting transversally. \\
\begin{enumerate}
\item Assume $Z_1\cap Z_2 \supseteq Z$. Then their strict transforms $\widetilde{Z_1}$ and $\widetilde{Z_2}$ intersect transversally and $\widetilde{Z_1}\cap \widetilde{Z_2}= \widetilde{Z_1\cap Z_2}$; in particular, if $Z_1\cap Z_2 = Z$, then $\widetilde{Z_1}\cap \widetilde{Z_2}= \emptyset$.\\
\item Assume $Z$ intersects transversally with $Z_1$ and $Z_2$, as well as with their intersection $Z_1\cap Z_2$. Then their strict transforms $\widetilde{Z_1}$ and $\widetilde{Z_2}$ intersect transversally and $\widetilde{Z_1}\cap \widetilde{Z_2}= \widetilde{Z_1\cap Z_2}$. \\
\item If $Z_1\supseteq Z$ and $Z_2$ intersects transversally with $Z$, then the intersections \\
\begin{gather*}
\widetilde{Z_1}\cap \widetilde{Z_2}\,\, \text{and}\,\, (E\cap \widetilde{Z_1})\cap \widetilde{Z_2}
\end{gather*}  
are transversal. Moreover, $\widetilde{Z_1}\cap \widetilde{Z_2}=\widetilde{Z_1\cap Z_2}$.\\
\end{enumerate}
\end{enumerate}
\end{lemma}
\begin{proof} (1) is standard; the proof of (2) follows from \cite[Lemma 2.9] {LiLi}.
\end{proof} 

\subsection{Proofs of Lemmas \ref{dominant} and \ref{normal}}In order to prove Lemma  \ref{dominant}, we consider the intersections of the iterated strict transforms of $\delta_I$ with the centers of each of the blowups in the sequence $\textbf{P}_I^{[3]}\rightarrow\PP^{d(n-1)-1}$. In addition, we distinguish two \textit{types of subvarieties} $\delta_J\cap \delta_I$ of $\delta_I$ under the identification (\ref{coincident}) of $\delta_I$ with 
\begin{center}
$\PP^{d(n-|I|)-1}=[x_{11}: x_{12}:  \dots :x_{1d}:\dots :x_{21}: x_{22}:\dots x_{2d}: \dots :x_{(n-|I|)1}: x_{(n-|I|)2}:\dots: x_{(n-|I|)d}]$
\end{center}
 as follows:
\begin{enumerate}[label=(\alph*)]
\item  Let $\delta_J \in \mathcal{H}^I_1$ (i.e. $J$ contains $I$). In this case $\delta_J\cap\delta_I =\delta_J \iso V(\{x_{ij}\}| i\in J\setminus I, j=1,\dots d)\iso \PP^{d(n-|J|)-1}$
\item  Let $\delta_J \in \mathcal{H}^I_2$ (i.e. $J$ is disjoint from $I$). In this case $\delta_J\cap \delta_I\iso V(\{x_{ik}-x_{jk}\} | i,j\in J, k=1,\dots d)\iso \PP^{d(n-|I|-|J|)-1}$
\end{enumerate}

Clearly, we have an equality of sets $$\mathcal{H}_{\mathcal{A}_+(I^c)}=\{\delta_J\cap \delta_I\subset \delta_I=\PP^{d(n-|I|)-1}\,|\delta_J \in \mathcal{H}^I_1\cup \mathcal{H}^I_2\}$$  
where $\mathcal{H}_{\mathcal{A}_+(I^c)}$ is defined in Definition \ref{defset}  for input data $(d,n-|I|+1, \mathcal{A}_+(I^c))$. We give $\mathcal{H}_{\mathcal{A}_+(I^c)}$ an order $(\lessdot)$ compatible with the order $(\prec)$ of Definition \ref{part}, that is:
\begin{itemize}
\item for any $\delta_J \in \mathcal{H}^I_1\cup \mathcal{H}^I_2$, $\delta_J\cap \delta_I \lessdot \delta_{J'}\cap \delta_I$ if and only if $\delta_J\prec \delta_{J'}$.
\end{itemize}

 \begin{lemma} \label{bset2}The ordered set $(\mathcal{H}_{\mathcal{A}_+(I^c)},\lessdot)$ satisfies the second condition of Theorem \ref{thmLi} (2).
 \end{lemma}
 \begin{proof} The proof is very similar to the proof of Lemma \ref{bset}, so we omit it.
 \end{proof}

For any $i\in\{1,2,3,4\}$ let $M_i$ be the cardinality of the set $\mathcal{H}^I_1\cup \dots \cup \mathcal{H}^I_i$. For any $k\in \ZZ$ such that $0<k\leq M_4$, we define $P_{I,k}$ to be $k$-th step in the sequence of blowups $\wtdn=\textbf{P}_I^{[4]}\rightarrow\PP^{d(n-1)-1}$ with respect to the order $(\prec)$. Moreover, for any $V\subset\PP^{d(n-1)-1}$, we denote by $V^{(k)}$ the iterated dominant transform of $V\subset\PP^{d(n-1)-1}$ in $P_{I,k}$. In particular $\delta_J^{(M_i)}=\delta_J^{[i]}$ for any $\delta_J\in \wbuildtdn$.
\begin{lemma} \label{helpful}
\begin{enumerate}[label=(\roman*)]
\item Let $\delta_J \in \mathcal{H}^I_1\cup \mathcal{H}^I_2$. For each $k$ such that $0\leq k\leq M_2$ the subvariety ${\delta_J}^{(k)}$ of $P_{I,k}$ either contains or intersects transversally with the center of the blowup $P_{I,k+1}\rightarrow P_{I,k}$.
\item $({\delta_I\cap\delta_J})^{(k)}= {\delta_I}^{(k)}\cap {\delta_J}^{(k)}$ for any $\delta_J\in \mathcal{H}_2^I$ and $0\leq k\leq M_1$. Moreover, each intersection ${\delta_I}^{(k)}\cap {\delta_J}^{(k)}$ is transversal.
\item ${\delta_I}^{(M_1)}\cap {\delta_J}^{(M_1)}=\emptyset$ for any $\delta_J\in \mathcal{H}_3^I$. Therefore ${\delta_I}^{(k)}\cap {\delta_J}^{(k)}=\emptyset$ for all $k>M_1$ as well.
\item $({\delta_I\cap\delta_J})^{(k)}= {\delta_I}^{(k)}\cap {\delta_J}^{(k)}$ for any $\delta_J\in \mathcal{H}_2^I$ and $M_1< k\leq M_2$. Moreover, each intersection ${\delta_I}^{(k)}\cap {\delta_J}^{(k)}$ is transversal.
\end{enumerate}
\end{lemma}
\begin{proof}(i) By Lemma \ref{anyextension}, $\mathcal{H}^I_1\cup \mathcal{H}^I_2$ is a building set. Also, observe that the order $(\prec)$ on $\mathcal{H}^I_1\cup \mathcal{H}^I_2$ is inclusion preserving. Therefore, the claim follows from Proposition \ref{buildpullback}(2).\\

\noindent (ii) Let $k$ such that $0\leq k\leq M_1$. By part (i), the iterated strict transform of each $\delta_J\in \mathcal{H}_2^I$ in $P_{I,k}$ must either intersect the blowup center inside $P_{I,k}$ transversally or it must contain that center. Also, observe that for $\delta_J\in \mathcal{H}_2^I$ we have that the intersection $\delta_I\cap \delta_J$ is transversal. Moreover, the iterated strict transform of $\delta_I$ in $P_{I,k}$ contains the corresponding center for all the above $k$. Therefore, by a repeated use of Lemma \ref{basic}(2)(a) and (c) we deduce (ii).\\

\noindent (iii) For any $\delta_J\in \mathcal{H}_3^I$, by definition, the set $J$ overlaps with $I$. Consider an element $j\in J\cap I$ and set $J':=j\cup (J\setminus I)$. Since ${\delta_{J'}}^{(M_1)}\supseteq{\delta_J}^{(M_1)}$, it is enough to show that ${\delta_I}^{(M_1)}\cap {\delta_{J'}}^{(M_1)}=\emptyset$. While $\delta_{J'}$ is not necessarily an element of $\wbuildtdn$, the same argument as in Lemma \ref{bset} shows that the set $\mathcal{H}_1^I\cup \{\delta_{J'}\}$, with $\mathcal{H}_1^I$ given an ascending dimension order and $\delta_J$ listed last, is a building set. Now the intersection $\delta_I\cap \delta_J'=\delta_{I\cup J'}$ belongs to $\mathcal{H}_1^I$; assume it is the $m$-th element of that set with respect to $(\prec)$ for some $m$ with $0<m\leq M_1$. Consider the iterated strict transforms $\delta_I^{(m-1)}$ and $\delta_{J'}^{(m-1)}$ of $\delta_I$ and $\delta_{J'}$ respectively under $P_{I,m-1}\rightarrow \PP^{d(n-1)-1}$. Since the intersection $\delta_I\cap \delta_{J'}$ is transversal, by using the argument of the previous paragraph we see that $\delta_I^{(m-1)}\cap\delta_{J'}^{(m-1)}=\delta_{I\cup J'}^{(m-1)}$. Now consider the $m$-th blowup $P_{I,m}\rightarrow P_{I,m-1}$. By Lemma \ref{basic}(2), we deduce that $\delta_I^{(m)}\cap\delta_{J'}^{(m)}=\emptyset$, so ${\delta_I}^{(M_1)}\cap {\delta_{J'}^{(M_1)}}=\emptyset$. \\

\noindent (iv) By Lemma \ref{anyextension}, $\mathcal{H}^I_1\cup \mathcal{H}^I_2$ is a building set and it is straightforward to see that the set $\mathcal{H}^I_1\cup \mathcal{H}^I_2\cup \{\delta_I\}$ is a building set as well. Since the latter is a subset of $\wbuildtdn$ we can  consider it ordered with order ($\prec$). Now, observe that the order $(\prec)$ on $\mathcal{H}^I_1\cup \mathcal{H}^I_2 \cup \{\delta_I\}$ is inclusion preserving. Therefore, by Proposition \ref{buildpullback}(2), $\delta_I^{(k)}$ intersects transversally with the center of the $k$-th blowup for all $k$ as in the statement and so does $\delta_J^{(k)}$ where $\delta_J\in \mathcal{H}_2^I$. The proof is then very similar to the proof of (ii). 
\end{proof}

\textit{Proof  Lemma \ref{dominant}}: By Lemma \ref{helpful}(3), we see that the iterated dominant transform $\delta_I^{[3]}\subset \textbf{P}_I^{[3]}$ is isomorphic to the iterated dominant transform $\delta_I^{[2]}\subset \textbf{P}_I^{[2]}$. By Lemma \ref{helpful}(2) and (4), the latter is, in turn, equal to the iterated blowup of $\delta_I$ at the ordered building set $(\mathcal{H}_{\mathcal{A}_+(I^c)},\lessdot)$. The Lemma now follows by Lemma \ref{bset2} and Theorem \ref{thmLi}(2).
\qed
\bigskip

\textit{Proof of Lemma \ref{normal}}: (1) First, note that the normal bundle of $\delta_I\iso\PP^{d(n-|I|)-1}$ in $\PP^{d(n-1)-1}$ is isomorphic to 
\begin{center}
$\bigoplus\limits_{i=1}^{d(|I|-1)} \mathcal{O}_{\PP^{d(n-|I|)-1}}(-1)$
\end{center}
	 Clearly $\mathcal{O}_{\PP^{d(n-1)-1}}(-1)$ pulls back to $\mathcal{O}_{\textbf{P}^{[1]}}(-1)$ on $\textbf{P}^{[1]}$, so $\mathcal{O}_{\PP^{d(n-|I|)-1}}(-1)$ pulls back to $\mathcal{O}_{{\delta_I}^{[1]}}(-1)$ on ${\delta_I}^{[1]}$. Therefore, by a repeated application of Lemma \ref{basic}(1)(b) we see that the normal bundle of the iterated strict transform of $\delta_I$ in $\textbf{P}^{[1]}$ is isomorphic to

\begin{align*}
\left(\bigoplus\limits_{i=1}^{d(|I|-1)} \mathcal{O}_{{\delta_I}^{[1]}}(-1)\right)\otimes \underbrace{\mathcal{O}_{{\delta_I}^{[1]}}(1)\otimes\dots\otimes\mathcal{O}_{{\delta_I}^{[1]}}(1)}_{p\, \text{times}}
=\bigoplus\limits_{i=1}^{d(|I|-1)} \mathcal{O}_{{\delta_I}^{[1]}}(p-1)
\end{align*}

Now, by Lemma \ref{helpful}(iii) and (iv), the intersection of each iterated strict transform of $\delta_I$ in the sequence of blowups $\textbf{P}^{[3]}\rightarrow\textbf{P}^{[1]}$ with every blowup center (corresponding to $\mathcal{H}_2^I \cup \mathcal{H}_3^I$) is transversal (even empty). Therefore, by applying Lemma \ref{basic}(1)(a) we complete the proof.\\
\indent The proof of part (2) is identical to the proof of (1) and is therefore omitted.
\qed

{\large{\bibliographystyle{amsalpha}}}
\bibliography{bib}
\end{document}